\documentclass[reqno, 11pt]{amsart}
\usepackage[left=1.03in, right=1.03in, top=1.05in, bottom=1.05in]{geometry}  

\usepackage{todonotes}
\presetkeys{todonotes}{inline}{}
\makeatletter
\providecommand\@dotsep{5}
\makeatother
\usepackage{amssymb, amsmath, amsthm, enumerate,  marginnote, mathabx, mathrsfs, mathtools, tikz, bm}
\usetikzlibrary{arrows,calc}
\usetikzlibrary{arrows.meta, cd, patterns}
\usepackage{bbm} 
\usepackage{esint} 
\usepackage{verbatim}

\usepackage[utf8]{inputenc}

\usepackage{hyperref}
\hypersetup{
    colorlinks=true,
    linkcolor=blue,
    filecolor=magenta,
}

\usepackage[normalem]{ulem}

\usetikzlibrary{decorations.pathreplacing}
\usetikzlibrary{decorations,calligraphy}

\usepackage{todonotes}

\usepackage{arydshln, multirow} 

\usepackage{comment}
\numberwithin{equation}{section}
\newtheorem{theorem}{Theorem}[section]
\newtheorem{corollary}[theorem]{Corollary}
\newtheorem{lemma}[theorem]{Lemma}
\newtheorem{proposition}[theorem]{Proposition}
\newtheorem{definition}[theorem]{Definition}
\newtheorem*{claim}{Claim}

\newtheorem{example}[theorem]{Example}

\theoremstyle{definition}

\newtheorem*{induction hypothesis}{Induction hypothesis}

\theoremstyle{remark}
\newtheorem*{remark}{Remark}

\def\inn#1#2{\langle#1,#2\rangle}

\def\bu{u}

\newcommand{\supp}{\mathrm{supp}\,}
\newcommand{\xisupp}{\mathrm{supp}_{\xi}\,}

\newcommand{\C}{\mathbb{C}}
\newcommand{\R}{\mathbb{R}}
\newcommand{\Z}{\mathbb{Z}}
\newcommand{\N}{\mathbb{N}}

\newcommand{\ud}{\mathrm{d}}

\newcommand{\xip}{\tfrac{\xi}{|\xi|}}

\newcommand{\be}{\mathbf{e}}
\newcommand{\ba}{\mathbf{a}}
\newcommand{\bg}{\mathbf{g}}

\newcommand{\floor}[1]{\lfloor #1 \rfloor }

\def\bbone{{\mathbbm 1}}

\newcommand{\Be}{\begin{equation}}
\newcommand{\Ee}{\end{equation}}

\newcommand{\Bm}{\begin{multline}}
\newcommand{\Em}{\end{multline}}
\newcommand{\Bea}{\begin{eqnarray}}
\newcommand{\Eea}{\end{eqnarray}}
\newcommand{\Beas}{\begin{eqnarray*}}
\newcommand{\Eeas}{\end{eqnarray*}}
\newcommand{\Benu}{\begin{enumerate}}
\newcommand{\Eenu}{\end{enumerate}}
\newcommand{\Bi}{\begin{itemize}}
\newcommand{\Ei}{\end{itemize}}

\def\intslash{\rlap{\kern  .32em $\mspace {.5mu}\backslash$ }\int}
\def\qsl{{\rlap{\kern  .32em $\mspace {.5mu}\backslash$ }\int_{Q_x}}}

\def\Im{\operatorname{Im\,}}

\def\N{\mathbb N}

\def\floor#1{{\lfloor #1 \rfloor }}
\def\ceil#1{{\lceil#1 \rceil }}
\def\emph#1{{\it #1 }}

\def\inn#1#2{\langle#1,#2\rangle}

\def\fI{{\mathfrak {I}}}

\def\fd{{\mathfrak {d}}}

\def\bbN{{\mathbb {N}}}

\newcommand{\br}{\mathbf{r}}

\title{Sharp $L^p$ bounds for the helical maximal function}

\author[D. Beltran]{David Beltran}\author[S. Guo]{Shaoming Guo}\author[J. Hickman]{Jonathan Hickman}\author[A. Seeger]{Andreas Seeger}
\date{\today}

\address{David Beltran: Department of Mathematics, University of Wisconsin, 480 Lincoln Drive, Madison, WI, 53706, USA.}
\email{dbeltran@math.wisc.edu}
\address{Shaoming Guo: Department of Mathematics, University of Wisconsin, 480 Lincoln Drive, Madison, WI, 53706, USA.}
\email{shaomingguo@math.wisc.edu}
\address{Jonathan Hickman: School of Mathematics, James Clerk Maxwell Building, The King's Buildings, Peter Guthrie Tait Road, Edinburgh, EH9 3FD, UK.}
\email{jonathan.hickman@ed.ac.uk}
\address{Andreas Seeger: Department of Mathematics, University of Wisconsin, 480 Lincoln Drive, Madison, WI, 53706, USA.}
\email{seeger@math.wisc.edu}

\begin{document}

\maketitle

\begin{abstract} We establish the $L^p(\R^3)$ boundedness of the helical maximal function for the sharp range $p>3$. Our results improve the previous known bounds for $p>4$. The key ingredient is a new microlocal smoothing estimate for averages along dilates of the helix, which is established via a square function analysis.
\end{abstract}




\section{Introduction}




\subsection{Main results} For $n\ge 2$ let $\gamma \colon I \to \R^n$ be a smooth curve, where $I \subset \R$ is a compact interval, and $\chi \in C^{\infty}(\R)$ be a bump function supported on the interior of $I$. Given $t>0$, consider the averaging operator
\begin{equation*}
    A_tf(x) := \int_{\R} f(x - t\gamma(s))\,\chi(s)\,\ud s
\end{equation*}
and define the associated maximal function
\begin{equation*}
    M_\gamma f(x):= \sup_{t>0} |A_t f(x)|.
\end{equation*}

We are interested in the $L^p$ mapping properties of $M_{\gamma}$. It is well-known that the range of exponents 
$p$ for which $M_{\gamma}$ is bounded on $L^p$ depends on the curvature of the underlying curve. Accordingly, we consider smooth curves $\gamma \colon I \to \R^n$ which are \textit{non-degenerate}, in the sense that there is a constant $c_0 > 0$ such that 
\begin{equation}\label{eq:nondegenerate}
    |\det(\gamma'(s), \cdots, \gamma^{(n)}(s))| \geq c_0 \qquad \textrm{for all $s \in I$.}
\end{equation}

A celebrated theorem of Bourgain \cite{Bourgain1985, Bourgain1986} states that if $\gamma \colon I \to \R^2$ is a smooth, non-degenerate plane curve, then $M_{\gamma}$ is bounded on $L^p(\R^2)$ if and only if $p > 2$. Here we establish a 3-dimensional variant of this result.

\begin{theorem}\label{intro max thm} If $\gamma: I \to \R^3$ is a smooth,  non-degenerate space curve, then $M_{\gamma}$ is bounded on $L^p(\R^3)$ if and only if $p > 3$.
\end{theorem}

In the $n=3$ case, the condition \eqref{eq:nondegenerate} is equivalent to the non-vanishing of the curvature and torsion functions. As a concrete example, Theorem~\ref{intro max thm} implies that the \textit{helical maximal operator}
\begin{equation*}
    M_{\mathrm{Helix}}f(x) := \sup_{t > 0} \Big|\int_0^{2 \pi} f(x_1 - t\cos \theta, x_2 - t\sin \theta, x_3 - t\theta)\,\ud \theta \Big|
\end{equation*}
is bounded on $L^p(\R^3)$ for all $p > 3$.\medskip

A simple Knapp-type example shows $L^p$ boundedness fails for $p \leq 3$ (see \S\ref{nec cond sec}). On the other hand, Pramanik and the fourth author~\cite{PS2007} proved that Wolff's decoupling inequality \cite{Wolff2000} for the light cone implies the boundedness of $M_\gamma$ for a suitable range of $p$. The optimal range for Wolff's inequality was obtained by Bourgain and Demeter \cite{BD2015}  and the  combination of the results in \cite{PS2007} and \cite{BD2015} yields the $L^p$ boundedness of 
$M_{\gamma}$ for the partial range $4 < p \leq \infty$. Thus, Theorem~\ref{intro max thm} closes the gap by establishing boundedness for the remaining exponents $3 < p \leq 4$.\medskip

To prove Theorem~\ref{intro max thm}, we follow the basic strategy introduced by Mockenhaupt, Sogge and the fourth author \cite{MSS1992} in the context of the classical circular maximal function in the plane. In particular, in \cite{MSS1992} the authors gave an alternative proof of Bourgain's maximal theorem, deriving it as a consequence of certain \textit{local smoothing} estimates for the wave propagator. In the case of maximal functions associated to space curves, Theorem~\ref{intro max thm} follows from a local smoothing estimate for a class of Fourier integral operators associated to the averages $A_t$. To give a simple statement of the key underlying inequality, set $\mathfrak{A}_{\gamma}f(x,t) := \rho(t) \cdot A_tf(x)$ for some $\rho \in C^{\infty}_c(\R)$ with $\supp \rho \subseteq [1,2]$. Our main theorem then reads as follows.

\begin{theorem}\label{intro LS thm} Suppose $\gamma: I \to \R^3$ is a smooth,  non-degenerate space curve and let $3 \leq p \leq 4$ and $\sigma < \sigma(p)$ where $\sigma(p) := \tfrac{1}{5}\big(1 + \tfrac{2}{p}\big)$. Then $\mathfrak{A}_{\gamma}$ maps $L^p(\R^3)$ boundedly into $L^p_{\sigma}(\R^4)$. 
\end{theorem}

Note that $\sigma(p)>1/p$ for $p > 3$. Thus, by a well-known Sobolev embedding argument, Theorem~\ref{intro LS thm} implies Theorem~\ref{intro max thm}. For completeness, the details of this implication are presented in~\S\ref{LS vs max sec}.




 \subsection{Comparison with previous results} It follows from work of Pramanik and the fourth author \cite{PS2007} (combined with sharp decoupling estimates from \cite{BD2015}) that, for each fixed $t$, the single average $A_t$ maps $L^p(\R^3)$ boundedly into $L^p_{\alpha}(\R^3)$ for all $2 \leq p \leq \infty$ and $\alpha < \alpha(p)$, where\footnote{In \cite{PS2007} the $\alpha=\alpha(p)$ endpoint estimate is also shown to hold for $p > 4$.}
\begin{equation*}
    \alpha(p):=\begin{cases}
    \frac{1}{3}(\frac{1}{2}+\frac{1}{p}) \qquad &\text{if $2 \leq p \leq 4$}\\
    \frac{1}{p} \qquad & \text{if $p \geq 4$}
    \end{cases}.
\end{equation*}
 Theorem~\ref{intro LS thm} represents a gain of $\sigma(p)-\alpha(p)-\varepsilon=\frac{1}{15}(\frac{1}{2} + \frac{1}{p})-\varepsilon$ derivatives when integrating locally in time in the range $3 \leq p \leq 4$. In this sense, Theorem~\ref{intro LS thm} is an example of \textit{local smoothing} (see, for instance, \cite{Sogge1991, MSS1992, GWZ2020, BHS2} for a discussion of the classical local smoothing phenomenon for the wave equation).\medskip
 
  Theorem~\ref{intro LS thm} complements previous local smoothing estimates from \cite{PS2007}, which deal with the supercritical\footnote{Here we are referring to criticality for the \textit{single average} operator, so that $p=4$ correspond to the critical point where the behaviour of the $\alpha(p)$ exponent changes.} regime $p > 4$. In \cite[Theorem 1.4]{PS2007} it is shown that $\mathfrak{A}_\gamma$ maps $L^p(\R^3)$ boundedly into $L^p_{\delta}(\R^4)$ for all $2 \leq p \leq \infty$ and $\delta < \delta(p)$, where
\begin{equation*}
    \delta(p):=\begin{cases}
    \frac{1}{3}(\frac{1}{2}+\frac{1}{p}) \qquad &\text{if $2 \leq p \leq 6$}\\
    \frac{4}{3p} \qquad & \text{if $p \geq 6$}
    \end{cases}.
\end{equation*}
Note that this does \textbf{not} yield any local smoothing estimates in the subcritical regime $2 \leq p \leq 4$, where $\alpha(p)$ and $\delta(p)$ agree. Consequently, the local smoothing estimates in \cite{PS2007} only imply $L^p(\R^3)$-boundedness of $M_\gamma$ for the restricted range $p>4$.\medskip

It is remarked that the (somewhat loosely) related problem of $L^p(\R^n) \to L^p(\R^{n+1})$ bounds for $\mathfrak{A}_{\gamma}$ (as opposed to Sobolev bounds) was investigated in \cite{Hickman2016}. This question is significantly easier than establishing local smoothing estimates and, accordingly, in \cite{Hickman2016} an almost complete characterisation of the $L^p(\R^n) \to L^q(\R^{n+1})$ mapping properties is obtained in all dimensions.




\subsection{Overview of the argument}\label{overview subsec} For $\gamma \colon I \to \R^n$ a smooth curve let $\mu$ denote the pushforward of the measure $\chi(s)\ud s$ under $\gamma$. Defining the dilates $\inn{\mu_t}{f} = \inn{\mu}{f(t\,\cdot\,)}$, it follows that the underlying averaging operators satisfy $A_tf = f \ast \mu_t$. Thus, in the frequency domain $A_t$ corresponds to multiplication against the Fourier transform
\begin{equation*}
    \widehat{\mu}_t(\xi) = \int_{\R} e^{-i t \inn{\gamma(s)}{\xi}} \chi(s)\, \ud s. 
\end{equation*}
Since the main estimate in Theorem~\ref{intro LS thm} is an $L^p$-Sobolev bound, we are led to studying the decay properties of the above oscillatory integral for large $\xi$.\medskip

Suppose $\gamma \colon I \to \R^3$ satisfies the non-degeneracy hypothesis \eqref{eq:nondegenerate}. This implies $\sum_{j=1}^3|\inn{\gamma^{(j)}(s)}{\xi}| \gtrsim |\xi|$ for all $s \in I$ and all $\xi \in \widehat{\R}^3$
and, consequently, a simple van der Corput estimate yields
\begin{equation*}
    |\widehat{\mu}_t(\xi)| \lesssim_{\gamma} (1+t|\xi|)^{-1/3}.
\end{equation*}
However, this slow decay rate only occurs on a small portion of the frequency domain, corresponding to a (neighbourhood of a) codimension 1 cone $\Gamma \subseteq \widehat{\R}^3$ generated by the binormal vector $\be_3(s)$ to the curve $\gamma$. In light of this, it is natural to dyadically decompose the frequency domain into conic regions according to the distance to $\Gamma$.\medskip

The pieces of the decomposition which are supported far away from $\Gamma$ satisfy improved decay estimates. In one extreme case, the non-degeneracy condition improves to $\sum_{j=1}^2|\inn{\gamma^{(j)}(s)}{\xi}| \gtrsim |\xi|$ and the van der Corput estimate therefore becomes \begin{equation*}
    |\widehat{\mu}_t(\xi)| \lesssim_{\gamma} (1+t|\xi|)^{-1/2}.
\end{equation*}
In this situation, the operator behaves in many ways like the circular average in the plane, and can be estimated using a lifted version of the argument developed to study the 2 dimensional problem in \cite{MSS1992} and \cite{Wolff2000}. In particular, to prove the desired local smoothing estimate in this extreme case, we observe that the Fourier transform of $\mathfrak{A}_{\gamma}$ in all $4$ variables $(x,t)$ is essentially supported in a neighbourhood of a codimension 1 cone $\widetilde{\Gamma}_1 \subseteq \widehat{\R}^4$. This surface is analogous to the light cone in $\widehat{\R}^3$ which is central to the analysis of local smoothing for the circular averages in \cite{MSS1992, Wolff2000} and, more recently, \cite{GWZ2020}. Following an argument of Wolff~\cite{Wolff2000}, the operator is further decomposed according to plate regions on $\widetilde{\Gamma}_1$ using a decoupling estimate. The individual pieces of this decomposition are then finally amenable to direct estimation.\medskip

The method described in the previous paragraph only directly applies very far from the binormal cone (and therefore far from the most singular parts of the operator). However, by using decoupling inequalities and rescaling, it can also be used to study pieces of the decomposition which lie closer to $\Gamma$. The key observation is that the pieces of the decomposition which lie close to $\Gamma$ can be decoupled into smaller pieces which, when rescaled, resemble the part of the decomposition far from $\Gamma$. This, roughly speaking, is the approach used in \cite{PS2007} to obtain Theorem~\ref{intro max thm} in the restricted range $4 < p \leq \infty$.\medskip

In order to prove the full range of $L^p$-boundedness of Theorem~\ref{intro max thm} a more direct method is required to analyse the pieces of the decomposition which lie close to the binormal cone. For this part of the operator, the microlocal geometry no longer resembles that of the 2-dimensional problem and, consequently, the decoupling and rescaling argument used in \cite{PS2007} is inefficient.\medskip

Close to the binormal cone, we observe that the Fourier transform of $\mathfrak{A}_{\gamma}$ in all $4$ variables $(x,t)$ is essentially supported in a neighbourhood of a codimension 2 cone $\widetilde{\Gamma}_2 \subseteq \widehat{\R}^4$. This cone is a lower-dimensional submanifold of the cone $\widetilde{\Gamma}_1$ we encountered earlier. Similarly to the previous case, the operator is further decomposed according to plate regions, now along $\widetilde{\Gamma}_2$. However, in order to efficiently carry out this decomposition, here we use a square function rather than a decoupling inequality, in the spirit of \cite{MSS1992}. The required square function estimate is deduced using a $4$-linear restriction estimate from \cite{BBFL2018}. After applying the square function, a series of weighted $L^2$ inequalities can be brought to bear on the problem to obtain, together with various corresponding Nikodym-type maximal bounds, a favourable estimate for this part of the operator. This final step of the argument is itself somewhat involved and a discussion of the details is beyond the scope of this introduction.\medskip

The above discussion focuses on two extreme cases of the problem: 
\begin{enumerate}[i)]
\item Far from the binormal cone $\Gamma$, where $\mathfrak{A}_{\gamma}$ is $(x,t)$-Fourier localised to a codimension 1 cone $\widetilde{\Gamma}_1$.
\item Close to the binormal cone $\Gamma$, where $\mathfrak{A}_{\gamma}$ is $(x,t)$-Fourier localised to a codimension 2 cone $\widetilde{\Gamma}_2$.
\end{enumerate}

For pieces of the decomposition which lie in the intermediate range, both cones $\widetilde{\Gamma}_1$ and $\widetilde{\Gamma}_2$ play a r\^ole in the analysis. This complicates matters somewhat, since it is necessary to carry out frequency decompositions simultaneously with respect to both geometries.

\subsection*{Outline of the paper} This paper is structured as follows: 
\begin{itemize}
    \item In \S\ref{LS vs max sec} we show how Theorem~\ref{intro LS thm} implies Theorem~\ref{intro max thm}.
    \item In \S\ref{sec:bandlimited} we reduce Theorem~\ref{intro LS thm} to its version for band-limited functions, which is Theorem~\ref{LS thm}.
    \item In \S\ref{curve sym sec} we introduce a class of model curves.
    \item In \S\ref{key ingredients sec} we state 3 key auxiliary results that feature in the proof of Theorem~\ref{LS thm}: a reverse square function estimate in $\R^{3+1}$, a forward square function estimate in $\R^3$ and a Nikodym maximal operator bound.
    \item In \S\S\ref{sec:slow decay cone}--\ref{J=3 sec} we present the proof of Theorem~\ref{LS thm}.
    \item In \S\ref{reverse SF sec} we present the proof of the reverse square function estimate in $\R^{3+1}$ (Theorem~\ref{Frenet reverse SF theorem}).
    \item In \S\ref{forward SF sec} we present the proof of the forward square function estimate in $\R^3$ (Proposition~\ref{f SF prop}).
    \item In \S\ref{Nikodym sec} we present the proof of the Nikodym maximal operator bound (Proposition~\ref{Nikodym prop}).
    \item In \S\ref{nec cond sec} we show the condition $p > 3$ is necessary for the boundedness of the global maximal function. 
    \item Appendix~\ref{BG appendix} contains an abstract broad/narrow decomposition lemma which features in the proof of Theorem~\ref{Frenet reverse SF theorem}.
    \item There are two further appendices which deal with various auxiliary results and technical lemmas used in the main argument.
\end{itemize}




\subsection*{Notational conventions} Given a (possibly empty) list of objects $L$, for real numbers $A_p, B_p \geq 0$ depending on some Lebesgue exponent $p$ or dimension parameter $n$ the notation $A_p \lesssim_L B_p$, $A_p = O_L(B_p)$ or $B_p \gtrsim_L A_p$ signifies that $A_p \leq CB_p$ for some constant $C = C_{L,p,n} \geq 0$ depending on the objects in the list, $p$ and $n$. In addition, $A_p \sim_L B_p$ is used to signify that both $A_p \lesssim_L B_p$ and $A_p \gtrsim_L B_p$ hold. Given $a$, $b \in \R$ we write $a \wedge b:= \min \{a, b\}$ and  $a \vee b:=\max \{a,b\}$. 
The length of a multiindex $\alpha\in \bbN_0^n$ is given by $|\alpha|=\sum_{i=1}^n{\alpha_i}$.

\subsection*{Acknowledgements}
{The authors thank the American Institute of Mathematics for funding their collaboration through the SQuaRE program, also  supported in part  by the National Science Foundation. D.B. was partially supported by NSF grant DMS-1954479. S.G. was partially supported by  NSF grant DMS-1800274. A.S. was partially supported by  NSF grant DMS-1764295 and by a Simons fellowship.  This material is partly based upon work supported by the National Science Foundation under Grant No. DMS-1440140 while the authors were in residence at the Mathematical Sciences Research Institute in Berkeley, California, during the Spring 2017 semester. D.B. and J.H. would also like to thank the LMS for funding a research visit through the LMS `Research in Pairs' Scheme 4 grant (Grant Ref 41802).}




\section{Local smoothing vs maximal bounds}\label{LS vs max sec}

For the readers' convenience, here we state and prove a general result relating local smoothing estimates for the operator $\mathfrak{A}_{\gamma}f(x,t) := \rho(t) \, A_tf(x)$ to $L^p$ estimates for the corresponding maximal function $M_{\gamma}$. 

\begin{proposition}\label{LS vs max prop} Let $\gamma \colon I \to \R^n$ be a smooth curve and suppose $\mathfrak{A}_{\gamma}$ maps $L^{p}(\R^n)$ boundedly into $L^{p}_{\sigma}(\R^{n+1})$ for some $2 \leq p < \infty$ and $\sigma > 1/p$. Then $M_{\gamma}$ is bounded on $L^p(\R^n)$.
\end{proposition}

Observe that the exponent $\sigma(p) := \tfrac{1}{5}(1 + \tfrac{2}{p})$ satisfies $\sigma(p) > 1/p$ for all $p > 3$. Consequently, Theorem~\ref{intro LS thm} combines with Proposition~\ref{LS vs max prop} to yield Theorem~\ref{intro max thm} in the restricted range $3 < p \leq 4$. The remaining estimates follow from interpolation with the trivial $L^{\infty}$ bound.\medskip

Before presenting the proof we introduce a system of Littlewood--Paley functions which will feature throughout the article. Fix $\eta \in C^\infty_c(\R)$ non-negative and such that
\begin{equation}\label{eta def}
\eta(r) = 1 \quad \textrm{if $r \in [-1,1]$} \quad \textrm{and} \quad \supp \eta \subseteq [-2,2]   
\end{equation}
and define $\beta^k$, $\tilde{\beta}^k \in C^\infty_c(\R)$ by    
\begin{equation}\label{beta def}
    \beta^k(r):=\eta(2^{-k}r) - \eta(2^{-k+1}r) \qquad \textrm{and} \qquad \tilde{\beta}^k(r):=\eta(2^{-k-1}r) - \eta(2^{-k+2}r)
\end{equation} 
for each $k \in \Z$. By a slight abuse of notation we also let $\eta$, $\beta^k$, $\tilde{\beta}^k \in C^{\infty}_c(\widehat{\R}^n)$ denote the radial functions obtained by evaluating the corresponding univariate functions at $|\xi|$. Finally, if $k=0$, then we drop the superscript and simply write $\beta := \beta^0$ and $\tilde{\beta} := \tilde{\beta}^0$. Note that the $\beta^k$ form a partition of unity of $\widehat{\R}^n$ subordinated to a family of dyadic annuli, and they satisfy the reproducing formula $\beta^k = \tilde{\beta}^k\cdot \beta^k$. 

\begin{proof}[Proof of Proposition~\ref{LS vs max prop}] Decompose the $t$ parameter into dyadic intervals
\begin{equation*}
 M_{\gamma} f(x) = \sup_{\ell \in \Z} \sup_{1 \leq t \leq 2} |A_{2^{\ell}t}f(x)|.
\end{equation*}
Performing a Littlewood--Paley decomposition on each of the averaging operators,
\begin{equation*}
    M_{\gamma} f(x) \leq \sum_{k = 1}^{\infty}\Big(\sum_{\ell \in \Z} \sup_{1 \leq t \leq 2} |A_{2^{\ell}t}\beta_{k-\ell}(D)f(x)|^p \Big)^{1/p} + CM_{\mathrm{HL}}f(x)
\end{equation*}
where $M_{\mathrm{HL}}$ is the Hardy--Littlewood maximal function. Indeed, it is not difficult to verify that the pointwise estimate
\begin{equation*}
    \sup_{\ell \in \Z} \sup_{1 \leq t \leq 2} |A_{2^{\ell}t}\eta_{-\ell}(D)f(x)| \leq C M_{\mathrm{HL}}f(x);
\end{equation*}
for $1 \leq t \leq 2$ the function $A_{2^{\ell}t} \eta_{-\ell}(D)f(x)$ roughly corresponds to an average of $f$ over a ball of radius $2^{\ell}$ centred at $x$. Thus, by the Hardy--Littlewood maximal theorem and the triangle inequality it suffices to show that 
\begin{equation}\label{LP vs max 1}
    \sum_{k = 1}^{\infty}\Big(\sum_{\ell \in \Z} \big\|\sup_{1 \leq t \leq 2} |A_{2^{\ell}t}\beta_{k-\ell}(D)f| \big\|_{L^p(\R^n)}^p \Big)^{1/p} \lesssim_{\gamma, p} \|f\|_{L^p(\R^n)}.
\end{equation}

By a simple scaling argument, one obtains the operator norm identity
\begin{equation*}
    \|\sup_{1 \leq t \leq 2} |A_{2^{\ell} t} \beta_{k-\ell}(D)|\|_{L^p(\R^n) \to L^p(\R^n)} = \|\sup_{1 \leq t \leq 2} |A_{t} \beta_{k}(D)|\|_{L^p(\R^n) \to L^p(\R^n)}.
\end{equation*}
Combining this with the hypothesised local smoothing estimate, it follows that
\begin{align}\label{LP vs max 2}
    \Big(\int_1^2\|A_{2^{\ell} t} \beta_{k-\ell}(D)f\|_{L^p(\R^n)}^p \,\ud t\Big)^{1/p} &\lesssim_{\gamma, p, \sigma} 2^{-\sigma k} \|\tilde{\beta}_{k-\ell}(D)f\|_{L^p(\R^n)}, \\
    \label{LP vs max 3}
    \Big(\int_1^2\|\frac{\partial}{\partial t} A_{2^{\ell} t} \beta_{k-\ell}(D)f\|_{L^p(\R^n)}^p \,\ud t\Big)^{1/p} &\lesssim_{\gamma, p, \sigma} 2^{-\sigma k + k} \|\tilde{\beta}_{k-\ell}(D)f\|_{L^p(\R^n)}.
\end{align}
The second estimate follows by noting that the Fourier multiplier associated to $\partial_t A_{2^\ell t} \beta_{k-\ell}(D)$ is essentially the same as the multiplier associated to $A_{2^\ell t} \beta_{k-\ell}(D)$ but with an extra $|\xi|$ factor. We therefore pick up an additional $2^k$ owing to the estimate $    \||D|\tilde{\beta}_k(D)f\|_{L^p(\R^n)} \lesssim 2^k \|\tilde{\beta}_k(D)f\|_{L^p(\R^n)}$.

Combining \eqref{LP vs max 2} and \eqref{LP vs max 3} with the elementary Sobolev embedding
\begin{equation*}
    \sup_{1 \leq t \leq 2}|F(t)|^p \leq \int_1^2|F(s)|^p\,\ud s + p \Big(\int_1^2|F'(s)|^p\,\ud s\Big)^{1/p} \Big(\int_1^2|F(s)|^p\,\ud s\Big)^{1/p'},
\end{equation*}
it follows that 
\begin{equation}\label{LP vs max 4}
   \|\sup_{1 \leq t \leq 2}
  |A_{2^{\ell} t} \beta_{k-\ell}(D)f| \|_{L^p(\R^n)} \lesssim_{\gamma, p, \sigma} 2^{-k(\sigma - 1/p)} \|\tilde{\beta}_{k-\ell}(D)f\|_{L^p(\R^n)}.
\end{equation}
Taking the $\ell^p$-norm of both sides of \eqref{LP vs max 4}, we may sum the resulting expression in $\ell$ using the elementary inequality
\begin{equation*}
    \Big(\sum_{\ell \in \Z}\|\tilde{\beta}_{\ell}(D)f\|_{L^p(\R^n)}^p \Big)^{1/p} \lesssim \|f\|_{L^p(\R^n)},
\end{equation*}
valid for $p \geq 2$. On the other hand, under the crucial hypothesis $\sigma > 1/p$, we have a geometric decay which allows us to sum in $k$. Thus, we deduce the desired estimate \eqref{LP vs max 1}.  
\end{proof}




\section{Reduction to band-limited estimates}\label{sec:bandlimited} We now turn to the proof of Theorem~\ref{intro LS thm}, which occupies almost the entirety of the article. Since we are interested in $L^p(\R^3) \to L^p_{\sigma}(\R^{3+1})$ estimates for $\sigma$ belonging to an \textit{open} range, the problem is immediately reduced to studying $L^p(\R^3) \to L^p(\R^{3+1})$ bounds for band-limited pieces of the operator. In order to describe this reduction in more detail, it is useful to set up some notational conventions.\medskip

Given $m \in L^{\infty}(\widehat{\R}^n \times \R)$, for each $t \in \R$ let $m(D;t)$ denote the associated multiplier operator
\begin{equation*}
    m(D;t)f (x) :=  \frac{1}{(2 \pi)^n} \int_{\widehat{\R}^n} e^{i \inn{x}{\xi}} m(\xi;t) \widehat{f}(\xi)\,\ud \xi,
\end{equation*}
defined initially for functions $f$ belonging to a suitable \textit{a priori} class. With this notation, the averaging operator $A_t$ is given by $A_t = \widehat{\mu}_t(D)$ where $\mu_t$ is the measure introduced in \S\ref{overview subsec}.\medskip 

The multipliers of interest are of the following form. Let $\gamma \colon I \to \R^n$ be a smooth curve and fix $\chi$, $\rho \in C^\infty_c(\R)$ supported in the interior of $I$ and $[1/2,4]$, respectively. Given a symbol $a \in C^{\infty}(\widehat{\R}^n\setminus\{0\} \times \R \times \R )$, define
\begin{equation}\label{multiplier definition}
    m[a](\xi;t) := \int_{\R} e^{-i t \inn{\gamma(s)}{\xi}} a(\xi;t; s)\chi(s) \rho(t)\,\ud s.
\end{equation}
Taking $a$ in this definition to be identically 1, we recover the ($t$-localised) multiplier $\rho(t) \widehat{\mu}_t(\xi)$. In general, we perform surgery on $\widehat{\mu}_t$ by choosing $a$ so that $m[a]$ is localised to a particular region of the frequency space. \medskip

For $a \in C^{\infty}(\widehat{\R}^n \setminus \{0\} \times \R \times \R)$ as above, we form a dyadic decomposition by writing
\begin{equation}\label{symbol dec}
    a = \sum_{k = 0}^{\infty} a_k \qquad \textrm{where} \qquad  a_k(\xi; t; s) :=   \left\{ \begin{array}{ll}
        a(\xi; t; s) \, \beta^k(\xi) & \textrm{for $k \geq 1$} \\
         a(\xi; t; s) \, \eta(\xi) & \textrm{for $k =0$}
     \end{array} \right. .
\end{equation}
Here $\eta$ and $\beta^k$ are the functions introduced in \eqref{eta def} and \eqref{beta def}.\medskip

With the above definitions, our main result is as follows. 

\begin{theorem}\label{LS thm} Let $\gamma:I \to \R^3$ be a smooth curve and suppose $a \in C^{\infty}(\widehat{\R}^3\setminus \{0\} \times \R \times \R)$ satisfies the symbol condition
\begin{equation*}
    |\partial_{\xi}^{\alpha}\partial_t^i \partial_s^j a(\xi;t;s)| \lesssim_{\alpha, i, j} |\xi|^{-|\alpha|} \qquad \textrm{for all $\alpha \in \N_0^3$ and $i$, $j \in \N_0$}
\end{equation*}
and that
\begin{equation}\label{J=3 condition}
    \sum_{j=1}^3|\inn{\gamma^{(j)}(s)}{\xi}| \gtrsim |\xi| \qquad \text{ for all $(\xi;s) \in \xisupp a \times I$}.
\end{equation}
Let $3 \leq p \leq 4$, $\varepsilon>0$ and $k\ge 1$. If $a_k$ is defined as in \eqref{symbol dec}, then
\begin{equation*}
    \Big(\int_1^2\|m[a_k](D;t)f\|_{L^p(\R^3)}^p\,\ud t\Big)^{1/p} \lesssim_{\varepsilon, p} 2^{-\frac{k}{5}(1+\frac{2}{p}) + k \varepsilon}\|f\|_{L^p(\R^3)}.
\end{equation*}
\end{theorem}

For $n=3$, the condition \eqref{J=3 condition} is equivalent to the non-degeneracy hypothesis \eqref{eq:nondegenerate}. Thus, Theorem~\ref{LS thm} immediately implies Theorem~\ref{intro LS thm} via the Littlewood--Paley characterisation of Sobolev spaces.\medskip

Under a stronger hypothesis on the phase function, a stronger local smoothing estimate holds, by a combination of the  work of Pramanik and the fourth author \cite{PS2007} with the full   decoupling theorem for the light cone by  Bourgain and Demeter \cite{BD2015}.
\footnote{The estimates in \cite{PS2007} are stated for $p >6$. The version of the result presented here for $2 \leq p \leq \infty$ follows via interpolation with trivial $L^2$-estimates.}

\begin{theorem}[{\it cf.} Theorem 4.1 in \cite{PS2007}]\label{PS LS J=2}
  Let $\gamma:I \to \R^3$ be a smooth curve and suppose that $a \in C^{\infty}(\widehat{\R}^3\setminus \{0\} \times \R \times \R)$ satisfies the symbol conditions
\begin{equation*}
    |\partial_{\xi}^{\alpha}\partial_t^i \partial_s^j a(\xi;t;s)| \lesssim_{\alpha, i, j} |\xi|^{-|\alpha|} \qquad \textrm{for all $\alpha \in \N_0^3$ and $i$, $j \in \N_0$}
\end{equation*}
and that
\begin{equation}\label{J=2 condition}
    |\inn{\gamma'(s)}{\xi}| + |\inn{\gamma''(s)}{\xi}| \gtrsim |\xi| \qquad \text{ for all $\,\, (\xi;s) \in \xisupp a \times I$}.
\end{equation}

Let $2 \leq p \leq 6$, $\varepsilon>0$ and $k \geq 1$. If $a_k$ is defined as in \eqref{symbol dec}, then
\begin{equation*}
    \Big(\int_1^2\|m[a_k](D;t)f\|_{L^p(\R^3)}^p\,\ud t\Big)^{1/p} \lesssim_{\varepsilon, p} 2^{-\frac{k}{2}(\frac{1}{2} + \frac{1}{p}) + k \varepsilon}\|f\|_{L^p(\R^3)}.
\end{equation*}
\end{theorem}

Owing to the strengthened hypothesis \eqref{J=2 condition}, Theorem~\ref{PS LS J=2} alone is insufficient for our purposes. Indeed, Theorem~\ref{PS LS J=2} only effectively deals with parts of the multiplier which are supported away from the main singularity. However, we still make use of Theorem~\ref{PS LS J=2} in the proof of Theorem~\ref{LS thm} to analyse the multiplier in this less singular region, in which it is effective.




\section{Symmetries and model curves}\label{curve sym sec}

A prototypical example of a smooth curve satisfying the non-degeneracy condition \eqref{eq:nondegenerate} is the \textit{moment curve} $\gamma_{\circ} \colon \R \to \R^n$, given by
\begin{equation*}
    \gamma_{\circ}(s) := \Big(s, \frac{s^2}{2}, \dots, \frac{s^n}{n!} \Big). 
\end{equation*}
Indeed, in this case the determinant appearing in \eqref{eq:nondegenerate} is everywhere equal to 1. Moreover, at small scales, any non-degenerate curve can be thought of as a perturbation of an affine image of $\gamma_{\circ}$. To see why this is so, fix a non-degenerate curve $\gamma \colon I \to  \R^n$ and $\sigma \in I$, $\lambda > 0$ such that $[\sigma - \lambda, \sigma+\lambda] \subseteq I$. Denote by $[\gamma]_{\sigma}$ the $n\times n$ matrix
\begin{equation*} 
    [\gamma]_{\sigma}:=
    \begin{bmatrix}
    \gamma^{(1)}(\sigma) & \cdots & \gamma^{(n)}(\sigma)
    \end{bmatrix},
\end{equation*}
where the vectors $\gamma^{(j)}(\sigma)$ are understood to be \textit{column} vectors. Note that this is precisely the matrix appearing in the definition of the non-degeneracy condition \eqref{eq:nondegenerate} and is therefore invertible by our hypothesis. It is also convenient to let $[\gamma]_{\sigma,\lambda}$ denote the $n \times n$ matrix
\begin{equation}\label{gamma transformation}
[\gamma]_{\sigma,\lambda} := [\gamma]_{\sigma} \cdot D_{\lambda},
\end{equation}
where $D_{\lambda}:=\text{diag}(\lambda, \dots, \lambda^n)$, the diagonal matrix with eigenvalues $\lambda$, $\lambda^2, \dots, \lambda^n$. Consider the portion of the curve $\gamma$ lying over the subinterval $[\sigma-\lambda, \sigma+\lambda]$. This is parametrised by the map  $s \mapsto \gamma(\sigma + \lambda s)$ for $s \in [-1,1]$.  The degree $n$ Taylor polynomial of $s \mapsto \gamma(\sigma + \lambda s)$ around $\sigma$ is given by
\begin{equation}\label{Taylor}
  s \mapsto  \gamma(\sigma) + [\gamma]_{\sigma,\lambda} \cdot \gamma_{\circ}(s),
\end{equation}
which is indeed an affine image of $\gamma_{\circ}$. Furthermore, by Taylor's theorem, the original curve $\gamma$ agrees with the polynomial curve \eqref{Taylor} to high order at $\sigma$. 

Inverting the affine transformation $x \mapsto  \gamma(\sigma) + [\gamma]_{\sigma,\lambda} \cdot x$ from \eqref{Taylor}, we can map the portion of $\gamma$ over $[\sigma - \lambda, \sigma + \lambda]$ to a small perturbation of the moment curve. 

\begin{definition}\label{rescaled curve def} Let $\gamma \in C^{n+1}(I;\R^{n})$ be a non-degenerate curve and $\sigma \in I, \lambda>0$ be such that $[\sigma-\lambda, \sigma+ \lambda] \subseteq I$. The \textit{$(\sigma,\lambda)$-rescaling of $\gamma$} is the curve $\gamma_{\sigma,\lambda} \in C^{n+1}([-1,1];\R^{n})$ given by
\begin{equation*}
    \gamma_{\sigma,\lambda}(s) := [\gamma]_{\sigma,\lambda}^{-1}\big( \gamma(\sigma+\lambda s) - \gamma(\sigma) \big).
\end{equation*}
\end{definition}

It follows from the preceding discussion that 
\begin{equation*}
   \gamma_{\sigma,\lambda}(s) = \gamma_{\circ}(s) + [\gamma]_{\sigma,\lambda}^{-1} \mathcal{E}_{\gamma,\sigma,\lambda}(s)  
\end{equation*}
where $\mathcal{E}_{\gamma,\sigma,\lambda}$ is the remainder term for the Taylor expansion \eqref{Taylor}. In particular, if $\gamma$ satisfies the non-degeneracy condition \eqref{eq:nondegenerate} with constant $c_0$, then
\begin{equation*}
    \|   \gamma_{\sigma,\lambda} - \gamma_{\circ} \|_{C^{n+1}([-1,1];\R^n)} \lesssim c_0^{-1} \lambda \,  \| \gamma \|_{C^{n+1}(I)}^n.
\end{equation*}
Thus, if $\lambda>0$ is chosen to be small enough, then the rescaled curve $\gamma_{\sigma,\lambda}$ is a minor perturbation of the moment curve. In particular, given any $0 < \delta < 1$, we can choose $\lambda$ so as to ensure that $\gamma_{\sigma,\lambda}$ belongs to the following class of \textit{model curves}.

\begin{definition} Given $n \geq 2$ and $0 < \delta < 1$, let $\mathfrak{G}_n(\delta)$ denote the class of all smooth curves $\gamma \colon [-1, 1] \to \R^n$ that satisfy the following conditions: 
\begin{enumerate}[i)]
    \item $\gamma(0) = 0$ and $\gamma^{(j)}(0) = \vec{e}_j$ for $1 \leq j \leq n$;
    \item $\|\gamma - \gamma_{\circ}\|_{C^{n+1}([-1,1])} \leq \delta$.
\end{enumerate}
Here $\vec{e}_j$ denotes the $j$th standard Euclidean basis vector and
\begin{equation*}
    \|\gamma\|_{C^{n+1}(I)} := \max_{1 \leq j \leq n + 1} \sup_{s \in I} |\gamma^{(j)}(s)| \qquad \textrm{for all $\gamma \in C^{n+1}(I;\R^n)$.}
\end{equation*}
\end{definition}

Given any $\gamma \in \mathfrak{G}_n(\delta)$, condition ii) and the multilinearity of the determinant ensures that $\det[\gamma]_s = \det[\gamma_{\circ}]_s + O(\delta) = 1 + O(\delta)$. Thus, there exists a dimensional constant $c_n > 0$ such that if $0 < \delta < c_n$, then any curve $\gamma \in \mathfrak{G}_n(\delta)$ is non-degenerate and, moreover, satisfies $\det [ \gamma]_s \geq 1/2$. Henceforth, it is always assumed that any such parameter $\delta > 0$ satisfies this condition, which we express succinctly as $0 < \delta \ll 1$.




\section{Key analytic ingredients in the proof}\label{key ingredients sec}

There are three key ingredients in the proof of Theorem~\ref{LS thm}: a square function on $\R^4$, a square function on $\R^3$ and a Nikodym-type maximal operator mapping functions in $\R^4$ to functions in $\R^3$. These operators are formulated in terms of the geometry of the underlying curve $\gamma \colon I \to \R^3$ and, in particular, are defined with respect to the Frenet frame on $\gamma$.\footnote{More precisely, the square function on $\R^4$ is defined with respect to Frenet frame associated to a lift of $\gamma$ to $\R^4$.} In this section each of the three key operators is introduced and the relevant norm bounds for these objects are stated in Theorem~\ref{Frenet reverse SF theorem}, Proposition~\ref{f SF prop} and Proposition~\ref{Nikodym prop} below. In \S\S\ref{LS rel G sec}-\ref{J=3 sec}, a careful decomposition of the multiplier $m[a_k]$ is carried out which facilitates application of these results in the proof of Theorem~\ref{LS thm}. We return to proofs of Theorem~\ref{Frenet reverse SF theorem}, Proposition~\ref{f SF prop} and Proposition~\ref{Nikodym prop} in \S\ref{reverse SF sec}, \S\ref{forward SF sec} and \S\ref{Nikodym sec}, respectively. 



\subsection{Frenet geometry}

It is convenient to recall some elementary concepts from differential geometry which feature in our proof. 
Given a smooth non-denegenate curve $\gamma:I \to \R^n$, the Frenet frame is the orthonormal basis resulting from applying the Gram--Schmidt process to  the vectors
\begin{equation*}
    \{ \gamma'(s), \dots, \gamma^{(n)}(s)\},
\end{equation*}
which are linearly independent in view of the condition \eqref{eq:nondegenerate}. Defining the functions\footnote{Note that the $\tilde{\kappa}_j$ depend on the choice of parametrisation and only agree with the (geometric) curvature functions 
\begin{equation*}
    \kappa_j(s) := \frac{\langle \be_j'(s), \be_{j+1}(s) \rangle}{|\gamma'(s)|}
\end{equation*}
if $\gamma$ is unit speed parametrised. Here we do not assume unit speed parametrisation.} 
\begin{equation*}
    \tilde{\kappa}_j(s) := \langle \be_j'(s), \be_{j+1}(s) \rangle  \qquad \text{for } j=1, \dots, n-1,
\end{equation*}
one has the classical Frenet formul\ae
\begin{align*}\be_1'(s)&=  \tilde{\kappa}_1(s) \be_2(s),
    \\ \be_i'(s)&= -\tilde{\kappa}_{i-1}(s)\be_{i-1}(s) + \tilde{\kappa}_{i}(s)\be_{i+1}(s),\,\,i=2,\dots, n-1,
    \\
    \be_n'(s)&=-\tilde{\kappa}_{n-1}(s)\be_{n-1}(s).
\end{align*}
Repeated application of these formul\ae\ shows that 
\begin{equation*}
    \be^{(k)}_{i}(s) \perp \be_{j}(s) \qquad \textrm{whenever} \qquad 0 \leq  k < |i - j|.
\end{equation*}
Consequently, by Taylor's theorem 
\begin{equation*}
    |\inn{\be_{i}(s_1)}{\be_{j}(s_2)}| \lesssim_{\gamma} |s_1 - s_2|^{|i-j|} \qquad \textrm{for $1 \leq i,j \leq n$ and $s_1, s_2 \in I$.}
\end{equation*}
Furthermore, one may deduce from the definition of $\{\be_j(s)\}_{j=1}^n$ that
\begin{equation}\label{Frenet bound alt 1}
    |\inn{\gamma^{(i)}(s_1)}{\be_{j}(s_2)}| \lesssim_{\gamma} |s_1 - s_2|^{(j-i)\vee 0} \qquad \textrm{for $1 \leq i,j \leq n$ and $s_1, s_2 \in I$.}
\end{equation}

In this paper, much of the microlocal geometry of the averaging operators  $A_t$ is expressed in terms of the Frenet frame. We further introduce the following definitions.

\begin{definition}\label{def Frenet box}  Given $1 \leq d \leq n-1$ and $0 < r \leq 1$, for each $s \in I$ let $\pi_{d-1}(s;\,r)$ denote the set of all $\xi \in \widehat{\R}^n$ satisfying the following conditions:
\begin{subequations} 
\begin{align}
\label{neighbourhood 1}
     |\inn{\be_j(s)}{\xi}| &\leq r^{d+1-j} \qquad \textrm{for $1 \leq j \leq d$,} \\
\label{neighbourhood 2}
    1/2\leq |\inn{\be_{d+1}(s)}{\xi}| &\leq 2 \\
\label{neighbourhood 3}
    |\inn{\be_j(s)}{\xi}| &\leq 1 \qquad \textrm{for $d+2 \leq j \leq n$.}
\end{align}
\end{subequations} 
Such sets $\pi_{d-1}(s;\,r)$ are referred to as $(d-1,r)$-\textit{Frenet boxes}. 
\end{definition} 

The relevance of the $d-1$ index is that the $\pi_{d-1}(s;r)$ correspond to plate regions defined with respect to a codimension $d-1$ cone. For $n=4$ and $d-1 = 2$, this geometric observation is discussed in detail in \S\ref{geo obs sec}. 

\begin{definition} A collection $\mathcal{P}_{d-1}(r)$ of $(d-1,r)$-Frenet boxes is a \textit{Frenet box decomposition along $\gamma$} if it consists of precisely the $(d-1,r)$-Frenet boxes $\pi_{d-1}(s;\,r)$ for $s$ varying over an $r$-separated subset of $I$. 
\end{definition}




\subsection{Reverse square function estimates in \texorpdfstring{$\R^{3+1}$}{}} The most important ingredient in the proof of Theorem~\ref{LS thm} is the following square function bound.

\begin{theorem}\label{Frenet reverse SF theorem} Let $0 < r < 1$ and $\mathcal{P}_2(r)$ be a $(2,r)$-Frenet box decomposition along a non-degenerate $\gamma \colon I \to \R^4$. For all $\varepsilon > 0$ the inequality
\begin{equation*}
    \Big\|\sum_{\pi \in \mathcal{P}_2(r)} f_{\pi} \Big\|_{L^4(\R^4)} \lesssim_{\gamma,\varepsilon} r^{ -\varepsilon}  \Big\|\big(\sum_{\pi \in \mathcal{P}_2(r)}|f_{\pi}|^2 \Big)^{1/2}\big\|_{L^4(\R^4)}
\end{equation*}
holds for any tuple of functions $(f_{\pi})_{\pi \in \mathcal{P}_2(r)}$ satisfying $\supp \widehat{f}_{\pi} \subseteq \pi$.
\end{theorem}

This bound pertains to curves in $\R^4$ rather than $\R^3$ and therefore does not directly apply to the curve $\gamma \colon I \to \R^3$ featured in the definition of our original helical maximal operator. Rather, in \S\ref{spatio temp subsec} we apply Theorem~\ref{Frenet reverse SF theorem} to a certain lift of the original curve $\gamma$ into the spatio-temporal domain $\R^{3+1}$. This is somewhat analogous to the situation in \cite{MSS1992} where a square function estimate in $\R^{2+1}$ is used to study the circular maximal function in $\R^2$.\medskip

Theorem~\ref{Frenet reverse SF theorem} is related to the Lee--Vargas~\cite{LV2012} estimate for the Mockenhaupt square function in $\R^3$. In particular, the Mockenhaupt square function corresponds to studying functions frequency localised with repect to a $(1,r)$-Frenet box decomposition in $\R^3$. Moreover, the strategy used to prove Theorem ~\ref{reverse sf thm} mirrors that of \cite{LV2012}. We first obtain a $4$-linear variant of Theorem~\ref{Frenet reverse SF theorem} via the multilinear Fourier restriction estimates of Bennett--Bez--Flock--Lee \cite{BBFL2018}. The linear result is then deduced from the 4-linear inequality using a variant of the Bourgain--Guth method \cite{Bourgain2011}. The details of the argument are provided in \S\ref{reverse SF sec}.




\subsection{Forward square function estimates in \texorpdfstring{$\R^3$}{}} We also make use of a (forward) $L^2$-weighted square function estimate in $\R^3$. Here the square function estimate is defined in relation to a $(0,r)$-Frenet decomposition. In contrast with Theorem~\ref{Frenet reverse SF theorem}, we work with an operator-theoretic formulation involving certain projection operators. 

As before, let $\eta \in C^\infty_c(\R)$ be non-negative and such that $\eta(r) = 1$ if $r \in [-1,1]$ and $\supp \eta \subseteq [-2,2]$ and define $\tilde{\beta} := \eta(2^{-1}\,\cdot\,) - \eta(4\,\cdot\,)$. Give an $(0,r)$-Frenet box $\pi = \pi_{0,\gamma}(s;r)$ let
\begin{equation}\label{chi pi}
    \chi_{\pi}(\xi) := \eta\big(r^{-1}\,\inn{\be_1(s)}{\xi}\big) \, \tilde{\beta}\big(\inn{\be_2(s)}{\xi}\big) \, \eta \big(\inn{\be_3(s)}{\xi}\big)
\end{equation}
so that $\chi_{\pi}(\xi) = 1$ if $\xi \in \pi_{0,\gamma}(s;r)$ and $\chi_{\pi}$ vanishes outside some fixed dilate of this set.

\begin{proposition}\label{f SF prop}  Let $0 < r < 1$ and $\mathcal{P}_0(r)$ be a $(0,r)$-Frenet box decomposition for a non-degenerate $\gamma \colon I \to \R^3$. For all $\varepsilon > 0$ the inequality
\begin{equation*}
    \int_{\R^3} \sum_{\pi \in \mathcal{P}_0(r)} |\chi_{\pi}(D)f(x)|^2 w(x)\,\ud x \lesssim_{\varepsilon} r^{-\varepsilon} \int_{\R^3} |f(x)|^2 \widetilde{\mathcal{N}}_{\gamma, r}^{\, (\varepsilon)} w(x)\,\ud x
\end{equation*}
holds for any non-negative $w \in L^1_{\mathrm{loc}}(\R^3)$, where $\widetilde{\mathcal{N}}_{\,\gamma, r}^{\,(\varepsilon)}$ is a maximal operator satisfying 
\begin{equation}\label{f SF eq}
    \|\widetilde{\mathcal{N}}_{\,\gamma, r}^{\, (\varepsilon)}\|_{L^2(\R^3) \to L^2(\R^3)} \lesssim_{\varepsilon, \varepsilon_{\circ}} r^{-\varepsilon_{\circ}} \qquad \textrm{for all $\varepsilon_{\circ}  > 0$.}
\end{equation}
\end{proposition}

The above proposition is related to a $L^2$-weighted version of the classical sectorial square function of C\'ordoba \cite{Cordoba1982}, due to Carbery and the fourth author \cite[Proposition 4.6]{CS1995}. The proof is presented in \S~\ref{forward SF sec} below.\medskip

 The definition of $\widetilde{\mathcal{N}}_{\,\gamma, r}^{\, (\varepsilon)}$ is rather complicated, involving a repeated composition of Nikodym-type maximal operators at different scales. For this reason, we do not provide an explicit description of the operator here. Further details of the definition and basic properties of this operator are provided in \S~\ref{forward SF sec}. \medskip




\subsection{A singular Nikodym-type maximal function}\label{Nikodym lem subsec} The bounds on the spatio-temporal frequency localised pieces of our operator $m[a](D;\cdot)$ are reduced to bounding a Nikodym maximal function mapping functions in $\R^4$ to functions in $\R^3$. Given $\br \in (0,1)^3$ and $s \in [-1,1]$, consider the \textit{plates}
\begin{equation*}
    \mathcal{T}_{\br}(s) := \big\{ (y,t) \in \R^3 \times [1,2] :   \big|\inn{y - t\gamma(s)}{\be_j(s)}\big| \leq r_j \, \textrm{ for $j=1,2,3$} \big\} \subset \R^4.
\end{equation*}
Using these sets, we define associated averaging and maximal operators
 \begin{equation*}
    \mathcal{A}_{\br}^{\,\mathrm{sing}} g(x; s) :=  \fint_{\mathcal{T}_{\br}(s)} g(x-y, t) \,\ud y \ud t \quad \textrm{and} \quad \mathcal{N}_{\br}^{\,\mathrm{sing}} g(x) := \sup_{-1 \leq s \leq 1} |\mathcal{A}_{\br}^{\,\mathrm{sing}} g(x; s)|. 
 \end{equation*}
Note that $\mathcal{N}_{\br}^{\,\mathrm{sing}}$ takes as its input some $g \in L^1_{\mathrm{loc}}(\R^4)$ and outputs a measurable function on $\R^3$. In particular, there is a discrepancy between the number of input and the number of output variables of the operator. 

\begin{proposition}\label{Nikodym prop} If $\br \in (0,1)^3$ satisfies $r_3 \leq r_2 \leq r_1 \leq r_2^{1/2}$ and $r_2 \leq r_{1}^{1/2} r_3^{1/2}$, then
  \begin{equation*}
      \|\mathcal{N}_{\br}^{\,\mathrm{sing}} g\|_{L^2(\R^3)} \lesssim |\log r_3|^3 \|g\|_{L^2(\R^4)}.
  \end{equation*}
\end{proposition}

This result can be thought of as a higher dimensional analogue of a Nikodym maximal estimate from \cite{MSS1992}, which is used to study the circular maximal function in the plane. Note that the parameter triple $\br = (r, r, r)$ for some $0 < r < 1$ satisfies the hypothesis of Proposition~\ref{Nikodym prop}, corresponding to the case of tubes former around the rays $t \mapsto t\gamma(s)$. More relevant to our study, however, is the highly anisotropic situation where $\br = (r, r^2, r^3)$; note that this case is also covered by the proposition. It is remarked that the situation here is somewhat different to that appearing in Proposition~\ref{f SF prop} (which will be defined in \S\ref{forward SF sec}), owing to the aforementioned disparity between the number of input and output variables. The proof of Proposition~\ref{Nikodym prop}, which is based on an oscillatory integral argument, is presented in \S\ref{Nikodym sec} below.




\section{Proof of Theorem~\ref{LS thm}: the slow decay cone}\label{sec:slow decay cone}

Throughout the remainder of the paper, we work with some fixed $0 < \delta_0 \ll 1$, chosen to satisfy the forthcoming requirements of the proofs. For the sake of concreteness, the choice of $\delta_0 := 10^{-10}$ is more than enough for our purposes. It suffices to prove Theorem~\ref{LS thm} in the special case where $\gamma \in \mathfrak{G}_3(\delta_0)$ and $\supp \chi \subseteq I_0 := [-\delta_0,\delta_0]$. Indeed, using the observations of \S\ref{curve sym sec}, we may decompose and rescale the operator $m[a_k](D;\,\cdot\,)$ to reduce to this situation.  

 Suppose $\gamma \in \mathfrak{G}_3(\delta_0)$ and $a \in C^{\infty}(\widehat{\R}^3\setminus \{0\} \times \R \times \R)$ satisfies the hypotheses Theorem~\ref{LS thm}. In view of Theorem~\ref{PS LS J=2}, we may further assume that 
\begin{equation}\label{4 derivative bound_old}
 \left\{\begin{array}{ll}
    |\inn{\gamma^{(3)}(s)}{\xi}| \ge \frac{9}{10}\, |\xi| \\[5pt]
|\inn{\gamma^{(j)}(s)}{\xi}| \le 8\delta_0 |\xi| & \textrm{for $j = 1,2$}
\end{array}
\right. \qquad \textrm{for all $(\xi;t; s)\in \supp a$.} 
\end{equation}
 We note two further consequences of this technical reduction:
\begin{itemize}
    \item Since $\gamma \in \mathfrak{G}_3(\delta_0)$, we have $\gamma^{(j)}(0)=\vec{e}_j$ for $1 \leq j \leq 3$ and so \eqref{4 derivative bound_old} immediately implies that 
    \begin{equation*}
      |\xi_3|\ge \tfrac{9}{10} \, |\xi|    \quad \textrm{and} \quad |\xi_j|\le 8 \delta_0 |\xi| \quad \textrm{for $j=1, 2$,}  \qquad \textrm{for all $\xi \in \xisupp a$.}
    \end{equation*} 
    \item Since $\gamma \in \mathfrak{G}_3(\delta_0)$, we have $\|\gamma^{(4)}\|_{\infty} \leq \delta_0$. Thus, provided $\delta_0$ is sufficiently small,
\begin{equation}\label{4 derivative bound}
    |\inn{\gamma^{(3)}(s)}{\xi}| \geq \tfrac{1}{2} \, |\xi| \qquad \textrm{for all $(\xi; s) \in \xisupp a \times [-1,1]$}.
\end{equation}
Observe that this inequality holds on the large interval $[-1,1]$, rather than just $I_0$. 
\end{itemize}
Henceforth, we also assume that $\xi_3>0$ for all $\xi \in \xisupp a$. In particular,
\begin{equation}\label{convex}
    \inn{\gamma^{(3)}(s)}{\xi} > 0 \qquad \textrm{for all $(\xi; s ) \in \xisupp a \times [-1,1]$}
\end{equation}
and thus, for each $\xi \in \xisupp a$, the function $s \mapsto \inn{\gamma'(s)}{\xi}$ is strictly convex on $[-1,1]$. The analysis for the portion of the symbol supported on the set $\{\xi_3<0\}$ follows by symmetry.\medskip

The first step is to isolate regions of the frequency space where the multiplier $m[a]$ decays relatively slowly. Owing to stationary phase considerations, this corresponds to a region around the conic variety
\begin{equation*}
    \Gamma :=\{\xi \in \xisupp a:  \inn{\gamma^{(j)}(s)}{\xi}=0, \,\, 1 \leq j \leq 2, \text{ for some } s\in I_0 \}.
\end{equation*}
To analyse this cone, we begin with the following observation.

\begin{lemma}\label{theta2 lem}
If $\xi\in \xisupp a$, then the equation $\inn{\gamma''(s)}{\xi}= 0$ has a unique solution in $s  \in [-1,1]$, which corresponds to the unique global minimum of the function $s \mapsto \inn{\gamma'(s)}{\xi}$. Furthermore, the solution has absolute value $O(\delta_0)$.
\end{lemma}

\begin{proof}
Given $\xi \in \xisupp a$, let 
\begin{equation}\label{theta2 lem 1}
  \phi \colon [-1,1] \to \R, \quad \phi \colon s \mapsto \inn{\gamma'(s)}{\xi}.
\end{equation}
By \eqref{convex}, $\phi''(s)>0$ for all $s \in [-1,1]$  and the equation $\phi'(s)=\inn{\gamma^{(2)}(s)}{\xi}=0$ has at most one solution on that interval.

On the other hand, by the mean value theorem,
\begin{equation*}
    \phi'(s)=\inn{\gamma^{(2)}(s)}{\xi}= \xi_2 + \omega(\xi;s) \, s,
\end{equation*}
where $\omega$ satisfies $|\omega(\xi;s)|\geq  \tfrac{1}{2}|\xi| >0$. As $|\xi_2| \leq 8 \delta_0 |\xi|$, it follows that $|\omega(\xi;s)| |s| > |\xi_2|$ if $|s| > 16 \delta_0$, and so the equation $\inn{\gamma^{(2)}(s)}{\xi}=0$ has a unique solution in the interval $[-16\delta_0,16\delta_0]$. Moreover, it immediately follows from \eqref{convex} that this solution is the unique global minimum of $\phi$ on $[-1, 1]$. 
\end{proof}

Using Lemma~\ref{theta2 lem}, we construct a smooth mapping $\theta_2 \colon \xisupp a  \to [-1,1]$ such that
\begin{equation*}
    \inn{\gamma'' \circ \theta_2(\xi)}{\xi} = 0 \qquad \textrm{for all $\xi \in \xisupp a$.}
\end{equation*}
It is easy to see that $\theta_2$ is homogeneous of degree 0. This function can be used to construct a natural Whitney decomposition with respect to the cone $\Gamma$ defined above. In particular, let 
\begin{equation}\label{J=3 u function}
    \bu(\xi) := \inn{\gamma' \circ \theta_2(\xi)}{\xi} \qquad \textrm{for all $\xi \in \xisupp a$.}
\end{equation}
This quantity plays a central r\^ole in our analysis. If $u(\xi)=0$, then $\xi \in \Gamma$ and so, roughly speaking, $u(\xi)$ measures the distance of $\xi$ from $\Gamma$.

\begin{lemma}\label{theta1 lem}
Let $\xi \in \xisupp a$ and consider the equation \begin{equation}\label{0404e3.29}
    \inn{\gamma'(s)}{\xi}=0.
\end{equation}
\begin{enumerate}[i)]
    \item If $u(\xi)>0$, then the equation \eqref{0404e3.29} has no solution on $[-1, 1]$.
\item If $u(\xi)=0$, then the equation \eqref{0404e3.29} has only the solution $s=\theta_2(\xi)$ on $[-1,1]$.
\item If $u(\xi)<0$, then the equation \eqref{0404e3.29} has precisely two solutions on $[-1,1]$. Both solutions have absolute value $O(\delta_0^{1/2})$.
\end{enumerate}
\end{lemma}

\begin{proof}
Given $\xi \in \xisupp a$, define $\phi$ as in \eqref{theta2 lem 1}.\medskip 

\noindent i)  In this case, Lemma~\ref{theta2 lem} implies that
\begin{equation*}
 \phi(s) = \inn{\gamma'(s)}{\xi} \geq u (\xi) >0 \quad \text{ for all $s \in [-1,1]$,} 
\end{equation*}
and so \eqref{0404e3.29} has no solutions.\medskip

\noindent ii) This case also follows immediately from Lemma~\ref{theta2 lem}, since $s=\theta_2(\xi)$ is the only global minimum for $\phi$ on $[-1,1]$. \medskip 

\noindent iii) Recall, by \eqref{convex}, the function $\phi$ is strictly convex on $[-1,1]$, and therefore $\phi(s)=0$ has at most two solutions on that interval. \smallskip

On the other hand, by (the proof of) Lemma~\ref{theta2 lem} we know that $|\theta_2(\xi)| \leq 16 \delta_0$. Moreover, the mean value theorem implies
\begin{equation}\label{theta1 lem 1}
   |u(\xi)| \leq |\xi_1| + \sup_{|s| \leq 16 \delta_0}|\gamma^{(2)}(s)||\xi| |\theta_2(\xi)| \leq 8 \Big(1+ 2\sup_{|s| \leq 16 \delta_0}| \gamma^{(2)}(s) |\Big)  \delta_0 |\xi|  \leq 40 \delta_0 |\xi|,
\end{equation}
since $\gamma \in \mathfrak{G}_3(\delta_0)$. By Taylor expansion of $\phi$ around $\theta_2(\xi)$, one obtains
\begin{equation}\label{theta1 lem 2}
 \phi(s) = u(\xi) + \omega(\xi;s) \, (s-\theta_2(\xi))^2,
\end{equation}
where $\omega$ arises from the remainder term and satisfies $\omega(\xi;s) \geq \tfrac{1}{4} \, |\xi|$. Combining \eqref{theta1 lem 1} and \eqref{theta1 lem 2}, it follows that if $|s-\theta_2(\xi)|\geq 20\delta_0^{1/2}$, then $\phi(s)>0$. Recall that $\phi \circ \theta_2(\xi)=u(\xi)<0$. Consequently, the equation $\phi(s) =0$ has exactly two solutions on the interval 
\begin{equation*}
   [-16\delta_0, 16\delta_0] + [-20\delta_0^{1/2},  20\delta_0^{1/2}] \subseteq [-36 \delta_0^{1/2}, 36 \delta_0^{1/2}],
\end{equation*} 
as required. 
\end{proof}

Using Lemma~\ref{theta1 lem}, we construct a (unique) pair of smooth mappings
\begin{equation*}
    \theta_1^{\pm} \colon \{ \xi \in \xisupp a :  u(\xi) <0 \} \to [-1, 1]
\end{equation*}
with $\theta_1^-(\xi)\le \theta_1^+(\xi)$ which satisfies
\begin{equation*}
    \inn{\gamma' \circ \theta_1^{\pm}(\xi)}{\xi}= 0 \quad \text{ for all $\xi \in \xisupp a$ with  $u(\xi) <0$.}
\end{equation*}
Define the functions 
\begin{equation*}
        v^{\pm}(\xi):=\inn{\gamma'' \circ \theta_1^\pm(\xi)}{\xi}  \qquad \text{ for all $\xi \in \supp a$ with  $u(\xi) <0$.}
\end{equation*}

\begin{lemma}\label{root control lem}
Let $\xi \in \supp a$ with $u(\xi)<0$. Then the following hold:
\begin{equation*}
    \big|v^{\pm}\big(\xip\big)\big| \sim |\theta_1^{\pm}(\xi) - \theta_2(\xi)| \sim  |\theta_1^+(\xi)-\theta_1^-(\xi)|\sim \big|u\big(\xip\big)\big|^{1/2}.
\end{equation*}
\end{lemma}

\begin{proof} By Taylor expansion around $\theta_2(\xi)$, we obtain
\begin{align*}
    v^{\pm}(\xi)& = \omega_{1}^{\pm}(\xi) \, (\theta_1^{\pm}(\xi)-\theta_2(\xi)),\\
    0&=\inn{\gamma' \circ \theta_1^{\pm}(\xi)}{\xi} = u(\xi) + \omega_2(\xi) \, (\theta_1^{\pm}(\xi)-\theta_2(\xi))^2 
\end{align*}    
where $|\omega_1^{\pm}(\xi)|  \sim |\omega_{2}(\xi)|\sim |\xi|$ by \eqref{4 derivative bound}. Similarly, Taylor expansion around $\theta_1^{\pm}(\xi)$ yields 
\begin{equation*}
     0 = \inn{\gamma' \circ \theta_1^{+}(\xi)}{\xi} = v^-(\xi)\, (\theta_1^{+}(\xi) - \theta_1^{-}(\xi)) + \omega_3 (\xi) \, (\theta_1^{+}(\xi) - \theta_1^{-}(\xi))^2 
\end{equation*}    
where again the remainder satisfies $|\omega_3(\xi)| \sim |\xi|$. As $\theta_1^+(\xi) \neq \theta_1^-(\xi)$, we can combine the identities above to obtain the desired bounds.
\end{proof}




\section{Proof of Theorem~\ref{LS thm}: Local smoothing relative to \texorpdfstring{$\Gamma$}{}}\label{LS rel G sec}




For $k \geq 1$, consider the frequency localised symbols $a_k:=a \, \beta^k$, as introduced in \S\ref{sec:bandlimited}. We decompose each $a_k$ with respect to the size of $|\bu(\xi)|$.
In particular, write\footnote{Here $\beta$ function should be defined slightly differently compared with \eqref{beta def} and, in particular, here $\beta(r) := \eta(2^{-2}r) - \eta(r)$. Such minor changes are ignored in the notation.}
    \begin{equation}\label{J=3 akell def}
       a_k = \sum_{\ell = 0}^{\floor{k/3}}  a_{k,\ell}
       \qquad \textrm{where} \qquad
    a_{k,\ell}(\xi; t; s) := 
    \left\{\begin{array}{ll}
       \displaystyle a_k(\xi; t; s) \, \beta\big(2^{-k+ 2\ell}\bu(\xi)\big) \quad   & \textrm{if $0 \leq \ell < \floor{k/3}$}  \\[8pt]
       \displaystyle a_{k}(\xi; t; s)\, \eta\big(2^{-k + 2\floor{k/3}}\bu(\xi)\big) & \textrm{if $\ell = \floor{k/3}$}
    \end{array}\right. .
    \end{equation}
Here $\floor{k/3}$ denotes the greatest integer less than or equal to $k/3$.\medskip

To prove Theorem~\ref{LS thm}, we establish local smoothing estimates for each of the operators $m[a_{k,\ell}](D;\,\cdot\,)$. The main result is as follows.

\begin{proposition}\label{J=3 LS prop} Let $0 \leq \ell \leq  \floor{k/3}$.
For all $2 \leq p \leq  4$ and $\varepsilon > 0$,
\begin{equation*}
    \|m[a_{k,\ell}](D; \,\cdot\,) f\|_{L^p(\R^{3+1})} \lesssim_{\varepsilon} 2^{-k/p - \ell(1 - 3/p)} 2^{ \varepsilon k}\| f \|_{L^p(\R^3)}.
\end{equation*}
\end{proposition}

Proposition~\ref{J=3 LS prop} provides an effective bound in the large $\ell$ regime (in particular, for $\floor{k/5} \leq \ell \leq \floor{k/3}$). This corresponds to those pieces of the multiplier which are supported close to the binormal cone $\Gamma$, and therefore lie in a neighbourhood of the most significant singularity. 

In addition to Proposition~\ref{J=3 LS prop}, we also use results from \cite{PS2007} to deal with the less singular pieces of the multiplier.

\begin{proposition}[\cite{PS2007}]\label{PS LS prop} Let $0 \leq \ell \leq  \floor{k/3}$.
For all $2 \leq p \leq  6$ and $\varepsilon > 0$,
\begin{equation*}
    \|m[a_{k,\ell}](D; \,\cdot\,) f\|_{L^p(\R^{3+1})} \lesssim_{\varepsilon} 2^{-\frac{k-\ell}{2}(\frac{1}{2} + \frac{1}{p}) + \varepsilon k}\| f \|_{L^p(\R^3)}.
\end{equation*}
\end{proposition}

This proposition follows from Theorem~\ref{PS LS J=2} via the sharp Wolff inequality for the light cone \cite{BD2015} and a rescaling argument (c.f. \S\ref{overview subsec}). The details of the proof can be found in~\cite[\S5]{PS2007}.

\begin{proof}[Proof of Theorem~\ref{LS thm}, assuming Proposition~\ref{J=3 LS prop}] Applying the decomposition \eqref{J=3 akell def} and the triangle inequality, 
\begin{equation*}
    \|m[a_k](D; \,\cdot\,) f\|_{L^p(\R^{3+1})} \leq \sum_{\ell=0}^{\floor{k/5}} \|m[a_{k,\ell}](D; \,\cdot\,) f\|_{L^p(\R^{3+1})} + \sum_{\ell=\floor{k/5} + 1}^{\floor{k/3}} \|m[a_{k,\ell}](D; \,\cdot\,) f\|_{L^p(\R^{3+1})}. 
\end{equation*}
For $2 \leq p \leq 4$ we may bound the terms of the first sum using Proposition~\ref{PS LS prop} and the terms of the second using Proposition~\ref{J=3 LS prop}. If, in addition, we assume $p \geq 3$, then the geometric series resulting from the constants can be evaluated to give the desired bound.
\end{proof}




\section{Proof of Theorem~\ref{LS thm}: the main argument}\label{J=3 sec} 

By the observations of the previous section, the problem is reduced to establishing Proposition~\ref{J=3 LS prop}. In this section we provide the details of the proof, following the scheme sketched in~\S\ref{overview subsec}. 

\subsection{Localisation along the curve}\label{loc curv subsec}
We begin by further decomposing the symbols with respect to the distance of the $s$-variable to the roots $\theta_1^{\pm}$ and $\theta_2(\xi)$. Here it is convenient to introduce a `fine tuning' constant $\rho > 0$. This is a small (but absolute) constant which plays a minor technical r\^ole in the forthcoming arguments: taking $\rho := 10^{-6}$ more than suffices for our purposes.

Recall from Lemma~\ref{theta1 lem} that the two distinct roots $\theta_1^{\pm}(\xi)$ only occur when $u(\xi) < 0$. In view of this, let $\beta^{>0}$, $\beta^{<0} \in C_c^\infty(\R)$ be the unique functions with $\supp \beta^{>0} \subset (0,\infty)$ and $\supp \beta^{<0} \subset (-\infty, 0)$ such that $\beta = \beta^{>0} +\beta^{<0}$. This induces a corresponding decomposition $a_{k,\ell} = a_{k,\ell}^{>0} + a_{k, \ell}^{<0}$ for $0 \leq \ell < \floor{k/3}$, where $u(\xi)$ is positive (respectively, negative) on the support of $a_{k,\ell}^{>0}$ (respectively, $a_{k,\ell}^{<0}$). Given $\varepsilon>0$, define
\begin{equation*}
a_{k,\ell}^{(\varepsilon), \pm}(\xi;t;s):=a_{k,\ell}^{<0}(\xi;t;s) \, \eta \big(\rho^{-1}2^{(k-\ell)/2}2^{-k\varepsilon}|s- \theta_1^{\pm}(\xi)|\big) \quad \textrm{if $0 \leq \ell < \floor{k/3}_{\,\varepsilon}$} 
\end{equation*}
and
\begin{equation}\label{akell dec} 
    a_{k,\ell}^{(\varepsilon)}(\xi;t;s):=   \left\{\begin{array}{ll}
       \displaystyle \sum_\pm a_{k,\ell}^{(\varepsilon), \pm}(\xi;t;s)  \quad
       & \textrm{if $0 \leq \ell < \floor{k/3}_{\,\varepsilon}$}  \\[8pt]
       \displaystyle a_{k,\ell}(\xi;t;s) \, \eta \big(\rho 2^{\ell(1-\varepsilon)}|s-\theta_2(\xi)|\big) & \textrm{if $\floor{k/3}_{\,\varepsilon} \leq \ell \leq \floor{k/3}$}
    \end{array}\right. ,
\end{equation}
where $\floor{k/3}_{\,\varepsilon} := \floor{\big(\frac{1 - \varepsilon}{3}\big) \cdot k }$ is a number we think of as being slightly smaller than $\floor{k/3}$. Note that 
\begin{equation*}
 \min_\pm |s-\theta_1^\pm(\xi)| \lesssim \rho 2^{-(k-\ell)/2 + k \varepsilon} \qquad  \textrm{for all $(\xi;t;s) \in \supp a_{k,\ell}^{(\varepsilon)}$ \quad  if $0 \leq \ell < \floor{k/3}_\varepsilon$}.   
\end{equation*}

\begin{remark} 
The symbols $a_{k,\ell}^{(\varepsilon), +}$ and $a_{k,\ell}^{(\varepsilon), -}$ have disjoint supports if $0 \leq \ell < \floor{k/3}_{\varepsilon}$. Indeed, the decomposition ensures that $|u(\xi)| \sim 2^{k-2\ell}$ for all $\xi \in \xisupp a_{k,\ell}^{(\varepsilon)}$ and so Lemma~\ref{root control lem} implies
\begin{equation*}
 |\theta_1^-(\xi) - \theta_1^+(\xi)| \gtrsim 2^{-\ell} \gtrsim 2^{-(k-\ell)/2}2^{k\varepsilon}.
\end{equation*}
Here we use the hypothesis $\ell < \floor{k/3}_{\,\varepsilon}$. Provided $\rho$ is chosen to be sufficiently small, the above separation condition ensures that the disjointness of the supports of $a_{k,\ell}^{(\varepsilon), +}$ and $a_{k,\ell}^{(\varepsilon), -}$. Consequently,
\begin{equation*}
    \min_\pm|s-\theta_1^\pm(\xi)| \gtrsim 2^{-(k-\ell)/2 + k \varepsilon} \qquad \textrm{for all $(\xi;t;s) \in \supp (a_{k,\ell}^{<0} - a_{k,\ell}^{(\varepsilon)})$}
\end{equation*}
if $0 \leq \ell < \floor{k/3}_\varepsilon$.
\end{remark}

The main contribution to $m[a_{k,\ell}]$ comes from the symbols $a_{k,\ell}^{(\varepsilon)}$.

\begin{lemma}\label{J=3 s loc lem} Let $2 \leq p < \infty$ and $\varepsilon>0$.
For all $0 \leq \ell \leq \floor{k/3}$
    \begin{equation*}
        \|m[a_{k,\ell} - a_{k,\ell}^{(\varepsilon)}](D; \,\cdot\,) f \|_{L^p(\R^{3+1})}  \lesssim_{N,\varepsilon,p} 2^{-kN} \| f \|_{L^p(\R^3)} \qquad \textrm{for all $N \in \N$.}
    \end{equation*}
\end{lemma}

\begin{proof}
It is clear that the multipliers satisfy a trivial $L^{\infty}$-estimate with operator norm $O(2^{Ck})$ for some absolute constant $C \geq 1$. Thus, by interpolation, it suffices to prove the rapid decay estimate for $p = 2$ only. This amounts to showing that, under the hypotheses of the lemma, 
\begin{equation}\label{J=3 curve loc 0}
     \|m[a_{k,\ell}-a_{k,\ell}^{(\varepsilon)}](\,\cdot\,; t)\|_{L^{\infty}(\widehat{\R}^3)} \lesssim_{N,\varepsilon} 2^{-kN} \qquad \textrm{for all $N \in \N$}
 \end{equation}
 uniformly in $1/2 \leq t \leq 4$.\medskip

\noindent \underline{Case: $\floor{k/3}_{\,\varepsilon} \leq \ell \leq \floor{k/3}$}. Here the localisation of the $a_{k,\ell}$ and $a_{k,\ell}^{(\varepsilon)}$ symbols ensures that
\begin{equation}\label{J=3 curve loc 1}
     |\bu(\xi)| \lesssim 2^{k-2\ell} \quad \text{ and } \quad |s-\theta_2(\xi)| \gtrsim \rho^{-1} 2^{-\ell(1 - \varepsilon)}  \quad \textrm{for all $(\xi;t;s) \in \supp (a_{k,\ell}-a_{k,\ell}^{(\varepsilon)})$,}
\end{equation}
where $u$ is the function introduced in \eqref{J=3 u function}.

Fix $\xi \in \xisupp (a_{k,\ell}-a_{k,\ell}^{(\varepsilon)})$ and consider the oscillatory integral $m[a_{k,\ell} - a_{k,\ell}^{(\varepsilon)}](\xi;t)$, which has phase $s \mapsto t \, \inn{\gamma(s)}{\xi}$. Taylor expansion around $\theta_2(\xi)$ yields 
\begin{align}\label{J=3 curve loc 3}
    \inn{\gamma'(s)}{\xi} &= \bu(\xi) + \omega_1(\xi;s) \, (s-\theta_2(\xi))^2  \\
\label{J=3 curve loc 4}
    \inn{\gamma''(s)}{\xi} &=\omega_2(\xi;s) \, (s-\theta_2(\xi))
\end{align}
where $\omega_i$ arise from the remainder terms and satisfy $|\omega_i(\xi;s)| \sim  2^k$. Provided $\rho$ is sufficiently small, \eqref{J=3 curve loc 1} implies that the $\omega_1(\xi;s) \, (s-\theta_2(\xi))^2$ term dominates the right-hand side of \eqref{J=3 curve loc 3} and therefore 
 \begin{equation}\label{J=3 curve loc 5}
     |\inn{\gamma'(s)}{\xi}| \gtrsim  2^{k} |s-\theta_2(\xi)|^2\qquad  \textrm{for all $(\xi;t;s) \in \supp (a_{k,\ell}-a_{k,\ell}^{(\varepsilon)})$.}
 \end{equation}
Furthermore, \eqref{J=3 curve loc 4}, \eqref{J=3 curve loc 5} and the localisation \eqref{J=3 curve loc 1} immediately imply
\begin{align*}
    |\inn{\gamma''(s)}{\xi}| &  \lesssim  2^{-k+3\ell(1-\varepsilon)}|\inn{\gamma'(s)}{\xi}|^2, \\
    |\inn{\gamma^{(j)}(s)}{\xi}| & \lesssim 2^{k} \lesssim_j 2^{-(k-3\ell(1-\varepsilon))(j-1)} |\inn{\gamma'(s)}{\xi}|^j
     \qquad \text{for all $j \geq 3$}
\end{align*}
and all $(\xi;t;s) \in \supp (a_{k,\ell}-a_{k,\ell}^{(\varepsilon)})$,
 where in the last inequality we have used $|s-\theta_2(\xi)|^{j-3} \lesssim 1$ for all $j \geq 3$.

On the other hand, by the definition of the symbols, \eqref{J=3 curve loc 5} and the localisation in \eqref{J=3 curve loc 1},
  \begin{equation*}
 |\partial_s^N (a_{k,\ell}-a_{k,\ell}^{(\varepsilon)})(\xi;s)| \lesssim_N 2^{\ell(1-\varepsilon) N} \lesssim 2^{-(k-3\ell)N - 3\varepsilon \ell N}|\inn{\gamma'(s)}{\xi}|^N \qquad \textrm{for all $N \in \N$}
 \end{equation*}
 and  all  $(\xi; t; s) \in \supp (a_{k,\ell} - a_{k,\ell}^{(\varepsilon)})$.
 Thus, by repeated integration-by-parts (via Lemma~\ref{non-stationary lem}, with $r=2^{k-3\ell+ 3\varepsilon \ell } \geq 1$ for $0 \leq \ell \leq k/3$), one concludes that
\begin{equation*}
    |m[a_{k,\ell}-a_{k,\ell}^{(\varepsilon)}](\xi;t)| \lesssim_N 2^{-(k-3\ell)N - 3\varepsilon \ell N} \qquad \textrm{for all $N \in \N$}
\end{equation*}
uniformly in $1/2 \leq t \leq 4$. Since $\floor{k/3}_{\,\varepsilon} \leq \ell \leq \floor{k/3} \leq k/3$, the desired bound follows.\medskip

\noindent \underline{Case: $0 \leq \ell < \floor{k/3}_{\,\varepsilon}$}. If $u(\xi) > 0$, then \eqref{convex} and \eqref{J=3 curve loc 3} imply 
\begin{equation*}
    |\inn{\gamma'(s)}{\xi}| \gtrsim |u(\xi)| + 2^k|s - \theta_2(\xi)|^2 \qquad \textrm{for all $(\xi;s) \in \supp a_{k,\ell}^{>0}$.}
\end{equation*}
Furthermore, the localisation of the symbol $a_{k,\ell}^{>0}$ guarantees that $u(\xi) \sim 2^{k-\ell}$ for all $\xi \in \supp a_{k,\ell}^{>0}$.  It is then a straightforward exercise to adapt the argument used in the previous case to show $
    \|m[a_{k,\ell}^{>0}](\,\cdot\,; t)\|_{\infty} \lesssim_{N,\varepsilon} 2^{-kN}$, splitting the analysis into the cases $|s-\theta_2(\xi)|\geq 2^{-\ell}$ and $|s-\theta_2(\xi)|\leq 2^{-\ell}$. Here we use the fact that $2^{-(k-3\ell)} \leq 2^{-\varepsilon k}$. 
    
Thus, the problem is reduced to proving 
\begin{equation*}
    \|m[a_{k,\ell}^{<0} -a_{k,\ell}^{(\varepsilon)}](\,\cdot\,; t)\|_{L^{\infty}(\widehat{\R}^3)} \lesssim_{N,\varepsilon} 2^{-kN}. 
\end{equation*}    
Here the localisation of the $a_{k,\ell}^{<0}$ and $a_{k,\ell}^{(\varepsilon)}$ symbols ensures that
\begin{equation}\label{J=3 curve loc 1a}
     |\bu(\xi)| \sim 2^{k-2\ell} \quad \text{ and } \quad \min_{\pm}|s-\theta_1^{\pm}(\xi)| \gtrsim  2^{-(k-\ell)/2 + k \varepsilon}  \quad \textrm{for all $(\xi;t;s) \in \supp (a_{k,\ell}^{<0} - a_{k,\ell}^{(\varepsilon)})$,}
\end{equation}
where $u$ is the function introduced in \eqref{J=3 u function}.

Fix $\xi \in \xisupp (a_{k,\ell}^{<0} - a_{k,\ell}^{(\varepsilon)})$ and consider the oscillatory integral $m[a_{k,\ell}^{<0} - a_{k,\ell}^{(\varepsilon)}](\xi;t)$, which has phase $s \mapsto t \, \inn{\gamma(s)}{\xi}$. If we define 
\begin{equation*}
    \phi \colon [-1,1] \to \R, \quad \phi \colon s \mapsto \inn{\gamma'(s)}{\xi},
\end{equation*}
then, by \eqref{convex}, this function is strictly convex. Thus, given $t \in [-1,1]$, the auxiliary function
\begin{equation*}
    q_t \colon [-1,1] \to \R, \quad q_t \colon s \mapsto \frac{\phi(s) - \phi(t)}{s - t} \quad \textrm{for $s \neq t$} \quad \textrm{and} \quad q_t \colon t \mapsto \phi'(t)
\end{equation*}
is increasing. Setting $t := \theta_1^-(\xi)$ and noting that $\phi\circ\theta_1^-(\xi)=0$, it follows that
\begin{equation*}
    \frac{\phi(s)}{s-\theta_1^-(\xi)}  \leq \frac{\phi\circ \theta_2(\xi)}{\theta_2(\xi)-\theta_1^-(\xi)} = \frac{u(\xi)}{\theta_2(\xi)-\theta_1^-(\xi)} < 0 \qquad \text{ for all } -1 \leq s \leq  \theta_2(\xi),
\end{equation*}
where we have used the fact that $u(\xi) < 0$ on the support of $a_{k,\ell}^{<0}$. If $s \in [\theta_2(\xi), 1]$, then we can carry out the same argument with respect to $t = \theta_1^+(\xi)$ to obtain a similar inequality. From this, we deduce the bound 
\begin{equation}\label{J=3 curve loc 2a}
    |\inn{\gamma'(s)}{\xi}| \geq \min_{\pm} \frac{|u(\xi)| |s-\theta_1^{\pm}(\xi)|} {|\theta_2(\xi)- \theta_1^{\pm}(\xi)|}  \qquad \text{ for all } -1 \leq s \leq 1.
\end{equation}
Recall from \eqref{J=3 curve loc 1a} that $|u(\xi)| \sim  2^{k-2\ell}$ and  therefore $|\theta_2(\xi)- \theta_1^{\pm}(\xi)| \sim 2^{-\ell}$ by Lemma~\ref{root control lem}. Substituting these bounds and the second bound in \eqref{J=3 curve loc 1a} into \eqref{J=3 curve loc 2a}, we conclude that 
\begin{equation}\label{J=3 curve loc 3a}
    |\inn{\gamma'(s)}{\xi}| \gtrsim 2^{k-\ell}\min_{\pm}|s-\theta_1^{\pm}(\xi)| \gtrsim 2^{(k-\ell)/2 + \varepsilon k} \qquad  \textrm{for all $(\xi;t;s) \in \supp (a_{k,\ell}^{<0} - a_{k,\ell}^{(\varepsilon)})$.}
\end{equation}
Furthermore, by the mean value theorem,
\begin{equation*}
    |\inn{\gamma''(s)}{\xi}| \lesssim \max_{\pm} |v^{\pm}(\xi)| + 2^k\min_{\pm}|s - \theta_1^{\pm}(\xi)| \lesssim 2^{k-\ell} + 2^{\ell}|\inn{\gamma'(s)}{\xi}| \lesssim 2^{-k\varepsilon}|\inn{\gamma'(s)}{\xi}|^2,
\end{equation*}
where we have used \eqref{J=3 curve loc 3a}, the condition $|v^{\pm}(\xi)| \sim 2^{k-\ell}$ for $\xi \in \supp a_{k,\ell}^{<0}$ from Lemma~\ref{root control lem} and $0 \leq \ell \leq k/3$ in the last inequality. For higher order derivatives, 
\begin{equation*}
    |\inn{\gamma^{(j)}(s)}{\xi}| \lesssim_j 2^k \lesssim_j 2^{-(j-1)k\varepsilon } |\inn{\gamma'(s)}{\xi}|^j \qquad \textrm{for all $j \geq 3$}
\end{equation*}
and all $(\xi;t;s) \in \supp (a_{k,\ell}^{<0} - a_{k,\ell}^{(\varepsilon)})$. On the other hand, by the definition of the symbols and \eqref{J=3 curve loc 3a} we have
  \begin{equation*}
 |\partial_s^N (a_{k,\ell}-a_{k,\ell}^{(\varepsilon)})(\xi;s)| \lesssim_N 2^{N(k-\ell)/2} 2^{-Nk \varepsilon} \lesssim 2^{-2Nk \varepsilon}|\inn{\gamma'(s)}{\xi}|^N \qquad \textrm{for all $N \in \N$}
 \end{equation*}
 and  all  $(\xi; t; s) \in \supp (a_{k,\ell}^{<0}-a_{k,\ell}^{(\varepsilon)})$.
 Thus, by repeated integration-by-parts (via Lemma~\ref{non-stationary lem}, with $r := 2^{k\varepsilon/2} \geq 1$), one obtains the desired bound \eqref{J=3 curve loc 0}.
 \end{proof}




\subsection{Fourier localisation}\label{J=3 Fourier loc subsec} We perform a radial decomposition of the symbols $a_{k,\ell}^{(\varepsilon)}$ with respect to the homogeneous functions $\theta_2$ and $\theta_1^{\pm}$.  Fix $\zeta \in C^{\infty}(\R)$ with $\supp \zeta \subseteq [-1,1]$ such that $\sum_{l \in \Z} \zeta(\,\cdot\, - l) \equiv 1$. 
For $k \in \N$ and $0 \leq \ell < \floor{k/3}_\varepsilon$, write
\begin{equation*}
    a_{k,\ell}^{(\varepsilon)} = \sum_\pm \sum_{\nu \in \Z}  a_{k,\ell}^{\nu, (\varepsilon), \pm}
\end{equation*}
where 
\begin{equation*}
    a_{k,\ell}^{\nu, (\varepsilon),\pm}(\xi;t;s) := 
       a_{k,\ell}^{(\varepsilon), \pm}(\xi;t;s) \,  \zeta\big(\rho^{-1}(2^{(k-\ell)/2}\theta_1^{\pm}(\xi) - \nu)\big) \qquad \textrm{if $0 \leq \ell < \floor{k/3}_{\varepsilon}$}.
\end{equation*}
Each of the two terms in $\sum_\pm$ can be treated analogously. In order to simplify the notation, we drop the symbol $\pm$ from $a_{k,\ell}^{\nu, (\varepsilon)}$ and $\theta_1^\pm$ and adopt the convention
\begin{equation}\label{dec nu}
    a_{k,\ell}^{(\varepsilon)} = \sum_{\nu \in \Z}  a_{k,\ell}^{\nu, (\varepsilon)}.
\end{equation}
The key properties of this decomposition are that
\begin{equation}\label{akellnu dec 1}
    |s-\theta_1(\xi)| \lesssim \rho 2^{-(k-\ell)/2 + k \varepsilon} \quad \text{ and } \quad  |\theta_1(\xi) - s_\nu| \lesssim \rho 2^{-(k-\ell)/2 } \quad  \text{for all $(\xi; t; s) \in \supp a_{k,\ell}^{\nu,(\varepsilon)} $},
\end{equation}
where $s_\nu := 2^{-(k-\ell)/2} \nu$ and $\theta_1 \in \{\theta_1^+(\xi), \theta_1^-(\xi)\}$. The decomposition \eqref{dec nu} is also extended to the range $ \floor{k/3}_\varepsilon \leq \ell \leq \floor{k/3}$, with 
\begin{equation}\label{akellnu dec 2}
    a_{k,\ell}^{\nu, (\varepsilon)}(\xi;t;s) :=a_{k,\ell}^{(\varepsilon)}(\xi;t;s) \, \zeta(2^{\ell}\theta_2(\xi) - \nu) \qquad  \textrm{if $\floor{k/3}_{\,\varepsilon} \leq \ell \leq \floor{k/3}$}.
\end{equation}

In the case $0 \leq \ell < \floor{k/3}_{\,\varepsilon}$ we also consider symbols formed by grouping the $a_{k,\ell}^{\nu, (\varepsilon)}$ into pieces at the larger scale $2^{-\ell}$. Given $0 \leq \ell < \floor{k/3}_{\,\varepsilon}$ we write $\Z = \bigcup_{\mu \in \Z} \mathfrak{N}_{\ell}(\mu)$, where the sets $\mathfrak{N}_{\ell}(\mu)$ are disjoint and satisfy
\begin{equation*}
    \mathfrak{N}_{\ell}(\mu)\subseteq \{\nu \in \Z: |\nu-2^{(k-3\ell)/2} \mu| \leq 2^{(k-3\ell)/2} \}.
\end{equation*}
For each $\mu \in \Z$, we then define
\begin{equation*}
    a_{k,\ell}^{*,\mu, (\varepsilon)} := \sum_{\nu \in \mathfrak{N}_{\ell}(\mu)} a_{k,\ell}^{\nu, (\varepsilon)}
\end{equation*}
and note that $|\theta_1^\pm(\xi)-s_{\mu}|\lesssim 2^{-\ell}$ on $\xisupp a_{k,\ell}^{*,\mu, (\varepsilon)}$, where $s_{\mu}:=2^{-\ell}\mu$. Of course, by the definition of the sets $\mathfrak{N}_{\ell}(\mu)$,
\begin{equation*}
    a_{k,\ell}^{(\varepsilon)}=\sum_{\mu \in \Z} a_{k,\ell}^{*,\mu, (\varepsilon)} = \sum_{\mu \in \Z} \sum_{\nu \in \mathfrak{N}_{\ell}(\mu)} a_{k,\ell}^{\nu, (\varepsilon)}.
\end{equation*}
It is notationally convenient to trivially extend these definitions by setting $\mathfrak{N}_{\ell}(\mu) := \{\mu\}$ for $\floor{k/3}_{\,\varepsilon} \leq \ell \leq \floor{k/3}$ and, in this case, defining $a_{k,\ell}^{*,\mu, (\varepsilon)} := a_{k,\ell}^{\mu, (\varepsilon)}$ accordingly.\medskip


Given $0 < r \leq 1$ and $s \in I$, recall the definition of the $(1,r)$-\textit{Frenet boxes} $\pi_{1}(s;\,r)$ introduced in Definition~\ref{def Frenet box}:
\begin{equation*}
    \pi_1(s;\,r):= \big\{ \xi \in \widehat{\R}^3: |\inn{\be_j(s)}{\xi}| \lesssim r^{3-j} \,\, \textrm{for $j=1,\,2$}, \quad 
    |\inn{\be_{3}(s)}{\xi}| \sim 1\big\}. 
\end{equation*}
It is also convenient to consider 2-parameter variants of the $(0,r)$-Frenet boxes. Given $0 < r_1, r_2$ and $s \in I$, define the set
\begin{equation*}
     \pi_0(s;\,r_1, r_2)\!:=\! \big\{\xi \in \widehat{\R}^3: |\inn{\be_1(s)}{\xi}| \lesssim r_1, \,\,  |\inn{\be_2(s)}{\xi}| \sim 1, \,\,
    |\inn{\be_3(s)}{\xi}| \lesssim r_2\big\}.
\end{equation*}
The geometric significance of these sets is made apparent in \S\ref{f freq loc subsec} (and, in particular, Lemma~\ref{freq resc lem}) below. 

 The multipliers $a_{k,\ell}^{*, \mu, (\varepsilon)}$ and $a_{k,\ell}^{\nu, (\varepsilon)}$ satisfy the following support properties.
 
\begin{lemma}\label{J=3 supp lem} For all $0 \leq \ell \leq \floor{k/3}$, $\varepsilon > 0$ and $\mu, \nu \in \Z$, 
\begin{enumerate}[a)]
    \item If $\nu \in \mathfrak{N}_{\ell}(\mu)$, then $\xisupp a_{k,\ell}^{\nu, (\varepsilon)} \subseteq 2^k \cdot \pi_1(s_{\mu}; 2^{-\ell})$, where $s_{\mu} := 2^{-\ell} \mu$;
\item If $\ell < \floor{k/3}_{\,\varepsilon}$, then $\xisupp a_{k,\ell}^{\nu, (\varepsilon)} \subseteq 2^{k-\ell} \cdot \pi_0(s_{\nu}; 2^{-(k- \ell)/2}, 2^{\ell})$, where $s_{\nu} := 2^{-(k-\ell)/2 } \nu$. 
\end{enumerate}
\end{lemma}

As an immediate consequence of part a), we see that $\xisupp a_{k,\ell}^{*, \mu, (\varepsilon)} \subseteq 2^k \cdot \pi_1(s_{\mu}; 2^{-\ell})$.

\begin{proof}[Proof of Lemma~\ref{J=3 supp lem}]

\noindent a) For $\xi \in \xisupp a_{k,\ell}^{\nu, (\varepsilon)}$ observe that the localisation in \eqref{J=3 akell def} implies
\begin{equation*}
    |\inn{\gamma^{(i)}\circ \theta_2(\xi)}{\xi}| \lesssim  2^{k-(3-i)\ell} \qquad \textrm{for $i = 1$, $2$,} \qquad  |\inn{\gamma^{(3)}\circ \theta_2(\xi)}{\xi}| \sim 2^k.
\end{equation*}
If $0 \leq \ell < \floor{k/3}_{\,\varepsilon}$, then $|s_{\nu} - \theta_1(\xi)| \lesssim 2^{-(k-\ell)/2}$ and so
\begin{equation*}
 |s_{\mu} - \theta_2(\xi)| \leq |s_{\mu} - s_{\nu}| + |s_{\nu} - \theta_1(\xi)| + |\theta_1(\xi) - \theta_2(\xi)| \lesssim 2^{-(k-\ell)/2} + 2^{-\ell} \lesssim 2^{-\ell}
\end{equation*} 
by Lemma~\ref{root control lem}. Note that the inequality $|s_{\mu} - \theta_2(\xi)| \lesssim 2^{-\ell}$ also extends to the case $\floor{k/3}_{\,\varepsilon} \leq \ell \leq \floor{k/3}$ in view of the definition of the symbol from \eqref{akellnu dec 2}. Taylor expansion around $\theta_2(\xi)$ therefore yields 
\begin{equation*}
    |\inn{\gamma^{(i)}(s_{\mu})}{\xi}| \lesssim  2^{k-(3-i)\ell} \qquad \textrm{for $i = 1$, $2$,} \qquad  |\inn{\gamma^{(3)}(s_{\mu})}{\xi}| \sim 2^k.
\end{equation*}

Since the Frenet vectors $\be_i(s_{\mu})$ are obtained from the $\gamma^{(i)}(s_{\mu})$ via the Gram--Schmidt process, the matrix corresponding to change of basis from $\big(\be_i(s_{\mu})\big)_{i=1}^3$ to $\big(\gamma^{(i)}(s_{\mu})\big)_{i=1}^3$ is lower triangular. Furthermore, the initial localisation implies that this matrix is an $O(\delta_0)$ perturbation of the identity. Consequently,
\begin{equation*}
   |\inn{\be_i(s_{\mu})}{\xi}| \lesssim 2^{k - (3 - i)\ell} \qquad \textrm{for $1 \leq i \leq 3$}.
   \end{equation*}
  Provided the parameter $\delta_0 > 0$ is sufficiently small, the argument can easily be adapted to prove the remaining lower bound $|\inn{\be_3(s_\mu)}{\xi}|\gtrsim 1$.\medskip

\noindent 
b) Let $0 \leq \ell < \floor{k/3}_{\,\varepsilon}$. For $\xi \in \xisupp a_{k,\ell}^{\nu, (\varepsilon)}$ observe that the localisation in \eqref{J=3 akell def} and Lemma~\ref{root control lem} imply
\begin{equation*}
    |\inn{\gamma'\circ \theta_1(\xi)}{\xi}| = 0, \qquad |\inn{\gamma''\circ \theta_1(\xi)}{\xi}| \sim 2^{k-\ell},  \qquad  |\inn{\gamma^{(3)}\circ \theta_1(\xi)}{\xi}| \sim 2^k.
\end{equation*}
It then follows from Taylor expansion around $\theta_1(\xi)$ that
\begin{equation*}
    |\inn{\gamma'(s_{\nu})}{\xi}| \lesssim  2^{(k-\ell)/2}, \quad |\inn{\gamma''(s_{\nu})}{\xi}| \sim  2^{k-\ell} \quad \textrm{and} \quad |\inn{\gamma^{(3)}(s_{\nu})}{\xi}| \sim 2^k,
\end{equation*}
provided $\rho$ is chosen sufficiently small. The $\gamma^{(j)}(s_{\nu})$ in the above estimates can then be replaced with the Frenet vectors $\be_j(s_{\nu})$ by a similar argument to that used in part a). 
\end{proof}




\subsection{Spatio-temporal Fourier localisation}\label{spatio temp subsec}
The symbols are further localised with respect to the Fourier transform of the $t$-variable. In particular, let
\begin{equation*}
    q(\xi) := \inn{\gamma\circ\theta_2(\xi)}{\xi} \qquad \textrm{and} \qquad \chi_{k,\ell}^{(\varepsilon)}(\xi, \tau) := \eta\big(2^{-(k-3\ell)-4\varepsilon k} (\tau + q(\xi))\big)
\end{equation*}
and define the multiplier $m_{k,\ell}^{\nu, (\varepsilon)}$ by
\begin{equation*}
    \mathcal{F}_t \big[m_{k,\ell}^{\nu, (\varepsilon)}(\xi;\,\cdot\,)\big](\tau):=    \chi_{k,\ell}^{(\varepsilon)}(\xi, \tau) \, \mathcal{F}_t\big[ m[a_{k,\ell}^{\nu, (\varepsilon)}](\xi; \,\cdot\,) \big](\tau).
\end{equation*}
Here $\mathcal{F}_t$ denotes the Fourier transform acting in the $t$ variable. Define $m_{k,\ell}^{*,\mu,(\varepsilon)}$ and $m_{k,\ell}^{(\varepsilon)}$ accordingly by setting
\begin{equation*}
m_{k,\ell}^{*,\mu,(\varepsilon)} := \sum_{\nu \in \mathfrak{N}_{\ell}(\mu)} m_{k,\ell}^{\nu, (\varepsilon)} \qquad \textrm{and} \qquad  m_{k,\ell}^{(\varepsilon)} := \sum_{\mu \in \Z} m_{k,\ell}^{*,\mu, (\varepsilon)}.
\end{equation*}
The main contribution to  $m[a_{k,\ell}^{\nu, (\varepsilon)}]$ comes from the multipliers $m_{k,\ell}^{\nu, (\varepsilon)}$.

\begin{lemma}\label{J=3 spatio temp loc lem}
Let $1 \leq p \leq \infty$ and $\varepsilon>0$.
For all $0 \leq \ell \leq \floor{k/3}$,
    \begin{equation*}
        \big\|\big(m[a_{k,\ell}^{\nu, (\varepsilon)}] - m_{k,\ell}^{\nu, (\varepsilon)}\big)(D;  \,\cdot\,) f \big\|_{L^p(\R^{3+1})}  \lesssim_{N,\varepsilon} 2^{-kN} \| f \|_{L^p(\R^3)} \qquad \textrm{for all $N \in \N$.}
    \end{equation*}
\end{lemma}

\begin{proof}
It suffices to show that
\begin{equation}\label{spatio temporal 1}
    |\partial_\xi^\alpha \big( m[a_{k,\ell}^{\nu, (\varepsilon)}] - m_{k,\ell}^{\nu, (\varepsilon)} \big) (\xi;t)| \lesssim_{N,\varepsilon} 2^{-kN}(1+|t|)^{-10} \qquad \text{ for $\alpha \in \N_0^3$, $\,\,|\alpha|\leq 10,\,\,$ and $\,N \in \N$. }
\end{equation}
Indeed, if \eqref{spatio temporal 1} holds, then Fourier inversion and repeated integration-by-parts imply
\begin{equation*}
    |\big( m[a_{k,\ell}^{\nu, (\varepsilon)}] - m_{k,\ell}^{\nu, (\varepsilon)} \big)(D;t) f(x)| \lesssim_{N,\varepsilon} 2^{-kN} (1+|t|)^{-10} (1+|\cdot|)^{-10} \ast f (x).
\end{equation*}
Taking the $L^p(\R^{3+1})$-norm of both sides of this inequality immediately yields the desired result.

By the Fourier inversion formula
\begin{equation*}
    \big( m[a_{k,\ell}^{\nu, (\varepsilon)}] - m_{k,\ell}^{\nu, (\varepsilon)} \big)(\xi;t) = \frac{1}{2\pi} \int_\R e^{i t \tau} \big(1-\chi_{k,\ell}^{(\varepsilon)}(\xi, \tau) \big) \, \mathcal{F}_t \big[ m[a_{k,\ell}^{\nu, (\varepsilon)}](\xi;\,\cdot\,)\big](\tau) \, \ud \tau.
\end{equation*}
Let $\Xi = (\xi, \tau) \in \widehat{\R}^{3+1}$ denote the spatio-temporal frequency variables. Clearly, there exists a constant $C\geq 1$ such that $| \partial_\Xi^{\alpha}\, \chi_{k,\ell}^{(\varepsilon)}(\Xi) | \lesssim 2^{Ck}$ for all $\alpha \in \N_0^4$ with $|\alpha|\leq 20$. Furthermore, if $(\xi, \tau) \in \supp  \partial_\Xi^{\alpha}\, \big(1-\chi_{k,\ell}^{(\varepsilon)}\big)$, then $|\tau + q(\xi)| \gtrsim 2^{-k+3\ell +4 \varepsilon k}$. Thus, by integration-by-parts in the $\tau$-variable, to prove \eqref{spatio temporal 1} it suffices to show
\begin{equation}\label{spatio temporal 2}
| \partial_\Xi^{\alpha}\, \mathcal{F} \big[ m[a_{k,\ell}^{\nu, (\varepsilon)}](\xi;\,\cdot\,)\big](\tau)| \lesssim_{N,\varepsilon} 2^{Ck} \big(1 + 2^{-k + 3\ell + 3\varepsilon k}|\tau + q(\xi)|\big)^{-N}, \quad \alpha \in \N_0^4,\,\, |\alpha|\leq 20, \,\, N \in \N,
\end{equation}
for some choice of absolute constant $C \geq 1$ (not necessarily the same as above).

By the Leibniz rule, 
\begin{equation}\label{spatio temporal 3}
    \partial_\Xi^{\alpha}\, \mathcal{F}_t \big[ m[a_{k,\ell}^{\nu, (\varepsilon)}](\xi;\,\cdot\,)\big](\tau)=\int_\R e^{-i r (\tau + q(\xi))}  m[b_{k,\ell}^{\nu, (\varepsilon),\alpha}] (\xi; r)\, \ud r
\end{equation}
where $b_{k,\ell}^{\nu, (\varepsilon),\alpha}(\xi;r;s) := e^{i r q(\xi)}a_{k,\ell}^{\nu, (\varepsilon),\alpha}(\xi;r;s)$  for some symbol $a_{k,\ell}^{\nu, (\varepsilon), \alpha}$ satisfying 
\begin{equation}\label{spatio temporal 9}
    \big|\partial_r^j\, a_{k,\ell}^{\nu, (\varepsilon), \alpha}(\xi;r;s)\big|\lesssim_j  2^{Ck} \qquad \text{ for all $j \in \N_0$, $\alpha \in \N_0^4$, $|\alpha| \leq 20$,  $|r| \lesssim 1$}
\end{equation}
and with $\supp a_{k,\ell}^{\nu,(\varepsilon), \alpha} \subseteq \supp a_{k,\ell}^{\nu,(\varepsilon)}$.
Note, in particular, that 
\begin{equation}\label{spatio temporal 4}
    m[b_{k,\ell}^{\nu, (\varepsilon),\alpha}](\xi;r) = \int_\R e^{-i r \inn{\gamma(s) - \gamma\, \circ\, \theta_2 (\xi)}{\xi}} a_{k,\ell}^{\nu, (\varepsilon), \alpha} (\xi;r;s) \rho(r) \chi(s)\, \ud s.
\end{equation}

By Taylor expansion around $\theta_2(\xi)$, the phase in \eqref{spatio temporal 4} can be written as
\begin{equation}\label{spatio temporal 6}
\inn{\gamma(s) - \gamma \circ \theta_2 (\xi)}{\xi} = u(\xi) \, (s-\theta_2(\xi)) + \omega(\xi;s) \, (s-\theta_2(\xi))^3
\end{equation}
where $\omega$ arises from the remainder term and satisfies $|\omega(\xi;s)|\sim 2^k$. Recall,
\begin{equation}\label{spatio temporal 7}
    |u(\xi)|\lesssim 2^{k-2\ell} \qquad \text{ and } \qquad |s-\theta_2(\xi)|\lesssim 2^{-\ell +  \varepsilon k} \qquad \text{ for all $\,\,(\xi;r;s) \in \supp a_{k,\ell}^{\nu, (\varepsilon)}$,}
\end{equation}
which follows from the definition of $a_{k,\ell}^{\nu, (\varepsilon)}$. Here, in the case $0 \leq \ell < \floor{k/3}_{\varepsilon}$, we use Lemma~\ref{root control lem} to deduce that
\begin{equation*}
    |s-\theta_2(\xi)|\leq  |s-\theta_1(\xi)| + |\theta_1(\xi) - \theta_2(\xi)| \lesssim 2^{-\ell}.
\end{equation*}
Combining the expansion \eqref{spatio temporal 6} and the localisation \eqref{spatio temporal 7} yields
\begin{equation}\label{spatio temporal 8}
    |\inn{\gamma(s) - \gamma \circ \theta_2 (\xi)}{\xi}| \lesssim 2^{k-3\ell +3\varepsilon k} \qquad \text{for all $(\xi;r;s) \in \supp a_{k,\ell}^{\nu, (\varepsilon)}$.}
\end{equation}

By \eqref{spatio temporal 8}, \eqref{spatio temporal 9} and integration by parts in \eqref{spatio temporal 3}, one obtains
\begin{equation*}
    |\partial_\Xi^{\alpha}\, \mathcal{F}_t \big[ m[a_{k,\ell}^{\nu, (\varepsilon)}](\xi;\,\cdot\,)\big](\tau)| \lesssim_{M} 2^{C k} |\tau + q(\xi)|^{-M} 2^{(k-3\ell +3\varepsilon k)M}  \qquad \text{for all $M \in \N$}
\end{equation*}
and all $\alpha \in \N_0^4$, $|\alpha|\leq 20$. This implies \eqref{spatio temporal 2} and concludes the proof.
\end{proof}


To understand the support properties of the multipliers $m_{k,\ell}^{*,\mu,(\varepsilon)}$, we introduce the primitive curve
\begin{equation*}
   \bar{\gamma} \colon I \to \R^4, \qquad \bar{\gamma} \colon s \mapsto \begin{bmatrix}
    \int_0^s \gamma \\
    s
    \end{bmatrix}.
\end{equation*}
Here $\int_0^s \gamma$ denotes the vector in $\R^3$ with $i$th component $\int_0^s \gamma_i$ for $1 \leq i \leq 3$. Note that $\bar{\gamma}$ is a non-degenerate curve in $\R^4$ and, in particular, $|\det(\bar{\gamma}^{(1)} \cdots \bar{\gamma}^{(4)})| = |\det(\gamma^{(1)}\cdots \gamma^{(3)})|$. Let $(\bar{\be}_j(s))_{j=1}^4$ denote the Frenet frame associated to $\bar{\gamma}$ and consider the $(2,r)$-Frenet boxes for $\bar \gamma$
\begin{equation*}
    \pi_{2,\bar{\gamma}}(s;\,r) := \big\{\Xi = (\xi, \tau) \in \widehat{\R}^3 \times \widehat{\R} : |\inn{\bar{\be}_j(s)}{\Xi}| \lesssim r^{4 - j} \textrm{ for $1 \leq j \leq 3$, } |\inn{\bar{\be}_4(s)}{\Xi}| \sim 1 \big\},
\end{equation*}
as introduced in Definition~\ref{def Frenet box}. 

\begin{lemma}\label{J=3 4d supp lem} For all $\ceil{4 \varepsilon k} \leq \ell \leq \floor{k/3}$ and $\mu \in \Z$, 
    \begin{equation*}
        \supp \mathcal{F}_t \big[m_{k,\ell}^{\ast,\mu,(\varepsilon)}\big] \subseteq 2^k \cdot \pi_{2,\bar{\gamma}}(s_{\mu}; 2^{4\varepsilon k}2^{-\ell}),
    \end{equation*} 
     where $s_{\mu}:=2^{-\ell }\mu$ and $\mathcal{F}_t$ denotes the Fourier transform in the $t$-variable.
\end{lemma}

\begin{proof} If $\Xi = (\xi, \tau) \in \supp \mathcal{F}_t \big[m_{k,\ell}^{\ast,\mu,(\varepsilon)}\big]$, then $\xi \in \xisupp a_{k,\ell}^{\ast,\mu,(\varepsilon)}$ and $|q(\xi) + \tau| \lesssim 2^{4\varepsilon k} 2^{k-3\ell}$. The former condition implies $|u(\xi)| \lesssim 2^{k-2\ell}$ and $|s - \theta_2(\xi)| \lesssim  2^{-\ell + \varepsilon k}$ (see \eqref{spatio temporal 7}) and so, by Taylor expansion around $\theta_2(\xi)$,
\begin{equation}\label{J=3 4d supp 1}
|\inn{\gamma(s_{\mu})}{\xi} + \tau| \lesssim |q(\xi) + \tau| + |u(\xi)||s-\theta_2(\xi)| + 2^k|s-\theta_2(\xi)|^3 \lesssim 2^{4\varepsilon k} 2^{k-3\ell}.
\end{equation}

Define the lifted curve and frame
\begin{equation*}
    \gamma_{\uparrow} \colon I \to \R^4, \quad \gamma_{\uparrow} \colon s \mapsto \begin{bmatrix}
    \gamma(s) \\
    1
    \end{bmatrix}
    \quad \textrm{and} \quad \be_{j,\uparrow} \colon I \to S^3, \quad \be_{j, \uparrow} \colon s \mapsto \begin{bmatrix}
    \be_j(s) \\
    0
    \end{bmatrix} 
    \quad \textrm{for $1 \leq j \leq 3$,}
\end{equation*}
respectively. This definition is motivated by our related
 work on $L^p$ Sobolev estimates for the moment curve in four  dimensions \cite{BGHS-Sobolev}.
Note that $\bar{\gamma}$ is a primitive for $\gamma_{\uparrow}$ in the sense that $\bar{\gamma}' = \gamma_{\uparrow}$. By the definition of the Frenet frame, it follows that
\begin{equation*}
    \bar{\be}_j(s) \in \langle \gamma_{\uparrow}(s), \gamma_{\uparrow}'(s), \dots, \gamma_{\uparrow}^{(j-1)}(s) \rangle \qquad \textrm{and} \qquad \gamma_{\uparrow}^{(i)}(s) \in  \langle \be_{1, \uparrow}(s), \dots, \be_{i,\uparrow}(s) \rangle
\end{equation*}
for $1 \leq i < j \leq 4$. Thus, one readily deduces that
\begin{equation*}
    |\inn{\bar{\be}_j(s)}{\Xi}| \lesssim |\inn{\gamma_{\uparrow}(s)}{\Xi}| + \sum_{i = 1}^{j-1} |\inn{\be_i(s)}{\xi}| \qquad \textrm{for $\, \Xi = (\xi, \tau) \in \widehat{\R}^{3+1}$ and $1 \leq j \leq 4$.} 
\end{equation*}
If $\Xi = (\xi, \tau) \in \supp \mathcal{F}_t \big[m_{k,\ell}^{\ast,\mu,(\varepsilon)}\big]$, then it follows from \eqref{J=3 4d supp 1} that
\begin{equation*}
   |\inn{\gamma_{\uparrow}(s_{\mu})}{\Xi}| = | \inn{\gamma(s_{\mu})}{\xi}+\tau| \lesssim 2^{4\varepsilon k} 2^{k-3\ell}.
\end{equation*}
On the other hand, since $\xi \in \xisupp  a_{k,\ell}^{\ast,\mu,(\varepsilon)}$, Lemma~\ref{J=3 supp lem} yields
\begin{equation*}
   |\inn{\be_i(s_{\mu})}{\xi}| \lesssim 2^{k - (3-i)\ell} \qquad \textrm{for $i = 1$, $2$, } \qquad |\inn{\be_3(s_{\mu})}{\xi}| \sim 2^k.
\end{equation*}
Combining these observations, $|\inn{\bar{\be}_j(s)}{\Xi}| \lesssim 2^{4\varepsilon k}2^{k-(4-j)\ell}$ for $1 \leq j \leq 3$ and therefore it suffices to prove $|\inn{\bar{\be}_4(s_{\mu})}{\Xi}| \sim 2^k$. Since our hypothesis $\ell \geq \ceil{4\varepsilon k}$ implies that $2^{4\varepsilon k}2^{k-(3-i)\ell} \leq 2^k$ for $0 \leq i \leq 2$, the above argument directly yields the upper bound. On the other hand, since $\gamma \in \mathfrak{G}_3(\delta_0)$ and we are localised to $|s_{\mu}| \lesssim \delta_0$, the change of basis mapping $(\bar{\be}_j(s_{\mu}))_{j=1}^4$ to $(\gamma_{\uparrow}^{(j-1)}(s_{\mu}))_{j=1}^4$ is an $O(\delta_0)$ perturbation of the identity. In view of this, the above argument can also be adapted to give the required lower bound. 
\end{proof}




\subsection{Reverse square function estimates in \texorpdfstring{$\R^{3+1}$}{}} In view of the Fourier localisation described in the previous subsection, Theorem~\ref{Frenet reverse SF theorem} implies the following square function estimate.

\begin{proposition}\label{J=3 rev sf prop}
Let $k \in \N$, $0 \leq \ell \leq \floor{k/3}$. For all $2 \leq p \leq 4$ and $\varepsilon>0$, one has
\begin{equation*}
    \| m[a_{k,\ell}^{(\varepsilon)}](D;\, \cdot \,) f \|_{L^p(\R^{3+1})} \lesssim_{\varepsilon, N} 2^{(k-3\ell)/4} 2^{O(\varepsilon k)} \Big\| \big(\sum_{\nu \in \Z} |m[a_{k,\ell}^{\nu,(\varepsilon)}](D; \,\cdot\,) f|^2 \big)^{1/2} \Big\|_{L^p(\R^{3+1})} + 2^{-kN} \|f\|_{L^p(\R^3)}.
\end{equation*}
\end{proposition}

\begin{proof} First suppose $\ceil{4 \varepsilon k} \leq \ell$ so that Lemma~\ref{J=3 4d supp lem} applies. Thus,
\begin{equation*}
    m_{k,\ell}^{(\varepsilon)}(D;\,\cdot\,)f = \sum_{\mu \in \Z} m_{k,\ell}^{*,\mu, (\varepsilon)}(D;\,\cdot\,)f
\end{equation*}
where each $m_{k,\ell}^{*,\mu, (\varepsilon)}(D;\,\cdot\,)f$ has spatio-temporal Fourier support in $2^k\cdot \pi_{2,\bar{\gamma}}(s_{\mu}; 2^{4\varepsilon k}2^{-\ell})$. The family of sets $\pi_{2,\bar{\gamma}}(s_{\mu}; 2^{4\varepsilon k}2^{-\ell})$ for $|\mu| \leq 2^{\ell}$ may be partitioned into $O(2^{4\varepsilon k})$ subfamilies, each forming a $(2, 2^{4\varepsilon k}2^{-\ell})$-Frenet box decomposition for the non-degenerate curve $\bar{\gamma}$ in $\R^4$.  Consequently, by Theorem~\ref{Frenet reverse SF theorem} and pigeonholing, 
\begin{equation*}
    \| m_{k,\ell}^{(\varepsilon)}(D; \,\cdot\,) f \|_{L^p(\R^4)} \lesssim_{\varepsilon} 2^{O(\varepsilon k)}  \Big\| \big(\sum_{\mu \in \Z} |m_{k,\ell}^{*, \mu,(\varepsilon)}(D; \,\cdot\,) f|^2 \big)^{1/2} \Big\|_{L^p(\R^4)}.
\end{equation*}
By a pointwise application of the Cauchy--Schwarz inequality, using the fact that $\# \mathfrak{N}_{\ell}(\mu) \lesssim 2^{(k-3\ell)/2}$ for all $\mu \in \Z$, we conclude that 
\begin{equation}\label{J=3 reverse sf 1}
    \| m_{k,\ell}^{(\varepsilon)}(D; \,\cdot\,) f \|_{L^p(\R^{3+1})} \lesssim_{\varepsilon} 2^{(k-3\ell)/4}2^{O(\varepsilon k)}  \Big\| \big(\sum_{\nu \in \Z} |m_{k,\ell}^{\nu,(\varepsilon)}(D; \,\cdot\,) f|^2 \big)^{1/2} \Big\|_{L^p(\R^{3+1})}.
\end{equation}
The desired estimate, involving the $m[a_{k,\ell}^{\nu,(\varepsilon)}]$ multipliers rather than the $m_{k,\ell}^{\nu,(\varepsilon)}$, now follows by combining \eqref{J=3 reverse sf 1} and Lemma~\ref{J=3 spatio temp loc lem}. 

On the other hand, if $0 \leq \ell \leq \ceil{4 \varepsilon k}$, then the result follows directly from the Cauchy--Schwarz inequality. 
\end{proof}

\begin{remark}
The above square function estimate is not very effective \textit{away} from the binormal cone ($\ell=0$ or small values of $\ell$), as in that case it essentially amounts to a trivial application of the Cauchy--Schwarz inequality. However, as noted in \S\ref{LS rel G sec}, Proposition~\ref{J=3 LS prop} is only used \textit{close} to the binormal cone ($\ell=\floor{k/3}$ or large values of $\ell$), for which Proposition~\ref{J=3 rev sf prop} is most effective. The small values of $\ell$ in proving Theorem~\ref{LS thm} are handled via Proposition~\ref{PS LS prop}.
\end{remark}

For $p = 2$ a stronger square function estimate is available simply due to Plancherel's theorem. In particular, this avoids the loss induced by the Cauchy--Schwarz inequality in the proof above.

\begin{lemma}\label{J=3 L2 rev sf prop}
Let $k \in \N$, $0 \leq \ell \leq \floor{k/3}$. For all $\varepsilon>0$, 
\begin{equation*}
    \| m[a_{k,\ell}^{(\varepsilon)}](D;\, \cdot \,) f \|_{L^2(\R^{3+1})} \lesssim_{\varepsilon}   \Big\| \big(\sum_{\nu \in \Z} |m[a_{k,\ell}^{\nu,(\varepsilon)}](D; \,\cdot\,) f|^2 \big)^{1/2} \Big\|_{L^2(\R^{3+1})}.
\end{equation*}
\end{lemma}

\begin{proof} This is simply a consequence of Plancherel's theorem and the fact that the symbols $a_{k,\ell}^{\nu,(\varepsilon)}$ are supported on the essentially disjoint sets $2^{k-\ell} \cdot \pi_0(s_{\nu}; 2^{-(k-\ell)/2}, 2^{\ell})$ by Lemma~\ref{J=3 supp lem}. 
\end{proof}




\subsection{Kernel estimates}

Given a symbol $a \in C^{\infty}(\widehat{\R}^n\setminus\{0\} \times \R \times \R)$, define the associated convolution kernel 
\begin{equation*}
    K[a](x,t) := \frac{1}{(2\pi)^n}\int_{\widehat{\R}^n} e^{i \inn{x}{\xi}} m[a](\xi;t)\,\ud \xi. 
\end{equation*}

Each of the localised symbols $a_{k,\ell}^{\nu, (\varepsilon)}$ satisfies the following kernel estimate, which yields a gain due to the localisation of the symbols in the $s$-variable introduced in \eqref{akell dec}.

\begin{lemma}\label{J=3 ker lem} For $k \in \N$ and $0 \leq \ell \leq \floor{k/3}$,
\begin{equation*}
    |K[a_{k,\ell}^{\nu, (\varepsilon)}](x,t)| \lesssim 2^{-(k-\ell)/2}2^{O(\varepsilon k)}  \,  \psi_{\,\mathcal{T}_{k,\ell}(s_{\nu})}(x,t)\, \rho(t)
\end{equation*}
where
\begin{equation*}
   \psi_{\,\mathcal{T}_{k,\ell}(s_{\nu})}(x,t) := 2^{(5k-3\ell)/2} \bigg(1 + \sum_{j=1}^3 2^{j(k-\ell)/2 \wedge k}|\inn{x - t \gamma(s_{\nu})}{\be_j(s_{\nu})}|\bigg)^{-100}.
\end{equation*} 
\end{lemma}

\begin{proof} Let $\nabla_{\bm{v}_j}$ denote the directional derivative with respect to the $\xi$ variable in the direction of the vector $\bm{v}_j := \be_j(s_\nu)$, so that
\begin{equation*}
    \Big(\frac{1}{i \inn{x - t \gamma(s_{\nu})}{\be_j(s_\nu)}}\nabla_{\bm{v_j}} - 1 \Big) e^{i\inn{x - t \gamma(s_{\nu})}{\xi}} = 0.
\end{equation*}
Thus, by repeated integration-by-parts, it follows that
\begin{align*}
    |K[a_{k,\ell}^{\nu, (\varepsilon)}](x,t)| &\leq |\inn{x - t \gamma(s_{\nu})}{\be_j(s_{\nu})}|^{-N} \int_{\widehat{\R}^3} \big|\nabla_{\bm{v}_j}^N\big[ e^{it\inn{\gamma(s_{\nu})}{\xi}}m[a_{k,\ell}^{\nu, (\varepsilon)}](\xi;t)\big]\big|\,\ud \xi \\
    &\lesssim  2^{(5k - 3\ell)/2} 2^{O(\varepsilon k)}|\inn{x - t \gamma(s_{\nu})}{\be_j(s_{\nu})}|^{-N} \sup_{\xi \in \widehat{\R}^3}\big|\nabla_{\bm{v}_j}^N\big[ e^{it\inn{\gamma(s_{\nu})}{\xi}}m[a_{k,\ell}^{\nu, (\varepsilon)}](\xi;t)\big]\big|;
\end{align*}
here the second inequality follows from the $\xi$-support properties of the symbols $a_{k,\ell}^{\nu, (\varepsilon)}$ from Lemma~\ref{J=3 supp lem} b) if $0 \leq \ell < \floor{k/3}_{\,\varepsilon}$ (in which case there is no $2^{O(\varepsilon k)}$ loss) or Lemma~\ref{J=3 supp lem} a) if $\floor{k/3}_{\,\varepsilon} \leq \ell \leq \floor{k/3}$. Observe that 
\begin{equation*}
 e^{it\inn{\gamma(s_{\nu})}{\xi}} m[a_{k,\ell}^{\nu, (\varepsilon)}](\xi;t) = \int_{\R} e^{-it \inn{\gamma(s) - \gamma(s_{\nu})}{\xi}} a_{k,\ell}^{\nu, (\varepsilon)}(\xi;s) \chi(s)\rho(t)\,\ud s.
\end{equation*}
Passing the differential operator $\nabla_{\bm{v}_j}$ into the $s$-integral, we therefore have
\begin{equation}\label{J3 ker 0}
\big|\nabla_{\bm{v}_j}^N\big[ e^{it\inn{\gamma(s_{\nu})}{\xi}}m[a_{k,\ell}^{\nu, (\varepsilon)}](\xi;t)\big]\big| \lesssim 2^{-(k-\ell)/2} 2^{O(\varepsilon k)} \sup_{s \in \R} \big|\nabla_{\bm{v}_j}^N\big[ e^{-it \inn{\gamma(s) - \gamma(s_{\nu})}{\xi}} a_{k,\ell}^{\nu, (\varepsilon)}(\xi;t;s)\big]\big| \, \rho(t).
\end{equation}
Here we have used the $s$-support properties of $a_{k,\ell}^{\nu, (\varepsilon)}$; in particular, the definition \eqref{akell dec}. 
\medskip

Consider the oscillatory factor $e^{-it \inn{\gamma(s) - \gamma(s_{\nu})}{\xi}}$ on the right-hand side of \eqref{J3 ker 0}. The $\xi$ derivatives of this function can be controlled on $\supp a_{k,\ell}^{\nu, (\varepsilon)}$ by noting that
\begin{equation*}
    |\inn{\gamma(s) - \gamma(s_{\nu})}{\bm{v}_j}| \leq \int_{s_{\nu}}^s |\inn{\gamma'(\sigma)}{\bm{v}_j}|\,\ud \sigma \lesssim |s - s_{\nu}|^j \lesssim 2^{-j(k-\ell)/2} 2^{O(\varepsilon k)} \qquad \textrm{for $1 \leq j \leq 3$,}
\end{equation*}
where we have used \eqref{Frenet bound alt 1} and triangle inequality and \eqref{akellnu dec 1} in the last inequality. Thus, by the Leibniz rule, the problem is reduced to showing
\begin{equation}\label{J3 ker 1}
   |\nabla_{\bm{v}_j}^{N} a_{k,\ell}^{\nu, (\varepsilon)}(\xi;t;s)| \lesssim_N 2^{-(j(k-\ell)/2 \wedge k)N} 2^{\varepsilon \ell N} \qquad \textrm{for all $1 \leq j \leq 3$ and all $N \in \N$.} 
\end{equation}

For all $N \in \N$, we claim the following: 
\begin{itemize}
    \item For $\xi \in \xisupp a_{k, \ell}^{\nu, (\varepsilon)}$ with $0 \leq \ell \leq \floor{k/3}$,
\begin{equation}\label{J3 ker 2a}
    2^{\ell}|\nabla_{\bm{v}_j}^N \theta_2(\xi)|, \quad  2^{-k+2\ell}|\nabla_{\bm{v}_j}^N u(\xi)|\;\; \lesssim_N \;\; 2^{-(j(k-\ell)/2 \wedge k)N} 2^{\varepsilon \ell N};
\end{equation}
    \item For $\xi \in \xisupp a_{k, \ell}^{\nu, (\varepsilon)}$ with $0 \leq \ell < \floor{k/3}_{\,\varepsilon}$,
\begin{equation}\label{J3 ker 2b}
    2^{(k-\ell)/2}|\nabla_{\bm{v}_j}^N \theta_1(\xi)| \lesssim_N 2^{-(j(k-\ell)/2 \wedge k)N}.
\end{equation}

\end{itemize}
Assuming that this is so, the derivative bounds \eqref{J3 ker 1} follow directly from the chain and Leibniz rule, applying \eqref{J3 ker 2a} and \eqref{J3 ker 2b}.\medskip

The claimed bound \eqref{J3 ker 2a} follows from repeated application of the chain rule, provided 
\begin{subequations}
\begin{align}\label{J3 ker 3a}
    |\inn{\gamma^{(3)}\circ \theta_2(\xi)}{\xi}| &\gtrsim 2^k, \\
    \label{J3 ker 3b}
    |\inn{\gamma^{(K)}\circ \theta_2(\xi)}{\xi}| &\lesssim_K 2^{k +\ell(K-3)},  \\
    \label{J3 ker 3c}
    |\inn{\gamma^{(K)}\circ \theta_2(\xi)}{\bm{v}_j}| &\lesssim_K 2^{-(j(k-\ell)/2 \wedge k) + k +\ell(K-3)}2^{\varepsilon \ell}
\end{align}
\end{subequations}
hold for all $K \geq 2$ and all $\xi \in \xisupp a_{k,\ell}^{\nu,(\varepsilon)}$. In particular, assuming \eqref{J3 ker 3a}, \eqref{J3 ker 3b} and \eqref{J3 ker 3c}, the bounds in \eqref{J3 ker 2a} are then a consequence of Lemma~\ref{imp deriv lem} in the appendix: \eqref{J3 ker 2a} corresponds to \eqref{multi imp der bound}
and \eqref{multi Faa di Bruno eq 2} 
whilst the hypotheses in the above display correspond to \eqref{multi imp deriv 2}
and \eqref{multi imp deriv 3} (see Example~\ref{deriv ex}). Here the parameters featured in the appendix are chosen as follows: 
\begin{center}
    \begin{tabular}{|c|c|c|c|c|c|c|} 
\hline
 & & & & & & \\[-0.8em]
$g$ & $h$ & $A$ & $B$ & $M_1$ & $M_2$ & $\be$ \\
 & & & & & & \\[-0.8em]
\hline
  & & & & & & \\[-0.8em]
$\gamma''$ & $\gamma'$ & 
$2^{k-\ell}$ & $2^{k-2\ell}$ & $2^{-(j(k-\ell)/2 \wedge k)} 2^{\varepsilon \ell}$  &  $2^{\ell}$  & $\bm{v}_j$ \\
 & & & & & & \\[-0.8em]
 \hline
\end{tabular}
\end{center}

The conditions \eqref{J3 ker 3a}, \eqref{J3 ker 3b} and \eqref{J3 ker 3c} are direct consequences of the support properties of the $a_{k,\ell}^{\nu, (\varepsilon)}$. 
Indeed, \eqref{J3 ker 3a} and the $K \geq 3$ case of \eqref{J3 ker 3b} are trivial consequences of the localisation of the symbol $a_k$. The remaining $K = 2$ case of \eqref{J3 ker 3b} follows immediately since $\inn{\gamma'' \circ \theta_2(\xi)}{\xi}=0$. Finally, the right-hand side of \eqref{J3 ker 3c} is always greater than 1 unless $j = 3$ and $K = 2$, and so we can immediately reduce to this case. If $0 \leq \ell < \floor{k/3}_{\varepsilon}$, then \eqref{Frenet bound alt 1} together with Lemma~\ref{root control lem} and the $\theta_1$ localisation in \eqref{akellnu dec 1} implies
\begin{equation*}
      |\inn{\gamma^{(2)}\circ \theta_2(\xi)}{\bm{v}_3}| \lesssim |\theta_2(\xi) - s_{\nu}| \leq |\theta_2(\xi) - \theta_1(\xi)| + |\theta_1(\xi) - s_{\nu}| \lesssim 2^{-\ell}.
\end{equation*}
On the other hand, if $\floor{k/3}_{\varepsilon} \leq \ell \leq \floor{k/3}$, then, by a similar argument, $|\inn{\gamma^{(2)}\circ \theta_2(\xi)}{\bm{v}_3}| \lesssim 2^{-\ell(1 -\varepsilon)}$. This concludes the proof of \eqref{J3 ker 3c}.\medskip

Similarly, the claimed bound \eqref{J3 ker 2b} follows from repeated application of the chain rule, provided 
\begin{subequations}
\begin{align}\label{J3 ker 4a}
    |\inn{\gamma^{(2)}\circ \theta_1(\xi)}{\xi}| &\gtrsim 2^{k-\ell}, \\
    \label{J3 ker 4b}
    |\inn{\gamma^{(K)}\circ \theta_1(\xi)}{\xi}| &\lesssim_K 2^{K(k-\ell)/2},  \\
    \label{J3 ker 4c}
   |\inn{\gamma^{(K)}\circ \theta_1(\xi)}{\bm{v}_j}| &\lesssim_K 2^{-(j(k-\ell)/2 \wedge k) + K(k-\ell)/2}
\end{align}
\end{subequations}
hold for all $K \geq 2$ and all $\xi \in \supp a_{k,\ell}^{\nu,(\varepsilon)}$ when $0 \leq \ell < \floor{k/3}_{\,\varepsilon}$. This again follows by Lemma~\ref{imp deriv lem} in the appendix. Here the parameters are chosen as follows: 
\begin{center}
    \begin{tabular}{|c|c|c|c|c|c|c|} 
\hline
 & & & & \\[-0.8em]
$g$  & $A$ & $M_1$ & $M_2$ & $\be$ \\
 & & & & \\[-0.8em]
\hline
  & & & & \\[-0.8em]
$\gamma'$  & 
$2^{(k-\ell)/2}$ & $2^{-(j(k-\ell)/2 \wedge k)}$  &  $2^{(k-\ell)/2}$  & $\bm{v}_j$ \\
 & & & & \\[-0.8em]
 \hline
\end{tabular}
\end{center}

The conditions \eqref{J3 ker 4a}, \eqref{J3 ker 4b} and \eqref{J3 ker 4c} are direct consequences of the support properties of the $a_{k,\ell}^{\nu, (\varepsilon)}$ for $0 \leq \ell < \floor{k/3}_{\,\varepsilon}$. Indeed, \eqref{J3 ker 4a} and the $K = 2$ case of \eqref{J3 ker 4b} is just a restatement of the condition $|v^{\pm}(\xi)| \sim 2^{k-\ell}$, which holds due to Lemma~\ref{root control lem}. The $K \geq 3$ case of \eqref{J3 ker 4b} follows immediately from the localisation of the symbols $a_k$. Finally, the right-hand side of \eqref{J3 ker 4c} is always greater than 1 unless $j = 3$ and $K = 2$, and so we can immediately reduce to this case. However, \eqref{Frenet bound alt 1} 
together with the $\theta_1$ localisation in \eqref{akellnu dec 1} implies
\begin{equation*}
      |\inn{\gamma^{(2)}\circ \theta_1(\xi)}{\bm{v}_3}| \lesssim  |\theta_1(\xi) - s_{\nu}|  \lesssim  2^{-(k-\ell)/2}2^{\varepsilon k}  \lesssim 2^{-\ell},
\end{equation*}
which concludes the proof of \eqref{J3 ker 4c}.
\end{proof}




\subsection{Localising the input function}\label{f freq loc subsec} At this juncture it is useful to note some further geometric properties of the support of the multipliers $m[a_{k,\ell}^{\nu, (\varepsilon)}]$ featured in the decomposition.\medskip

Recall from Lemma~\ref{J=3 supp lem} a) that
\begin{equation}\label{rec supp eq a}
  \xisupp a_{k,\ell}^{\nu, (\varepsilon)} \subseteq 2^k \cdot \pi_1(s_{\mu}; 2^{-\ell}) \qquad \textrm{for all $\nu \in \mathfrak{N}_{\ell}(\mu)$,}
\end{equation}
where $s_\mu:=2^{-\ell}\mu$. The right-hand set is contained in a certain sector in the frequency space. In particular, given $0 \leq \ell \leq \floor{k/3}$ and $m \in \Z$ define
\begin{equation}\label{sector def}
\Delta_{k, \ell}(m) := \big\{ \xi \in \widehat{\R}^3 : |\xi_2 - \xi_3 2^{-\ell}m| \leq C 2^{-\ell} \xi_3 \textrm{ and } C^{-1} 2^k \leq \xi_3 \leq C 2^k \big\},
\end{equation}
where $C \geq 1$ is an absolute constant, chosen sufficiently large so as to satisfy the requirements of the forthcoming argument. 

\begin{lemma}\label{sec supp lem} If $\mu$, $\nu \in \Z$ satisfy $\nu \in \mathfrak{N}_{\ell}(\mu)$, then there exists some $ m(\mu) \in \Z$ such that
\begin{equation}\label{f freq claim} 
    2^k \cdot \pi_1(s_{\mu}; 2^{-\ell}) \subseteq \Delta_{k, \ell}\big(m(\mu)\big).
\end{equation}
Furthermore, for each fixed $k$ and $\ell$, given $m \in \Z$ there are only $O(1)$ values of $\mu \in \Z$ such that $m = m(\mu)$.
\end{lemma}

\begin{proof} Define $G \colon I_0 \to \R^3$ by $G(s) := \be_{33}(s)^{-1} \be_3(s)$. As a consequence of the Frenet equations, the vectors $G'(s)$, $G''(s)$ span $\R^2 \times \{0\}$. Given $\xi \in \widehat{\R}^3$, it follows that there exist $\eta_1$, $\eta_2 \in \R$ such that 
\begin{equation}\label{f freq loc 1}
    \xi - \xi_3 G(s) = \sum_{j=1}^2 2^{- \ell j} \eta_j G^{(j)}(s). 
\end{equation}
Taking the inner product of both sides of this identity with respect to the $\be_j(s)$ for $j = 1$, $2$ and applying the Frenet equations
\begin{equation}\label{f freq loc 2}
    \begin{bmatrix}
    \inn{\xi}{\be_1(s)} \\
    \inn{\xi}{\be_2(s)}
    \end{bmatrix}
    =
    \begin{bmatrix}
    0 & \inn{G^{(2)}(s)}{\be_1(s)} \\
    \inn{G^{(1)}(s)}{\be_2(s)} & \inn{G^{(2)}(s)}{\be_2(s)}
    \end{bmatrix}
    \begin{bmatrix}
    2^{-\ell} \eta_1 \\
    2^{-2\ell} \eta_2
    \end{bmatrix}
\end{equation}
where the anti-diagonal entries of the right-hand $2 \times 2$ matrix have size $\sim 1$.\footnote{A similar computation is carried out in more detail in \S\ref{geo obs sec}.}\smallskip

 Let $\xi \in 2^k \cdot \pi_1(s; 2^{-\ell})$ so that $|\inn{\xi}{\be_1(s_{\mu})}| \lesssim 2^{k-2\ell}$ and $|\inn{\xi}{\be_2(s_{\mu})}| \lesssim 2^{k-\ell}$. Combining these bounds with \eqref{f freq loc 1} and \eqref{f freq loc 2}, it follows that 
\begin{equation*}
    |\xi_2 - \xi_3G_2(s_{\mu})| \leq |\xi - \xi_3G(s_{\mu})| \lesssim 2^{k-\ell} \sim 2^{-\ell} \xi_3.
\end{equation*}
If we take $m(\mu)$ to be the integer which minimises $|2^{-\ell}m -G_2(s_{\mu})|$, then we obtain \eqref{f freq claim}. On the other hand, the Frenet equations ensure that $G_2(s) = \be_{33}(s)^{-1}\be_{32}(s)$ satisfies $|G_2'(s)| \sim 1$ for all $s \in I_0$. Consequently, the assignment $\mu \mapsto m(\mu)$ is $O(1)$-to-1, as claimed.  
\end{proof}

For each $\mu \in \Z$ define the smooth cutoff function
\begin{equation*}
    \chi_{k,\ell}^{*,\mu}(\xi) := \eta\big(C^{-1}|2^{\ell}\xi_2/\xi_3 - m(\mu)|\big) \, \big( \eta(C^{-1}2^{-k}\xi_3) - \eta(2C2^k \xi_3)\big).
\end{equation*}
If $\xi \in \xisupp a_{k,\ell}^{\nu, (\varepsilon)}$ for $\nu \in \mathfrak{N}(\mu)$, then \eqref{rec supp eq a} and Lemma~\ref{sec supp lem} imply $\chi_{k,\ell}^{*,\mu}(\xi) = 1$. Thus, if we define the corresponding frequency projection
\begin{equation*}
    f^{*, \mu}_{k,\ell} := \chi_{k,\ell}^{*,\mu}(D)f,
\end{equation*}
it follows that
\begin{equation*}
    m[a_{k,\ell}^{\nu, (\varepsilon)}](D;\,\cdot\,)f = m[a_{k,\ell}^{\nu, (\varepsilon)}](D;\,\cdot\,)f^{*, \mu}_{k,\ell} \qquad \textrm{for all $\nu \in \mathfrak{N}_{\ell}(\mu)$.}
\end{equation*}

Recall from Lemma~\ref{J=3 supp lem} b) that we also have
\begin{equation}\label{rec supp eq b}
  \xisupp a_{k,\ell}^{\nu, (\varepsilon)} \subseteq 2^{k-\ell} \cdot \pi_0(s_{\nu}; 2^{-(k-\ell)/2}, 2^{\ell}), \qquad \textrm{where $s_\nu:=2^{-(k-\ell)/2}\nu$.}
\end{equation}
Fix some $0 \leq \ell \leq \floor{k/3}$ and $\mu \in \Z$ with $s_{\mu} := 2^{-\ell} \mu \in [-1,1]$. To simplify notation, let $\sigma := s_{\mu}$, $\lambda := 2^{-\ell}$ and let $\widetilde{\gamma} := \gamma_{\sigma, \lambda}$ denote the rescaled curve, as defined in Definition~\ref{rescaled curve def}, so that
\begin{equation}\label{J4 gamma resc}
    \widetilde{\gamma}(s) := \big([\gamma]_{\sigma, \lambda}\big)^{-1} \big( \gamma(\sigma + \lambda s) - \gamma(\sigma)\big). 
\end{equation}
Let $(\tilde{\be}_j)_{j=1}^4$ denote the Frenet frame defined with respect to $\widetilde{\gamma}$. Given $0 < r \leq 1$ and $s \in I_0$, recall the definition of the $(0,r)$-\textit{Frenet boxes} (with respect to $(\tilde{\be}_j)_{j=1}^3$) introduced in Definition~\ref{def Frenet box}: 
\begin{equation*}
     \pi_{0,\widetilde{\gamma}}(s;r):= \big\{\xi \in \widehat{\R}^3: |\inn{\tilde{\be}_1(s)}{\xi}| \lesssim r, \quad |\inn{\tilde{\be}_2(s)}{\xi}| \sim 1,  \quad |\inn{\tilde{\be}_3(s)}{\xi}| \lesssim 1\big\}.
\end{equation*}
 Note that all these definitions depend of the choice of $\mu$ and $\ell$, but we suppress this dependence in the notation. 

\begin{lemma}\label{freq resc lem} With the above setup, and $\nu \in \mathfrak{N}(\mu)$,
\begin{equation*}
    [\gamma]_{\sigma,\lambda}^{\top}  \cdot 2^{k-\ell} \cdot \pi_{0,\gamma}(s_{\nu}; 2^{-(k-\ell)/2}, 2^{\ell}) \subseteq 2^{k-3\ell} \cdot \pi_{0,\widetilde{\gamma}}(\tilde{s}_{\nu}; 2^{-(k-3\ell)/2}),
\end{equation*}
 where $\tilde{s}_{\nu} := 2^{\ell} (s_{\nu} - s_{\mu})$ for $s_{\nu} := 2^{-(k - \ell)/2}\nu$. 
\end{lemma}

\begin{proof} Let $\xi \in 2^{k-\ell} \cdot \pi_{0,\gamma}(s_{\nu}; 2^{-(k-\ell)/2}, 2^{\ell})$ so that 
\begin{equation*}
     |\inn{\be_1(s_{\nu})}{\xi}| \lesssim 2^{(k-\ell)/2} \quad, \quad |\inn{\be_2(s_{\nu})}{\xi}| \sim 2^{k-\ell}, \quad
    |\inn{\be_4(s_{\nu})}{\xi}| \sim 2^k.
\end{equation*}
Since the matrix corresponding to the change of basis from $\big(\be_j(s_{\nu})\big)_{j=1}^3$ to $\big(\gamma^{(j)}(s_{\nu})\big)_{j=1}^3$ is lower triangular and an $O(\delta_0)$ perturbation of the identity, provided $\delta_0$ is sufficiently small,
\begin{equation*}
     |\inn{\gamma^{(1)}(s_{\nu})}{\xi}| \lesssim 2^{(k-\ell)/2}, \quad |\inn{\gamma^{(2)}(s_{\nu})}{\xi}| \sim 2^{k-\ell}, \quad
    |\inn{\gamma^{(3)}(s_{\nu})}{\xi}| \sim 2^k.
\end{equation*}

Now define $\tilde{\xi} := \big([\gamma]_{\sigma, \lambda}\big)^{\top} \cdot \xi$. Since $\lambda := 2^{-\ell}$, it follows from the definition of $\widetilde{\gamma}$ from \eqref{J4 gamma resc} that
\begin{equation*}
   \inn{\widetilde{\gamma}^{(j)}(\tilde{s}_{\nu})}{\tilde{\xi}\,} = 2^{-j\ell}\inn{\gamma^{(j)}(s_{\nu})}{\xi}  \qquad \textrm{for $j \geq 1$}.
\end{equation*}
Combining the above observations,
\begin{equation*}
     |\inn{\widetilde{\gamma}^{(1)}(\tilde{s}_{\nu})}{\tilde{\xi}\,}| \lesssim 2^{(k-3\ell)/2}, \quad |\inn{\widetilde{\gamma}^{(2)}(\tilde{s}_{\nu})}{\tilde{\xi}\,}| \sim 2^{k-3\ell}, \quad
    |\inn{\widetilde{\gamma}^{(3)}(\tilde{s}_{\nu})}{\tilde{\xi}\,}| \sim 2^{k-3\ell}.
\end{equation*}
Provided $\delta_0$ is sufficiently small, the desired result now follows since the matrix corresponding to the change of basis from $\big(\tilde{\be}_i(\tilde{s}_{\nu})\big)_{i=1}^3$ to $\big(\widetilde{\gamma}^{(i)}(\tilde{s}_{\nu})\big)_{i=1}^3$ is also lower triangular and an $O(\delta_0)$ perturbation of the identity. 
\end{proof}

For $\nu \in \mathfrak{N}_{\ell}(\mu)$ define the smooth cutoff
\begin{equation}\label{nu smooth cutoff}
    \chi_{k, \ell}^{\nu}(\xi) := \chi_{\tilde{\pi}}\big(C^{-1} 2^{-(k-3\ell)} [\gamma]_{\sigma, \lambda}^{\top} \cdot \xi \big)
\end{equation}
where $\chi_{\tilde{\pi}}$ is as defined in \eqref{chi pi} for $\tilde{\pi} := \pi_{0,\widetilde{\gamma}}(\tilde{s}_{\nu}; 2^{-(k-3\ell)/2})$ as above. If $\xi \in \xisupp a_{k,\ell}^{\nu, (\varepsilon)}$, then \eqref{rec supp eq b} and Lemma~\ref{freq resc lem} imply $\chi_{k,\ell}^{\nu}(\xi) = 1$. Thus if we define the corresponding frequency projection 
\begin{equation*}
    f_{k,\ell}^{\nu} := \chi_{k,\ell}^{\nu}(D)f_{k,\ell}^{*,\mu},
\end{equation*}
it follows that
\begin{equation*}
    m[a_{k,\ell}^{\nu, (\varepsilon)}](D;\,\cdot\,)f = m[a_{k,\ell}^{\nu, (\varepsilon)}](D;\,\cdot\,)f_{k,\ell}^{*, \mu} = m[a_{k,\ell}^{\nu, (\varepsilon)}](D;\,\cdot\,)f_{k,\ell}^{\nu}  \qquad \textrm{for all $\nu \in \mathfrak{N}_{\ell}(\mu)$.}
\end{equation*}




\subsection{\texorpdfstring{$L^2$}{}-weighted bounds in \texorpdfstring{$\R^{3+1}$}{}}\label{L2 wtd 3+1 sec} We apply a standard duality argument to analyse the square function appearing in Proposition~\ref{J=3 rev sf prop}. In particular, we use $L^2$-weighted approach and the key ingredient is the Nikodym-type maximal inequality from \S\ref{Nikodym lem subsec}.\medskip

By duality, there exists a non-negative $g \in L^2(\R^{3+1})$ with $\|g\|_{L^2(\R^{3+1})} = 1$ such that
\begin{equation*}
     \Big\| \big(\sum_{\nu \in \Z} |m[a_{k,\ell}^{\nu,(\varepsilon)}](D; \,\cdot\,) f|^2 \big)^{1/2} \Big\|_{L^4(\R^{3+1})}^2 = \sum_{\nu \in \Z} \int_{\R^{3+1}}  |m[a_{k,\ell}^{\nu,(\varepsilon)}](D;t)f(x)|^2 g(x;t)\,\ud x \ud t.
\end{equation*}
By the observations of the previous subsection,
\begin{equation*}
    m[a_{k,\ell}^{\nu,(\varepsilon)}](D;t)f = m[a_{k,\ell}^{\nu,(\varepsilon)}](D;t) f_{k,\ell}^{\nu}.
\end{equation*}
Let $\psi_{\,\mathcal{T}_{k,\ell}(s_{\nu})}$ be the weight introduced in Lemma~\ref{J=3 ker lem}. Since the $\psi_{\,\mathcal{T}_{k,\ell}(s_{\nu})}(\,\cdot\,;t)$ are $L^1$-normalised uniformly in $t$, it follows from Lemma~\ref{J=3 ker lem} and the Cauchy--Schwarz inequality that
\begin{equation}\label{L2 weighted 3+1 1}
   |m[a_{k,\ell}^{\nu,(\varepsilon)}](D;t)f(x)|^2 \lesssim 2^{-(k-\ell)} 2^{O(\varepsilon k)}  \psi_{\,\mathcal{T}_{k,\ell}(s_{\nu})}(\,\cdot\,;t) \ast |f_{k,\ell}^{\nu}|^2(x) \, \rho(t).
\end{equation}

Define the Nikodym-type maximal operator
\begin{equation*}
    \widetilde{\mathcal{N}}_{k,\ell}^{\,\mathrm{sing}}\,g(x) := \max_{\nu \in \Z \,:\, |s_{\nu}| \leq \delta_0} \int_{\R^4} |g(x - y,t)|  \psi_{\,\mathcal{T}_{k,\ell}(s_{\nu})}(y,t) \, \rho(t)\,\ud y \ud t.
\end{equation*}
By \eqref{L2 weighted 3+1 1} and Fubini's theorem, it follows that
 \begin{equation*}
     \sum_{\nu \in \Z} \int_{\R^{3+1}}  | m[a_{k,\ell}^{\nu, (\varepsilon)}](D;t)f(x)|^2 g(x;t)\,\ud x \ud t \lesssim 2^{-(k-\ell)}2^{O(\varepsilon k)} \int_{\R^{3}}\sum_{\nu \in \Z}  |f_{k,\ell}^{\nu}(x)|^2 \widetilde{\mathcal{N}}_{k,\ell}^{\,\mathrm{sing}}\,g(x)\,\ud x.
  \end{equation*}

Note that $\widetilde{\mathcal{N}}_{k,\ell}^{\,\mathrm{sing}}$ is essentially a smooth version of the maximal operator $\mathcal{N}_{\mathbf{r}}^{\,\mathrm{sing}}$ from \S\ref{Nikodym lem subsec} with parameters $r_1 := 2^{-(k-\ell)/2}$, $r_2 := 2^{-(k-\ell)}$ and $r_3 := 2^{-k}$. By the restriction $0 \leq \ell \leq \floor{k/3}$, it follows that this choice of $\br$ satisfies the hypotheses 
\begin{equation*}
   r_3 \leq r_2 \leq r_1 \leq r_2^{1/2} \qquad \textrm{and} \qquad  r_2 \leq r_{1}^{1/2} r_3^{1/2}
\end{equation*}
from the statement of Proposition~\ref{Nikodym prop}. Thus, by pointwise dominating $\psi_{\,\mathcal{T}_{k,\ell}(s_{\nu})}$ by a weighted series of indicator functions and applying Proposition~\ref{Nikodym prop}, one readily deduces the norm bound 
\begin{equation*}
    \big\|\widetilde{\mathcal{N}}_{k,\ell}^{\,\mathrm{sing}}\big\|_{L^2(\R^{3+1}) \to L^2(\R^3)} \lesssim_{\varepsilon} 2^{\varepsilon k}.
\end{equation*}

By combining the above observations with an application of the Cauchy--Schwarz inequality,
\begin{equation}\label{L2 weighted 3+1 2}
    \Big\| \big(\sum_{\nu \in \Z} |m[a_{k,\ell}^{\nu,(\varepsilon)}](D; \,\cdot\,) f|^2 \big)^{1/2} \Big\|_{L^4(\R^{3+1})} \lesssim 2^{-(k-\ell)/2 + O(\varepsilon k)} \big\| \big(\sum_{\nu \in \Z}  |f_{k,\ell}^{\nu}|^2\big)^{1/2}\big\|_{L^4(\R^3)}.
\end{equation}
It remains to bound the right-hand square function, which involves only functions of $3$ variables.




\subsection{\texorpdfstring{$L^2$}{}-weighted bounds in \texorpdfstring{$\R^3$}{}} A similar $L^2$-weighted approach is now applied one dimension lower to estimate the square function appearing in the right-hand side of \eqref{L2 weighted 3+1 2}.

\begin{proposition}\label{L4 forward SF prop}
Let $k \in \N$, $0 \leq \ell \leq \floor{k/3}$ and $\varepsilon>0$. Then
\begin{equation}\label{L2 wtd 3 main} 
   \big\| \big(\sum_{\nu \in \Z}  |f_{k,\ell}^{\nu}|^2\big)^{1/2}\big\|_{L^4(\R^3)} \lesssim_{\varepsilon} 2^{O(\varepsilon k)} \|f\|_{L^4(\R^3)}.
\end{equation}
\end{proposition}

\begin{proof}
By duality, there exists a non-negative $w \in L^2(\R^3)$ with $\|w\|_{L^2(\R^3)} = 1$ such that
\begin{equation}\label{L2 wtd 3 0}
   \big\| \big(\sum_{\nu \in \Z}  |f_{k,\ell}^{\nu}|^2\big)^{1/2}\big\|_{L^4(\R^3)}^2 = \sum_{\mu \in \Z} \int_{\R^3} \sum_{\nu \in \mathfrak{N}_{\ell}(\mu)} |f_{k,\ell}^{\nu}(x)|^2 w(x)\,\ud x. 
\end{equation}

Recall that the $f_{k,\ell}^{\nu}$ are defined by
\begin{equation*}
    f_{k,\ell}^{\nu} := \chi_{k,\ell}^{\nu}(D)f_{k,\ell}^{*,\mu}  \qquad \textrm{for $\nu \in \mathfrak{N}_{\ell}(\mu)$} 
\end{equation*}
where the smooth cutoff function $\chi_{k,\ell}^{\nu}$ is as defined in \eqref{nu smooth cutoff}. Fix $\mu$ and, as in \S\ref{f freq loc subsec}, let $\sigma := s_{\mu}$ and $\lambda := 2^{-\ell}$. Define $\tilde{f}_{k,\ell}^{\nu} := f_{k,\ell}^{\nu} \circ [\gamma]_{\sigma,\lambda}$, $\tilde{f}_{k,\ell}^{\,*, \mu} := f_{k,\ell}^{*, \mu} \circ [\gamma]_{\sigma,\lambda}$ and $\tilde{w} := w \circ [\gamma]_{\sigma, \lambda}$ so, by a change of variables, 
\begin{equation}\label{L2 wtd 3 1}
   \int_{\R^3} \sum_{\nu \in \mathfrak{N}_{\ell}(\mu)} |f_{k,\ell}^{\nu}(x)|^2\, w(x)\,\ud x = |\det[\gamma]_{\sigma, \lambda}| \int_{\R^3} \sum_{\nu \in \mathfrak{N}_{\ell}(\mu)} |\tilde{f}_{k,\ell}^{\nu}(x)|^2\, \tilde{w}(x)\,\ud x.
\end{equation}
By the definition of $\chi_{k,\ell}^{\nu}$ and Lemma \ref{freq resc lem}, each of the $\tilde{f}_{k,\ell}^{\nu}$ is Fourier supported in a $2^{k-3\ell}$ dilate of a $(0, 2^{-(k-3\ell)/2})$-Frenet box. In view of this, we may apply  Proposition~\ref{f SF prop} to deduce that
\begin{equation}\label{L2 wtd 3 1 b}
    \int_{\R^3} \sum_{\nu \in \mathfrak{N}_{\ell}(\mu)} |f_{k,\ell}^{\nu}(x)|^2 \, w(x)\,\ud x \lesssim_{\varepsilon} 2^{\varepsilon k} |\det[\gamma]_{\sigma, \lambda}| \int_{\R^3} |\tilde{f}_{k,\ell}^{\,*,\,\mu}(x)|^2 \, \widetilde{\mathcal{N}}^{\,\mu, (\varepsilon)}_{k,\ell}\tilde{w}(x)\,\ud x
\end{equation}
where the operator $\widetilde{\mathcal{N}}^{\,\mu, (\varepsilon)}_{k,\ell}$ is defined by
\begin{equation*}
    \widetilde{\mathcal{N}}^{\,\mu, (\varepsilon)}_{k,\ell} := \mathrm{Dil}_{2^{k-3\ell}} \circ \widetilde{\mathcal{N}}^{\,(\varepsilon)}_{\widetilde{\gamma},\tilde{r}} \circ \mathrm{Dil}_{2^{-(k-3\ell)}} 
\end{equation*}
for $\tilde{r} := 2^{-(k-3\ell)/2}$ and $\mathrm{Dil}_{\rho} \colon L^2(\R^3) \to L^2(\R^3)$ the dilation operator $\mathrm{Dil}_{\rho}\, f := f(\rho \;\cdot\,)$ for $\rho > 0$. Here $\widetilde{\mathcal{N}}^{\,(\varepsilon)}_{\widetilde{\gamma},\tilde{r}}$ is the maximal operator featured in the statement of Proposition~\ref{f SF prop} (the precise definition is given in \S\ref{forward SF sec}).
Note that $\widetilde{\mathcal{N}}^{\,\mu, (\varepsilon)}_{k,\ell}$ depends on the choice of $\mu$. By reversing the change of variables in \eqref{L2 wtd 3 1}, we can show the following.

\begin{claim} There exists a maximal function $\mathcal{N}_{k,\ell}^{(\varepsilon)}$, independent of $\mu$, such that
\begin{equation*}
    \big([\gamma]_{\sigma,\lambda}\big)^{-1} \circ \widetilde{\mathcal{N}}_{k, \ell}^{\mu, (\varepsilon)} \circ [\gamma]_{\sigma,\lambda} \cdot  w(x) \lesssim_\gamma \mathcal{N}_{k, \ell}^{\,(\varepsilon)} w(x) \qquad \textrm{for all $x \in \R^3$},
\end{equation*}
where $[\gamma]_{\sigma, \lambda} \cdot f:= f \circ [\gamma]_{\sigma,\lambda}$, and
\begin{equation}\label{rescaled max L2 bound}
    \|\mathcal{N}^{\, (\varepsilon)}_{k,\ell}\|_{L^2(\R^3) \to L^2(\R^3)} \lesssim_{\varepsilon} 2^{\varepsilon k}.
\end{equation}
\end{claim}

The proof of the above claim requires additional information on the form of the maximal operators arising from Proposition~\ref{f SF prop}. Since the definitions involved are somewhat unwieldy, the details are postponed until \S\ref{Nik scale subsec}.\smallskip

Assuming the claim, changing variables in \eqref{L2 wtd 3 1 b} yields
\begin{equation*}
    \int_{\R^3} \sum_{\nu \in \mathfrak{N}_{\ell}(\mu)} |f_{k,\ell}^{\nu}(x)|^2 \, w(x)\,\ud x \lesssim 2^{\varepsilon k}  \int_{\R^3} |f_{k,\ell}^{*,\,\mu}(x)|^2 \, \mathcal{N}^{\, (\varepsilon)}_{k,\ell} w(x)\,\ud x.
\end{equation*}
Recalling \eqref{L2 wtd 3 0}, one can sum the above inequality in $\mu \in \Z$, and use \eqref{rescaled max L2 bound} and the Cauchy--Schwarz inequality to obtain
\begin{equation}\label{L2 wtd 3 3a}
     \big\| \big(\sum_{\nu \in \Z}  |f_{k,\ell}^{\nu}|^2\big)^{1/2}\big\|_{L^4(\R^3)}^2 \lesssim_{\varepsilon} 2^{2\varepsilon k}  \big\| \big(\sum_{\mu \in \Z}  |f_{k,\ell}^{*,\mu}|^2\big)^{1/2}\big\|_{L^4(\R^3)}^2 .
\end{equation}
Recall that each $f_{k,\ell}^{*,\mu}$ corresponds to a (smooth) frequency projection of $f$ onto the set $\Delta_{\ell}(m(\mu))$, as defined in \eqref{sector def}. Furthermore, by Lemma~\ref{sec supp lem} the assignment $m \mapsto m(\mu)$ is $O(1)$-to-1. Thus, the right-hand square function in \eqref{L2 wtd 3 3a} falls under the scope of the classical sectorial square function of C\'ordoba \cite{Cordoba1982}. In particular, by \cite[Theorem 1]{Cordoba1982} (see also \cite{CS1995}) and a Fubini argument, we have
\begin{equation}\label{L2 wtd 3 3b}
    \big\| \big(\sum_{\mu \in \Z}  |f_{k,\ell}^{*,\mu}|^2\big)^{1/2}\big\|_{L^4(\R^3)}^2  \lesssim_\varepsilon 2^{O(\varepsilon k)} \|f\|_{L^4(\R^3)}.
\end{equation}
The inequalities \eqref{L2 wtd 3 3a} and \eqref{L2 wtd 3 3b} imply the desired estimate \eqref{L2 wtd 3 main}.
\end{proof}




\subsection{Putting everything together} We combine our observations to establish favourable $L^4$ and $L^2$ estimates for the localised multipliers $m[a_{k,\ell}]$.\medskip

\noindent \textit{$L^4$ estimates}. By Lemma~\ref{J=3 s loc lem},
\begin{equation*}
\|m[a_{k,\ell}](D;\,\cdot\,)f\|_{L^4(\R^{3+1})} \lesssim_{\varepsilon, N} \|m[a_{k,\ell}^{(\varepsilon)}](D;\,\cdot\,)f\|_{L^4(\R^{3+1})} + 2^{-kN} \|f\|_{L^4(\R^3)}.
\end{equation*}
Decompose each $m[a_{k,\ell}^{(\varepsilon)}]$ as a sum of multipliers $m[a_{k,\ell}^{\nu,(\varepsilon)}]$ as defined in \S\ref{J=3 Fourier loc subsec}. By Proposition~\ref{J=3 rev sf prop}, it follows that 
\begin{equation*}
\|m[a_{k,\ell}](D;\,\cdot\,)f\|_{L^4(\R^{3+1})} \lesssim_{\varepsilon, N} 2^{(k-3\ell)/4} 2^{O(\varepsilon k)} \Big\| \big(\sum_{\nu \in \Z} |m[a_{k,\ell}^{\nu,(\varepsilon)}](D; \,\cdot\,) f|^2 \big)^{1/2} \Big\|_{L^4(\R^{3+1})} + 2^{-kN} \|f\|_{L^4(\R^3)} .  
\end{equation*}
Thus, \eqref{L2 weighted 3+1 2} and \eqref{L2 wtd 3 main} combine with the previous display to yield the $L^4$ estimate
\begin{equation*}
 \|m[a_{k,\ell}](D;\,\cdot\,)f\|_{L^4(\R^{3+1})} \lesssim_{\varepsilon} 2^{(k-3\ell)/4} 2^{-(k-\ell)/2} 2^{O(\varepsilon k)} \|f\|_{L^4(\R^3)}. 
\end{equation*}
Since $\varepsilon > 0$ may be chosen arbitrarily, this corresponds to the $p=4$ case of Proposition~\ref{J=3 LS prop}.\medskip

\noindent \textit{$L^2$ estimates}. Arguing as in the proof of the $L^4$ estimate, but now using Lemma~\ref{J=3 L2 rev sf prop} rather than Proposition~\ref{J=3 rev sf prop}, it follows that 
\begin{equation*}
\|m[a_{k,\ell}](D;\,\cdot\,)f\|_{L^2(\R^{3+1})} \lesssim_{\varepsilon, N} \Big\| \big(\sum_{\nu \in \Z} |m[a_{k,\ell}^{\nu,(\varepsilon)}](D; \,\cdot\,) f|^2 \big)^{1/2} \Big\|_{L^2(\R^{3+1})} + 2^{-kN} \|f\|_{L^2(\R^3)} .  
\end{equation*}
Recall from \eqref{L2 weighted 3+1 1} that
\begin{equation*}
    |m[a_{k,\ell}^{\nu,(\varepsilon)}](D;t)f(x)|^2 \lesssim 2^{-(k-\ell)/2} 2^{O(\varepsilon k)} \psi_{\,\mathcal{T}_{k,\ell}(s_{\nu})}(\,\cdot\,;t) \ast |f_{k,\ell}^{\nu}|^2.
\end{equation*}
Thus, by Young's convolution inequality and the fact that the $\psi_{\,\mathcal{T}_{k,\ell}(s_{\nu})}(\,\cdot\,;t)$ are $L^1$-normalised,
\begin{equation*}
\|m[a_{k,\ell}](D;\,\cdot\,)f\|_{L^2(\R^{3+1})} \lesssim_{\varepsilon, N} 2^{-(k-\ell)/2} 2^{O(\varepsilon k)} \big\| \big(\sum_{\nu \in \Z} |f_{k,\ell}^{\nu}|^2 \big)^{1/2} \big\|_{L^2(\R^3)} + 2^{-kN} \|f\|_{L^2(\R^3)} .  
\end{equation*}
Finally, as the $f_{k,\ell}^{\nu}$ have essentially disjoint Fourier support, by Plancherel's theorem,
\begin{equation*}
\|m[a_{k,\ell}](D;\,\cdot\,)f\|_{L^2(\R^{3+1})} \lesssim_{\varepsilon} 2^{-(k-\ell)/2} 2^{O(\varepsilon k)} \|f\|_{L^2(\R^3)} .  
\end{equation*}
Since $\varepsilon > 0$ may be chosen arbitrarily, this corresponds to the $p=2$ case of Proposition~\ref{J=3 LS prop}.\medskip

Interpolating the above estimates, given $2 \leq p \leq 4$ and $\varepsilon > 0$, it follows that
\begin{equation*}
    \|m[a_{k,\ell}](D;\,\cdot\,)f\|_{L^p(\R^{3+1})} \lesssim_{\varepsilon, N} 2^{-(k-\ell)/2}2^{(k-3\ell)(1/2 -1/p)} 2^{\varepsilon k} \|f\|_{L^p(\R^3)}, 
\end{equation*}
which is precisely the desired inequality from Proposition~\ref{J=3 LS prop}.




\section{Proof of the reverse square function inequality in \texorpdfstring{$\R^{3+1}$}{}}\label{reverse SF sec}




\subsection{Geometric observations}\label{geo obs sec} 
The first step is to relate the Frenet boxes $\pi_{2,\gamma}(s;r)$ to a codimension 2 cone $\widetilde{\Gamma}_2$ in the $(\xi, \tau)$-space. \medskip




\noindent \textit{The underlying cone.} Let $\gamma \in \mathfrak{G}_4(\delta_0)$ for $0 < \delta_0 \ll 1$ and $\be_j \colon[-1,1] \to S^3$ for $1 \leq j \leq 4$ be the associated Frenet frame. Without loss of generality, in proving Theorem~\ref{Frenet reverse SF theorem} we may always localise so that we only consider the portion of the curve lying over the interval $I_0 = [-\delta_0, \delta_0]$. In this case
\begin{equation}\label{Frenet loc}
\be_j(s) = \vec{e}_j + O(\delta_0) \qquad \textrm{for $1 \leq j \leq 4$}
\end{equation}
where, as usual, the $\vec{e}_j$ denote the standard basis vectors. 

We consider the conic surface $\widetilde{\Gamma}_2$ `generated' over the curve $s \mapsto \be_4(s)$. This is similar to the analysis of \cite{PS2007}, where a cone in $\R^3$ generated by the binormal $\be_3$ features prominently in the arguments. Define $G \colon I_0 \to \R^4$ by $G(s) := \be_{44}(s)^{-1} \be_4(s)$ for all $s \in I_0$ (note that $\be_{44}(s)$ is bounded away from $0$ by \eqref{Frenet loc}), so that $G$ is of the form 
\begin{equation*}
    G(s) =
    \begin{bmatrix}
    g(s) \\
    1
    \end{bmatrix} 
    \qquad \textrm{for} \qquad g(s) := 
    \Big(
    \frac{\be_{41}(s)}{\be_{44}(s)},\; \frac{\be_{42}(s)}{\be_{44}(s)},\; \frac{\be_{43}(s)}{\be_{44}(s)} 
    \Big)^{\top}.
\end{equation*}
For $U := [1/4, 4]\times I_0$, the $2$-dimensional cone $\widetilde{\Gamma}_2$ is parametrised by the function
\begin{equation*}
    \widetilde{\Gamma}_2 \colon U \to \R^4, \qquad (\rho, s) \mapsto \rho\, G(s). 
\end{equation*}



\noindent \textit{Non-degeneracy conditions.} We claim that the curve $g \colon I_0 \to \R^3$ is non-degenerate. To see this, first note that
\begin{equation*}
 G^{(i)}(s) \in \langle \be_4(s), \be_4^{(1)}(s), \dots, \be_4^{(i)}(s) \rangle
\end{equation*}
where the right-hand expression denotes the linear span of the vectors $\be_4(s), \be_4^{(1)}(s), \dots, \be_4^{(i)}(s)$. Thus, one concludes from the Frenet formul\ae\ that
\begin{equation}\label{g a non deg 3}
    G^{(i)}(s)  \in \langle \be_{4 -i}(s), \dots, \be_4(s)  \rangle \qquad \textrm{for $0 \leq i \leq 3$.}
\end{equation}
On the other hand, the Frenet formul\ae\ together with the Leibniz rule show that 
\begin{equation*}
   \inn{G^{(i)}(s)}{\be_{4 -i}(s)} = (-1)^i\Big( \prod_{\ell = 4 -i}^3 \tilde{\kappa}_{\ell}(s) \Big) \, \be_{44}(s)^{-1}
\end{equation*}
and, consequently, 
\begin{equation}\label{g a non deg 5}
   |\inn{G^{(i)}(s)}{\be_{4 -i}(s)} | \sim 1 \qquad \textrm{for all $1 \leq i \leq 3$.}
\end{equation}
Thus, combining \eqref{g a non deg 3} and \eqref{g a non deg 5}, it follows that the vectors $G^{(i)}(s)$, $1 \leq i \leq 3$, are linearly independent. From this, we immediately conclude that 
\begin{equation*}
    |\det [g]_{s}| \gtrsim 1
\end{equation*}
for all $s \in I_0$, which is the claimed non-degeneracy condition.
\medskip




\noindent \textit{Frenet boxes revisited.} By the preceding observations, the vectors $G^{(i)}(s)$ for $1 \leq i \leq 3$ form a basis of $\R^3 \times \{0\}$. Fixing $\xi \in \widehat{\R}^3$ and $r > 0$, one may write
\begin{equation}\label{Frenet box 1}
    \xi -  \xi_4 G(s) =  \sum_{i=1}^3 r^{i} \eta_i G^{(i)}(s)
\end{equation}
for some vector of coefficients $(\eta_1, \eta_2, \eta_3) \in \R^3$. The powers of $r$ appearing in the above expression play a normalising r\^ole below. For each $1 \leq k \leq 3$ form the inner product of both sides of the above identity with the Frenet vector $\be_k(s)$. Combining the resulting expressions with the linear independence relations inherent in \eqref{g a non deg 3}, the coefficients $\eta_k$ can be related to the numbers $\inn{\xi}{\be_k(s)}$ via a lower anti-triangular transformation, viz. 
\begin{equation}\label{Frenet box 2}
    \begin{bmatrix}
    \inn{\xi}{\be_1(s)} \\
    \inn{\xi}{\be_2(s)} \\
    \inn{\xi}{\be_3(s)}
    \end{bmatrix}
    =
    \begin{bmatrix}
    0 & 0 & \inn{G_{\ba}^{(3)}(s)}{\be_1(s)} \\
    0 & \inn{G_{\ba}^{(2)}(s)}{\be_2(s)} & \inn{G_{\ba}^{(3)}(s)}{\be_2(s)} \\
    \inn{G_{\ba}^{(1)}(s)}{\be_3(s)} & \inn{G_{\ba}^{(2)}(s)}{\be_3(s)} & \inn{G_{\ba}^{(3)}(s)}{\be_3(s)}
    \end{bmatrix}
    \begin{bmatrix}
    r\eta_1 \\
    r^2\eta_2 \\
    r^3 \eta_3
    \end{bmatrix}.
\end{equation}
Recall that
\begin{equation*}
    \pi_{2,\bar{\gamma}}(s;\,r) := \big\{\xi \in \widehat{\R}^4 : |\inn{\be_j(s)}{\xi}| \lesssim r^{4 - j} \textrm{ for $1 \leq j \leq 3$, } \,  |\inn{\be_4(s)}{\xi}| \sim 1 \big\}.
\end{equation*}
Thus, if $\xi \in \pi_{2,\gamma}(s;r)$, then it follows from combining the above definition and \eqref{g a non deg 5} with \eqref{Frenet box 2}  that $|\eta_i| \lesssim_{\gamma} 1$ for $1 \leq i \leq 3$. Similarly,
the localisation \eqref{Frenet loc} implies that 
\begin{equation*}
 \pi_{2,\gamma}(s;r) \subseteq   \mathcal{R} := [-2,2]^3 \times [1/4,4].
 \end{equation*}

The identity \eqref{Frenet box 1} can be succinctly expressed using matrices. In particular, for $s \in I_0$ and $r > 0$, define the $4 \times 4$ matrix
\begin{equation}\label{Frenet box 3}
    [g]_{\mathcal{C}, s,r} :=
    \begin{pmatrix}
    [g]_{s,r} & g(s) \\
    0 & 1
    \end{pmatrix}.
\end{equation}
Here the block $[g]_{s,r}$ is the $3 \times 3$ matrix as defined in \eqref{gamma transformation}. With this notation, the identity \eqref{Frenet box 1} may be written as
\begin{equation*}
    \xi = [g]_{\mathcal{C},s,r} \cdot \eta \qquad \textrm{where $\eta = (\eta_1, \eta_2, \eta_3, \xi_4)$.}
\end{equation*}
Moreover, if $\xi \in \pi_{2,\gamma}(s;r)$, then the preceding observations show that $\eta$ in the above equation may be taken to lie in a bounded region and so
\begin{equation}\label{Frenet box 4}
    \pi_{2,\gamma}(s;r) \subseteq [g]_{\mathcal{C},s,C r}\big([-2,2]^4\big) \cap \mathcal{R} ,
\end{equation}
where $C \geq 1$ is a suitably large dimensional constant.\medskip




\subsection{A square function estimate for cones generated by non-degenerate curves} Here the geometric setup described in \S\ref{geo obs sec} is abstracted. 

\begin{definition}
For $g \colon [-1,1] \to \R^3$ a smooth curve, let $\Gamma_g $ denote the codimension $2$ cone in $\R^4$  parametrised by
\begin{equation*}
 (\rho,s) \mapsto \rho \,
 \begin{pmatrix}
    g(s) \\
    1
    \end{pmatrix} \qquad \textrm{for } (\rho, s) \in U := [1/4,4] \times [-1,1].
\end{equation*}
In this case, $\Gamma_g $ is referred to as the \textit{cone generated by $g$}.
\end{definition}

In view of \eqref{Frenet box 4}, one wishes to establish a reverse square function estimate with respect to the $r$-\textit{plates}
\begin{equation*}
  \theta(s;r) :=  [g]_{\mathcal{C}, s,r} \big( [-2,2]^4 \big) \cap \mathcal{R}. 
\end{equation*}
In some cases it will be useful to highlight the choice of function $g$ by writing $\theta(g; s;r)$ for $\theta(s;r)$. Note that each of these plates lies in a neighbourhood of the cone $\Gamma_g$. We think of the union of all plates $\theta(s;r)$ as $s$ varies over the domain $[-1,1]$ as forming an anisotropic neighbourhood of $\Gamma_g$.

\begin{definition} A collection $\Theta (r)$ of $r$-plates is a \textit{plate family for $\Gamma_g$} if it consists of $\theta(g;s;r)$ for $s$ varying over an $r$-separated subset of $[-1,1]$.
\end{definition}

In view of the preceding observations, Theorem~\ref{Frenet reverse SF theorem} is a consequence of the following result.

\begin{theorem}\label{reverse sf thm} Suppose $g \colon [-1,1] \to \R^3$ is a smooth, non-degenerate curve and $\Theta(r)$ is an $r$-plate family for $\Gamma_g$ for some dyadic $0 < r \leq 1$. For all $\varepsilon > 0$ the inequality
  \begin{equation*}
 \Big\|\sum_{\theta \in \Theta(r)}f_{\theta}\Big\|_{L^4(\R^4)} \lesssim_{\varepsilon} r^{-\varepsilon}   \Big\| \big(\sum_{\theta \in \Theta(r)}|f_{\theta}|^2\big)^{1/2}\Big\|_{L^4(\R^4)}
\end{equation*}
holds whenever $(f_{\theta})_{\theta \in \Theta(r)}$ is a sequence of functions satisfying $\supp \widehat{f}_{\theta} \subseteq \theta$ for all $\theta \in \Theta(r)$.
\end{theorem}




\subsection{Multilinear estimates}\label{multilinear subsec}

The proof of Theorem~\ref{reverse sf thm} follows an argument of Lee--Vargas~\cite{LV2012} which relies on first establishing a multilinear variant of the desired square function inequality.\medskip

Let $\fI$ denote the collection of all dyadic subintervals of $[-1,1]$ and for any dyadic number $0 < r \leq 1$ let $\fI(r)$ denote the subset of $\fI$ consisting of all intervals of length $r$. Given any pair of dyadic scales $0 < \lambda_1 \leq \lambda_2 \leq 1$ and $J \in \fI(\lambda_2)$, let $\fI(J;\,\lambda_1)$ denote the collection of all $I \in \fI(\lambda_1)$ which satisfy $I \subseteq J$.\smallskip

Fix $0 < r \leq 1$ and for each $0 \leq \lambda \leq 1$ decompose $\Theta(r)$ as a disjoint union of subsets $\Theta(I;\,r)$ for $I \in \fI(\lambda)$ such that:
\begin{enumerate}[i)]
    \item If $\theta(s;r) \in \Theta(I;\,r)$, then $s \in I$;
    \item If $r \leq \lambda_1 \leq \lambda_2$ and $J \in \fI(\lambda_2)$, then $ \Theta(J;\,r) = \bigcup_{I \in \fI(J;\lambda_1)} \Theta(I;\,r)$.
\end{enumerate}
Thus, if for all $r \leq \lambda \leq 1$ we define
\begin{equation}\label{rsq dyadic dec}
    f_I := \sum_{\theta \in \Theta(I;\,r)} f_{\theta} \qquad \textrm{for all $I \in \fI(\lambda)$},
\end{equation}
then for all $r \leq \lambda_1 \leq \lambda_2$ it follows that 
\begin{equation*}
    f_J = \sum_{I \in \fI(J;\,\lambda_1)} f_I \qquad \textrm{for all $J \in \fI(\lambda_2)$.} 
\end{equation*}

For each dyadic number $0 < \lambda \leq 1$ let $\fI^4_{\mathrm{sep}}(\lambda)$ denote the collection of $4$-tuples of intervals $\vec{I} = (I_1, \dots, I_4) \in \fI(\lambda)^4$ which satisfy the separation condition
\begin{equation*}
    \mathrm{dist}(I_1, \dots, I_4) := \min_{1 \leq \ell_1 < \ell_2 \leq 4} \mathrm{dist}(I_{\ell_1}, I_{\ell_2}) \geq \lambda.
\end{equation*}

\begin{proposition}\label{multi sf prop} Let $0 < r \leq \lambda < 1$ be dyadic. If $(I_1, \dots, I_4) \in \fI^4_{\mathrm{sep}}(\lambda)$ and $\varepsilon > 0$, then
  \begin{equation*}
 \Big\| \prod_{\ell=1}^4|\sum_{\theta \in \Theta(I_{\ell}; r)}f_{\theta}|^{1/4}\Big\|_{L^4(\R^4)} \lesssim_{\varepsilon} M(\lambda) r^{-\varepsilon} \prod_{\ell=1}^4  \Big\| \big(\sum_{\theta \in \Theta(I_{\ell}; r)}|f_{\theta}|^2\big)^{1/2}\Big\|_{L^4(\R^4)}^{1/4}
\end{equation*}
holds whenever $(f_{\theta})_{\theta \in \Theta(r)}$ is a sequence of functions satisfying $\supp \widehat{f}_{\theta} \subseteq \theta$ for all $\theta \in \Theta(r)$, where $\sup_{\lambda \in [\lambda_0, 1]} M(\lambda) <\infty$ for all $\lambda_0 > 0$.
\end{proposition}

Using a standard argument, Proposition~\ref{multi sf prop} will follow from a $4$-linear Fourier restriction estimate. To state the latter inequality, given an interval $J \subseteq [-1,1]$ let $\Gamma_J$ denote the image of $\Gamma_g \colon (\rho, s) \mapsto \rho \, (g(s),1)^\top$ restricted to the set $U_J :=  [1/4,4] \times J$ and, for $r > 0$, let $N_r \Gamma_J$ denote the $r$-neighbourhood of $\Gamma_J$.

\begin{proposition}\label{mult rest prop} If $(I_1, \dots, I_4) \in \mathfrak{I}_{\mathrm{sep}}^4(\lambda)$ , then for all $0 < r \leq \lambda$ and all $\varepsilon > 0$ the inequality
\begin{equation*}
 \big\| \prod_{\ell =1}^4 |F_{\ell}|^{1/4}\big\|_{L^4(\R^4)} \lesssim_{\varepsilon} M(\lambda) r^{1-\varepsilon} \prod_{\ell=1}^4 \|F_{\ell}\|_{L^2(\R^4)}^{1/4}
\end{equation*}
 holds for all $F_{\ell} \in L^2(\R^4)$ with $\mathrm{supp}\, \widehat{F}_{\ell} \subseteq N_r \Gamma_{I_\ell}$ for $1 \leq \ell \leq 4$.
\end{proposition}

  Given an interval $J \subset [-1,1]$, define the extension operator
\begin{equation*}
E_Jf(x) := \int_{U_J} e^{i \inn{\Gamma(u)}{x}}  f(u) \,\ud u \qquad \textrm{for all $f \in L^1(U_J)$,}
\end{equation*}
where $U_J := [1/4,4] \times J$ as above. By standard uncertainty principle techniques and Plancherel's theorem (see, for instance, \cite{Bennett2006} or \cite[Appendix]{Tao2020}), Proposition~\ref{mult rest prop} is a consequence of the following multilinear extension estimate.

\begin{proposition}\label{multilin ex prop} If $(I_1, \dots, I_4) \in \mathfrak{I}_{\mathrm{sep}}^4(\lambda)$, then for all $R \geq 1$ and all $\varepsilon > 0$ the inequality
\begin{equation*}
 \big\| \prod_{\ell=1}^4 |E_{I_{\ell}} f_{\ell}|^{1/4}\big\|_{L^4(B_R)} \lesssim_{\varepsilon}  M(\lambda) R^{\varepsilon} \prod_{\ell=1}^4 \|f_{\ell}\|_{L^2(U_{I_\ell})}^{1/4}
\end{equation*}
 holds for all $f_{\ell} \in L^2(U)$ for $1 \leq \ell \leq 4$, where $B_R$ denotes a ball of radius $R$.
\end{proposition}

We refer to the above references for the argument use to pass from Proposition~\ref{multilin ex prop} to Proposition~\ref{mult rest prop} and turn to the proof of the extension estimate.

\begin{proof}[Proof of Proposition~\ref{multilin ex prop}] This inequality is a special case (via a compactness argument) of the recent generalisation of the  Bennett--Carbery--Tao restriction theorem \cite{Bennett2006} due to Bennett--Bez--Flock--Lee \cite[Theorem 1.3]{BBFL2018}; an improved version with $R^\varepsilon$ replaced by $(\log R)^{O(d)}$ has also been obtained by Zhang \cite[(1.8)]{Zhang}, although the $R^\varepsilon$ loss suffices for our purposes. In order to see this, we must verify a certain linear-algebraic condition on the tangent planes to $\Gamma$. The setup is recalled presently.\medskip

Fix $u_{\ell} = (\rho_\ell, s_{\ell}) \in U_{I_{\ell}}$ for $1 \leq \ell \leq 4$. We construct a \textit{Brascamp--Lieb datum} $(\mathbf{L}, \mathbf{p})$ by taking
\begin{equation*}
    \mathbf{L} := (\pi_1, \dots, \pi_4) \qquad \textrm{and} \qquad \mathbf{p} := (p_1, \dots, p_4) := (1/2, \dots, 1/2) 
\end{equation*} 
where each $\pi_{\ell} \colon \R^4 \to V_{\ell}$ is the orthogonal projection map from $\R^4$ to the 2-dimensional tangent space $V_{\ell}$ to $\Gamma$ at $\Gamma(u_{\ell})$. With this definition, the problem is to show that $\mathrm{BL}(\mathbf{L}, \mathbf{p}) < \infty$, where the \textit{Brascamp--Lieb constant} $\mathrm{BL}(\mathbf{L}, \mathbf{p})$ is as defined in, for instance, \cite{BBFL2018}. By the characterisation of finiteness of the Brascamp--Lieb constant from \cite{BCCT2008} and our choice of datum, it suffices to verify the following two conditions:
\begin{enumerate}[i)]
\item $\displaystyle \sum_{\ell = 1}^4 (\dim \mathrm{Im}\,\pi_{\ell}) \, p_{\ell} = 4$.
\item $\displaystyle \dim V \leq \frac{1}{2} \sum_{\ell = 1}^4 \dim \big(\pi_{\ell} V\big)$ holds for all linear subspaces $V \subseteq \R^4$. 
\end{enumerate}
The scaling condition i) is immediate from the choice of datum and it remains to prove the dimension condition ii).\medskip

Clearly one may replace $\pi_{\ell}$ with the linear map associated to the $2 \times 4$ Jacobian matrix $\ud \Gamma |_{(\rho_\ell,s_{\ell})}$. By subtracting the first column from the third column and applying the fundamental theorem of calculus,
\begin{equation*}
    \det
    \begin{bmatrix}
    g(s_{\ell_1}) & g'(s_{\ell_1}) & g(s_{\ell_2}) & g'(s_{\ell_2}) \\
    1 & 0 & 1 & 0
    \end{bmatrix}
    = -\int_{s_{\ell_1}}^{s_{\ell_2}} 
    \det
    \begin{bmatrix}
    g'(s_{\ell_1}) & g'(s) & g'(s_{\ell_2})
    \end{bmatrix}
    \,\ud s.
\end{equation*}
Furthermore, by repeated application of column reduction and the fundamental theorem of calculus, it follows from the non-degeneracy hypothesis and the initial localisation that
\begin{equation*}
    \big|\det
    \begin{bmatrix}
    g'(s_{\ell_1}) & g'(s) & g'(s_{\ell_2})
    \end{bmatrix}\big| \gtrsim |s_{\ell_2} - s_{\ell_1}||s-s_{\ell_1}||s_{\ell_2}-s|;
\end{equation*}
see, for instance, \cite[Proposition 4.1]{GGPRY}. Consequently, the determinant has constant sign and
\begin{equation}\label{dim cond 1}
   \big| \det 
   \begin{bmatrix}
   \ud \Gamma |_{(\rho_{\ell_1}, s_{\ell_1})} &
   \ud \Gamma |_{(\rho_{\ell_2}, s_{\ell_2})}
   \end{bmatrix}\big| \gtrsim |\rho_{\ell_1}||\rho_{\ell_2}||s_{\ell_2} - s_{\ell_1}|^4 \gtrsim \lambda^4,
\end{equation}
where the final bound is due to the separation between the $I_{\ell}$. Note that \eqref{dim cond 1} is equivalent to the geometric condition that $V_{\ell_1} + V_{\ell_2} = \R^4$ and therefore
\begin{equation}\label{dim cond 2}
V_{\ell_1}^{\perp} \cap V_{\ell_2}^{\perp} = \big(V_{\ell_1} + V_{\ell_2}\big)^{\perp} = \{0\}.
\end{equation}
With this observation, it is now a simple matter to verify the dimension condition ii) above. 
\begin{itemize}
  \item If $\dim V = 4$ or $\dim V = 0$, then ii) is trivial.
    
  \item If $\dim V = 1$, then it suffices to show that $\dim\pi_{\ell} V = 1$ for at least two values of $\ell$. Suppose $\dim\pi_{\ell_1} V = \dim\pi_{\ell_2} V  = 0$ for some $1 \leq \ell_1 < \ell_2 \leq 4$, so that 
  \begin{equation*}
      V \subseteq \ker \pi_{\ell_1} \cap \ker \pi_{\ell_2} = V_{\ell_1}^{\perp} \cap V_{\ell_2}^{\perp}.
  \end{equation*} 
  However, in this case it follows from \eqref{dim cond 2} that $V = \{0\}$, which contradicts our dimension hypothesis. Thus, $\dim\pi_{\ell} V = 0$ for at most a single value of $\ell$, which more than suffices for our purpose.
  \item If $\dim V = 2$, then we may assume that $\dim \pi_{\ell_0} V = 0$ for some $1 \leq \ell_0 \leq 4$, since otherwise ii) is immediate. By dimensional considerations, it follows that $V = V_{\ell_0}^{\perp}$. Now let $1 \leq \ell \leq 4$ with $\ell \neq \ell_0$. By \eqref{dim cond 2}, it follows that $V \cap V_{\ell}^{\perp} = \{0\}$. Thus, by the rank-nullity theorem applied to the mapping $\pi_{\ell}|_V \colon V \to V_{\ell}$, we deduce that $\dim \pi_{\ell} V = 2$. Since this is true for three distinct values of $\ell$, property ii) holds.  
    \item If $\dim V = 3$, then it is clear that $\dim \pi_{\ell} V \geq 1$ for all $1 \leq \ell \leq 4$. Suppose there exist $1 \leq \ell_1 < \ell_2 \leq 4$ such that $\dim (\pi_{\ell_1} V) = \dim (\pi_{\ell_2} V) = 1$. In this case, by the rank-nullity theorem applied to $\pi_{\ell_i}|_V \colon V \to V_{\ell_i}$ and dimensional considerations,
    \begin{equation*}
        V_{\ell_1}^{\perp} + V_{\ell_2}^{\perp} = \ker \pi_{\ell_1} + \ker \pi_{\ell_2} \subseteq V.
    \end{equation*}
    However, in this case it follows from \eqref{dim cond 2} that $V = \R^4$, which contradicts our dimension hypothesis. Thus, $\dim \pi_{\ell} V = 1$ for at most a single value of $\ell$, and for the remaining values of $\ell$ the dimension is at least 2. This again more than suffices for our purpose. 
\end{itemize}
This establishes the finiteness of the Brascamp--Lieb constant and concludes the proof.
\end{proof}

Having established the multilinear restriction estimate, it is a simple matter to deduce the desired multilinear square function bound. 

\begin{proof}[Proof of Proposition~\ref{multi sf prop}] Let $B$ be a ball of radius $r^{-1}$ in $\R^4$ with centre $x_0$. Fix $\eta \in \mathcal{S}(\R^4)$ with $\supp \widehat{\eta} \subset B(0,1)$ and $|\eta(x)| \gtrsim 1$ on $B(0,1)$ and define $\eta_{B}(x) := \eta\big(r(x-x_0)\big)$. By the rapid decay of $\eta$, it suffices to show that 
  \begin{equation*}
 \Big\| \prod_{\ell=1}^4|\sum_{\theta \in \Theta(I_{\ell}; r)}f_{\theta}|^{1/4}\Big\|_{L^4(B)} \lesssim_{\varepsilon} r^{-\varepsilon} \prod_{\ell=1}^4  \Big\| \big(\sum_{\theta \in \Theta(I_{\ell}; r)}|f_{\theta}|^2\big)^{1/2}\, |\eta_B|^2\Big\|_{L^4(\R^4)}^{1/4}.
\end{equation*}
Indeed, once established, this inequality can be summed over a collection of finitely-overlapping balls $B$ which cover $\R^4$ to obtained the desired global estimate. 

For $1 \leq \ell \leq 4$ define 
\begin{equation*}
    F_{\ell} :=\sum_{\theta \in \Theta(I_{\ell}; r)}f_{\theta} \, \eta_B
\end{equation*}
so that each $F_{\ell}$ is Fourier supported in an $O(r)$-neighbourhood of $\Gamma_{I_{\ell}}$. Applying Proposition~\ref{mult rest prop} to these functions, it follows that 
  \begin{equation*}
  \Big\| \prod_{\ell=1}^4|\sum_{\theta \in \Theta(I_{\ell}; r)}f_{\theta}|^{1/4}\Big\|_{L^4(B)} \lesssim \Big\| \prod_{\ell=1}^4|F_{\ell}|^{1/4}\Big\|_{L^4(\R^4)} \lesssim_{\varepsilon} M(\lambda) r^{1-\varepsilon} \prod_{j=1}^4 \Big\|\sum_{\theta \in \Theta(I_{\ell}; r)}f_{\theta} \, \eta_B\Big\|_{L^2(\R^4)}^{1/4}.
\end{equation*}
Note that the functions $f_{\theta}\, \eta_B$ appearing in the right-hand sum have essentially disjoint Fourier support. Consequently, by Plancherel's theorem and H\"older's inequality,
\begin{align*}
 \Big\|\sum_{\theta \in \Theta(I_{\ell}; r)}f_{\theta} \, \eta_B\Big\|_{L^2(\R^4)}^{1/4} &\lesssim \Big\|\big(\sum_{\theta \in \Theta(I_{\ell}; r)}|f_{\theta}|^2\big)^{1/2}\,|\eta_B|\Big\|_{L^2(\R^4)} \\
  &\lesssim r^{-1} \Big\|\big(\sum_{\theta \in \Theta(I_{\ell}; r)}|f_{\theta}|^2\big)^{1/2}\,|\eta_B|^2\Big\|_{L^4(\R^4)}.
\end{align*}
Combining the previous two displays completes the proof. 
\end{proof}

\subsection{Rescaling} By combining Proposition~\ref{multilin ex prop}  with an affine rescaling argument, one may deduce a useful refined version of the multilinear inequality. This improves the dependence on separation parameter $\lambda$ under an additional localisation hypothesis on the intervals $J_1, \dots, J_4$.

Given dyadic scales $0 < \lambda_1 \leq \lambda_2 \leq 1$ and $J \in \fI(\lambda_2)$, let $\fI^4_{\mathrm{sep}}(J;\,\lambda_1)$ denote the collection of all $4$-tuples of intervals $\vec{I} = (I_1, \dots, I_4) \in \fI^4_{\mathrm{sep}}(\lambda_1)$ such that $I_\ell \subseteq J$ for all $1 \leq \ell \leq 4$. 

With this definition, the refined version of Proposition~\ref{multi sf prop}  reads as follows. 

\begin{corollary}\label{resc multilin sf cor} 
Fix dyadic scales $0 < r \leq \lambda_1 \leq \lambda_2 \leq 1$. If $J \in \fI(\lambda_2)$, $(I_1, \dots, I_4) \in \fI_{\mathrm{sep}}^4(J;\,\lambda_1)$ and $\varepsilon > 0$, then 
\begin{equation*}
 \Big\| \prod_{\ell=1}^4|\sum_{\theta \in \Theta(I_{\ell}; r)}f_{\theta}|^{1/4}\Big\|_{L^4(\R^4)} \lesssim_{\varepsilon} M(\lambda_1/\lambda_2) r^{-\varepsilon} \prod_{j=1}^4  \Big\| \big(\sum_{\theta \in \Theta(I_{\ell}; r)}|f_{\theta}|^2\big)^{1/2}\Big\|_{L^4(\R^4)}^{1/4}
\end{equation*}
holds whenever $(f_{\theta})_{\theta \in \Theta(r)}$ is a sequence of functions satisfying $\supp \widehat{f}_{\theta} \subseteq \theta$ for all $\theta \in \Theta(r)$.
\end{corollary}

\begin{proof} The result is a consequence of Proposition~\ref{multi sf prop} and a rescaling argument. Let $J = [\sigma - \lambda_2, \sigma + \lambda_2] \subseteq [-1,1]$ and recall the definition of the rescaled curve
\begin{equation*}
    g_{\sigma, \lambda_2}(\tilde{s}) := \big([g]_{\sigma, \lambda_2}\big)^{-1} ( g(\sigma + \lambda_2 \tilde{s}) - g(\sigma)).
\end{equation*}
Differentiating this expression, it follows that $g_{\sigma, \lambda_2}^{(j)}(\tilde{s}) = \lambda_2^j \big([g]_{\sigma, \lambda_2}\big)^{-1} g^{(j)}(\sigma + \lambda_2 \tilde{s})$ for $j \geq 1$ and so
\begin{equation*}
    [g_{\sigma, \lambda_2}]_{\tilde{s},\tilde{r}} = \big([g]_{\sigma, \lambda_2}\big)^{-1} \circ [g]_{s,r} \qquad \textrm{where $s = \sigma + \lambda_2 \tilde{s}$ and $r = \lambda_2 \tilde{r}$.}
\end{equation*}
From this and the definition \eqref{Frenet box 3}, it is not difficult to deduce that
\begin{equation*}
    [g_{\sigma, \lambda_2}]_{\mathcal{C}, \tilde{s},\tilde{r}} = \big([g]_{\mathcal{C},\sigma, \lambda_2}\big)^{-1} \circ [g]_{\mathcal{C},s,r}.
\end{equation*}
Suppose $\theta \in \Theta(J;\,r)$ and $\supp \widehat{F}_{\theta} \subseteq \theta$. If $\theta = \theta(s,r)$, then
\begin{equation*}
    \supp \widehat{F}_{\theta} \circ [g]_{\mathcal{C},\sigma, \lambda_2} \subseteq \tilde{\theta}(\tilde{s}, \tilde{r})
\end{equation*}
where $\tilde{\theta}(\tilde{s}, \tilde{r})$ is the $\tilde{r}$-plate centred at $\tilde{s}$ defined with respect to $\tilde{g} := g_{\sigma, \lambda_2}$.  Finally, note that the above rescaling maps the intervals $(I_1, \dots, I_4) \in \fI^4_{\mathrm{sep}}(J;\lambda_1)$ to intervals $(\tilde{I}_1, \dots \tilde{I}_4) \in \fI^4_{\mathrm{sep}}(\lambda_1/\lambda_2)$.
\end{proof}




\subsection{Broad/narrow analysis} Here arguments from \cite{Ham2014} are adapted to pass from the multilinear estimates of Proposition~\ref{multi sf prop} (or, more precisely, Corollary~\ref{resc multilin sf cor}) to the linear estimates in Theorem~\ref{reverse sf thm}.\medskip

The key ingredient is the following decomposition lemma, which follows by iteratively applying the decomposition scheme discussed in \cite{Ham2014}. 

\begin{lemma}\label{broad narrow lemma E} Let $\varepsilon > 0$ and $r > 0$. There exist dyadic numbers $C_\varepsilon \geq 1$, $r_{\mathrm{n}}$ and $r_{\mathrm{b}}$ satisfying
\begin{equation}\label{broad narrow equation E 1}
   r < r_{\mathrm{n}} \lesssim_{\varepsilon, 1} r, \qquad r < r_{\mathrm{b}} \leq 1
\end{equation} 
such that
\begin{equation}\label{brd nrw E}
      \big\|\sum_{\theta \in \Theta(r)} f_{\theta}\big\|_{L^4(\R^4)} \lesssim_{\varepsilon} r^{-\varepsilon} \Big( \sum_{I \in \fI(r_{\mathrm{n}})}\|f_I\|_{L^4(\R^4)}^4\Big)^{1/4} + r^{-\varepsilon} \Big(\sum_{\substack{J \in \fI(C_\varepsilon r_{\mathrm{b}}) \\ \vec{I} \in \fI^4_{\mathrm{sep}}(J;\,r_{\mathrm{b}})} }\big\|\prod_{\ell=1}^4|f_{I_{\ell}}|^{1/4}\big\|_{L^4(\R^4)}^4 \Big)^{1/4}
\end{equation}
holds whenever $(f_{\theta})_{\theta \in \Theta(r)}$ is a sequence of functions satisfying $\supp \widehat{f}_{\theta} \subseteq \theta$ for all $\theta \in \Theta(r)$.
\end{lemma} 

We provide a proof of (an abstract version of) the above lemma in Appendix~\ref{BG appendix} (more precisely, Lemma \ref{broad narrow lemma E} follows from applying Lemma \ref{broad narrow lemma} to the decomposition $f:= \sum_{\theta \in \Theta(r)} f_\theta$ for a fixed dyadic scale $0< r \leq 1$ and $\varepsilon>0$).

We are now in position to prove the desired reverse square function estimate.

\begin{proof}[Proof of Theorem~\ref{reverse sf thm}] Fix $0 < r \leq 1$ a choice of dyadic scale and $\varepsilon > 0$, and apply Lemma \ref{broad narrow lemma E}.  The analysis splits into two cases depending on which of the right-hand terms in \eqref{brd nrw E} dominates. We refer to the first term as the \textit{narrow} term and to the second term as the \textit{broad} term.

\subsubsection*{The narrow case} Suppose the narrow term dominates the right-hand side of \eqref{brd nrw E} in the sense that
\begin{equation*}
     \big\|\sum_{\theta \in \Theta(r)} f_{\theta}\big\|_{L^4(\R^4)} \lesssim_{\varepsilon} r^{-\varepsilon} \Big( \sum_{I \in \fI(r_{\mathrm{n}})}\big\|\sum_{\theta \in \Theta(I;\,r)} f_{\theta}\big\|_{L^4(\R^4)}^4\Big)^{1/4}.
\end{equation*}
This case is dealt with using a trivial argument. If $I \in \fI(r_{\mathrm{n}})$, then 
\begin{equation}\label{narrow 1}
    \big|\sum_{\theta \in \Theta(I;\,r)} f_{\theta}\big| \lesssim_{\varepsilon} \big( \sum_{\theta \in \Theta(I;\,r)}  |f_{\theta}|^2 \big)^{1/2}
\end{equation}
by Cauchy--Schwarz, since the condition $r_{\mathrm{n}} \sim r$ from \eqref{broad narrow equation E 1} implies that there are only $O_{\varepsilon}(1)$ intervals belonging to $\fI(I;\,r)$. Thus,
\begin{align*}
        \big\|\sum_{\theta \in \Theta(r)} f_{\theta}\big\|_{L^4(\R^4)} \lesssim_{\varepsilon} r^{-\varepsilon} \Big\|\big( \sum_{I \in  \fI(r_{\mathrm{n}})}\sum_{\theta \in \Theta(I;\,r)}  |f_{\theta}|^2\big)^{1/2}\Big\|_{L^4(\R^4)} = r^{-\varepsilon} \Big\|\big( \sum_{\theta \in \Theta(r)}  |f_{\theta}|^2\big)^{1/2}\Big\|_{L^4(\R^4)},
\end{align*}
where the first step follows from \eqref{narrow 1} and the embedding $\ell^2 \hookrightarrow \ell^4$ and the last step from the definition of $\mathfrak{I}(r_{\mathrm{n}})$ and $\Theta(I;r)$.

\subsubsection*{The broad case} Suppose the broad term dominates the right-hand side of \eqref{brd nrw E} in the sense that
\begin{equation*}
      \big\|\sum_{\theta \in \Theta(r)} f_{\theta}\big\|_{L^4(\R^4)} \lesssim_{\varepsilon}  r^{-\varepsilon} \Big(\sum_{\substack{J \in \fI(C_\varepsilon r_{\mathrm{b}}) \\ \vec{I} \in \fI^4_{\mathrm{sep}}(J;\,r_{\mathrm{b}})} }\big\|\prod_{\ell=1}^4\big|\sum_{\theta \in \Theta(I_{\ell}; r)} f_{\theta}\big|^{1/4}\big\|_{L^4(\R^4)}^4 \Big)^{1/4}.
\end{equation*}
This case is treated using the rescaled multilinear inequality from Corollary~\ref{resc multilin sf cor}. Since $\#\fI^4(J;\,r_{\mathrm{b}}) \lesssim_{\varepsilon} 1$ for each $J \in \fI(C_{\varepsilon}r_{\mathrm{b}})$, by H\"older's inequality
 \begin{equation*}
        \big\|\sum_{\theta \in \Theta(r)} f_{\theta}\big\|_{L^4(\R^4)} \lesssim_{\varepsilon} r^{-\varepsilon}
\Big(\sum_{J \in \fI(C_\varepsilon r_{\mathrm{b}})} \Big(\sum_{\vec{I} \in \fI^4_{\mathrm{sep}}(J;\,r_{\mathrm{b}})} \Big\|\prod_{\ell=1}^4 \big|\sum_{\theta \in \Theta(I_{\ell}; r)} f_{\theta}\big|^{1/4}\Big\|_{L^4(\R^4)}^{16}\Big)^{1/4} \Big)^{1/4}. 
\end{equation*}
Applying Corollary~\ref{resc multilin sf cor} with $\lambda_1 := r_{\mathrm{b}}$ and $\lambda_2 := C_\varepsilon r_{\mathrm{b}}$, one deduces that
 \begin{equation*}
\big\|\sum_{\theta \in \Theta(r)} f_{\theta}\big\|_{L^4(\R^4)} \lesssim_{\varepsilon} r^{-\varepsilon} \Big(\sum_{J \in \fI(C_\varepsilon r_{\mathrm{b}})} \Big(\sum_{\vec{I} \in \fI^4_{\mathrm{sep}}(J;\,r_{\mathrm{b}})} \prod_{\ell = 1}^4\Big\| \big(\sum_{\theta \in \Theta(I_{\ell};\,r)} |f_{\theta}|^2 \big)^{1/2} \Big\|_{L^4(\R^4)}^4\Big)^{1/4} \Big)^{1/4}. 
\end{equation*}
Relaxing the inner range of summation to all $\vec{I} \in \fI(J;\,r_{\mathrm{b}})^4$ (that is, dropping the separation condition), 
 \begin{equation*}
        \big\|\sum_{\theta \in \Theta(r)} f_{\theta}\big\|_{L^4(\R^4)} \lesssim_{\varepsilon} r^{-\varepsilon}  
\Big(\sum_{I \in \fI(r_{\mathrm{b}})} \Big\| \big(\sum_{\theta \in \Theta(I;\,r)} |f_{\theta}|^2 \big)^{1/2} \Big\|_{L^4(\R^4)}^4 \Big)^{1/4}. 
\end{equation*}
Arguing as in the last steps of the narrow case, using the embedding $\ell^2 \hookrightarrow \ell^4$, now concludes the argument. 
\end{proof}




\section{Proof of the forward square function inequality in \texorpdfstring{$\R^3$}{}}\label{forward SF sec}

In this section we establish the $L^2$ weighted forward square function estimate from Proposition~\ref{f SF prop}. Before we commence, it is useful to recall the basic setup. Let $\gamma \in \mathfrak{G}_3(\delta_0)$ for $0 < \delta_0 \ll 1$ and $\be_j:[-1,1] \to S^3$ for $1 \leq j \leq 3$ be the associated Frenet frame. Recall that this satisfies
\begin{equation}\label{fsq prelim 1}
    \be_j(s) = \vec{e}_j + O(\delta_0) \qquad \textrm{for $1 \leq j \leq 3$ and $s \in I_0=[-\delta_0,\delta_0]$,}
\end{equation}
where the $\vec{e}_j$ denote the standard basis vectors. For $0 < r \leq 1$, recall that a \textit{$(0,r)$-Frenet box} is a set of the form
\begin{equation*}
    \pi_{0,\gamma}(s;\,r) := \big\{ \xi \in \widehat{\R}^3 : |\inn{\be_1(s)}{\xi}| \leq r,  \,\, 1/2 \leq |\inn{\be_2(s)}{\xi}| \leq 1,  \,\, |\inn{\be_3(s)}{\xi}| \leq 1 \big\}
\end{equation*}
for some $s \in [-1,1]$. Proposition~\ref{f SF prop} concerns smooth frequency projections $\chi_{\pi}(D)$ where $\chi_{\pi}$ is a bump function adapted to a $(0,r)$-Frenet box $\pi$.




\subsection{Geometric observations}\label{fsq geo obs sec} We begin by reparametrising the sets $\pi_{0,\gamma}(s;r)$ using an argument similar to that of \S\ref{geo obs sec}. 
Define the functions $g_j: I_0 \to \R^3$ by $g_j(s) := -\be_{1j}(s)\, \be_{11}(s)^{-1}$ for $j = 2$, $3$ (note that $\be_{1,1}(s)$ is bounded away from $0$ by \eqref{fsq prelim 1}) so that
\begin{equation}\label{fsq Frenet 1}
    \inn{\be_1(s)}{\xi} = \be_{11}(s) \big( \xi_1 - \xi_2g_2(s) - \xi_3g_3(s) \big). 
\end{equation}
Thus, we have the containment property
\begin{equation}\label{fsq Frenet 2}
    \pi_{0, \gamma}(s;\,r) \subseteq \theta(s;\,C r)
\end{equation}
where $\theta(s;\,r)$ is the region
\begin{equation*}
    \theta(s;\,r) := \Big\{ \xi \in \widehat{\R}^3 : \big|\xi_1 - \sum_{j=2}^3\xi_jg_j(s)\big| < r \textrm{ and } 1/4 \leq |\xi_2| \leq 4, \, |\xi_3| \leq 4\Big\}. 
\end{equation*}
We refer to the sets $\theta(s;\,r)$ as `plates'.\medskip 

It is useful to note that the curves $g_j \colon I_0 \to \R^3$ satisfy a certain regularity condition. In particular, for each $\ba = (a_2, a_3) \in \R^2$ define the function $g_{\ba}(s) := a_2g_2(s) + a_3g_3(s)$. By differentiating \eqref{fsq Frenet 1} with respect to $s$ and evaluating the result at $\xi = (0, a_2, a_3)$, provided the parameter $\delta_0 > 0$ featured in \eqref{fsq prelim 1} is chosen sufficiently small, it follows that
\begin{equation}\label{fsq Frenet 3}
     |g_{\ba}'(s)| \sim 1 \qquad \textrm{for all $\ba  \in [1/4, 4] \times [-1,1]$.}
\end{equation}
Indeed, this is a simple consequence of the Frenet equations.\medskip

We also observe a dual version of the containment condition \eqref{fsq Frenet 2}. In particular, if we define the dual Frenet box and dual plate
\begin{align*}
\pi^*_{0,\gamma}(s;r) &:= \big\{x\in \R^3 : |\inn{\be_1(s)}{x}| \leq r^{-1} \textrm{ and } |\inn{\be_j(s)}{x}| \leq 1 \textrm{ for $j=2$, $3$} \big\}, \\
\theta^*(s;r) &:= \big\{x\in \R^3 : |x_1| \leq r^{-1} \textrm{ and } |x_j + g_j(s)x_1| \leq 4 \textrm{ for $j=2$, $3$} \big\},
\end{align*}
then it follows that $\pi^*_{0,\gamma}(s; r) \subseteq \theta^*(s; C^{-1} r)$. To this, we first observe the identity
\begin{equation}\label{fsq Frenet 4}
    \begin{bmatrix}
    \inn{x}{\be_2(s)} \\
    \inn{x}{\be_3(s)}
    \end{bmatrix}
    =
    \begin{bmatrix}
    \be_{22}(s) & \be_{23}(s) \\
    \be_{32}(s) & \be_{33}(s)
    \end{bmatrix}
    \begin{bmatrix}
    x_2 + g_2(s)x_1 \\
    x_3 + g_3(s)x_1
    \end{bmatrix},
\end{equation}
which follows from the orthogonality between the Frenet vectors $\big(\be_j(s)\big)_{j=1}^3$. Since the right-hand $2\times 2$ matrix is a small perturbation of the identity, the claimed containment property follows. 




\subsection{The iteration scheme}


Our proof of Proposition~\ref{f SF prop} uses an iteration argument. This is based on the approach of Carbery and the fourth author in \cite[Proposition 4.6]{CS1995}, where a related inequality for the C\'ordoba sectorial square function was obtained. Driving the iteration scheme is an elementary pointwise square function bound due to Rubio de Francia~\cite{RdF1983}. Here it is convenient to state a slight generalisation of this result.

\begin{lemma}\label{RdF lem} Let $\psi \in \mathscr{S}(\widehat{\R}^n)$, $A \in \mathrm{GL}(\R, n)$ and $G \colon \Z^m \to \R^n$. For all $N \in \N$ the pointwise inequality \begin{equation*}
\sum_{\nu \in \Z^m} \big|\psi\big(AD - G(\nu)\big)f(x)\big|^2 \lesssim_{\psi, N} \sup_{\nu_2 \in \Z^m} \sum_{\nu_1 \in \Z^m} e^{- |G(\nu_1) - G(\nu_2)|/2} \int_{\R^n}  |f(x - A^{\top} y)|^2 (1 + |y|)^{-N}\,\ud y \end{equation*}
holds for all $f \in \mathscr{S}(\R^n)$. 
\end{lemma}

\begin{proof} The case where $G \colon \Z^n \to \Z^n$ is the identity map is proven in \cite{RdF1983}. The argument can be generalised to prove the above lemma, by replacing an application of Plancherel's theorem with a $T^*T$ argument involving the Schur test. For convenience, the details of the argument are presented in Appendix~\ref{RdF appendix}.
\end{proof}

To describe the iteration step, we first define smooth cutoff functions adapted to the plates $\theta$ defined above. As usual, let $\eta \in C^{\infty}_c(\R)$ satisfy $\eta(u) = 1$ for $|u| \leq 1/2$ and $\supp \eta \subseteq [-1,1]$ and define the multipliers
\begin{equation}\label{m r nu}
 m_r^{\nu}(\xi) := \eta\Big(r^{-1}\big( \xi_1 - \sum_{j=2}^3 \xi_j g_j(s_{\nu})\big) \Big) \qquad \textrm{for $\nu \in \Z$ and $s_{\nu} := r \nu$.}
\end{equation}
Let $b(\xi) = \tilde{\beta}(4^{-1}\xi_2) \, \eta(4^{-1} \xi_3)$ where here $\tilde{\beta}$ is as defined in \eqref{chi pi} so that $(m_r^{\nu}\cdot b)(\xi) = 1$ if $\xi \in \theta(s_{\nu};r)$. 
For the iteration scheme, we in fact work with truncated versions of the plates. Given $K \geq 1$, $-1 \leq s \leq 1$, $0 < r \leq 1$ and $\ba=(a_2,a_3) \in \R^2$, consider the truncated plate 
\begin{equation*}
    \theta^{\ba, K}(s;\,r) := \Big\{ \xi \in \widehat{\R}^3 : \big|\xi_1 - \sum_{j=2}^3 \xi_j g_j(s)\big| \leq r  \textrm{ and } |\xi_j - a_j| \leq K^{-1} \textrm{ for $j=2$, $3$}\Big\}.
\end{equation*}
Correspondingly, we let $\zeta \in C^{\infty}_c(\R)$ satisfy $\supp \zeta \subseteq [-1,1]$ and $\sum_{k \in \Z} \zeta(\,\cdot - k) \equiv 1$ and decompose
\begin{equation}\label{fsq b dec}
    b = \sum_{\ba \in K^{-1}\Z^2} b_{\ba} \qquad \textrm{where} \qquad b_{\ba}(\xi) := \prod_{j=2}^3\zeta(K(\xi_j - a_j)\big) \, b(\xi).
\end{equation}

For $\br := (r_1, r_2,r_3) \in (0,1]^3$ and $s \in [-1,1]$ let  $T_{\be, \br}(s)$ denote the parallelepiped consisting of all vectors $x \in \R^3$ satisfying $|\inn{x}{\be_j(s)}| \leq r_j^{-1}$ for $1 \leq j \leq 3$. These sets should be thought of a scaled versions of the dual Frenet box $\pi_{0,\gamma}^*(s;r)$ introduced in \S\ref{fsq geo obs sec}. Consider the weighted averaging and Nikodym-type maximal operators associated to these sets, given by
\begin{equation}\label{wtd 3d Nik ops}
    \widetilde{\mathcal{A}}_{\be, \br} g (x;s) :=   \int_{\R^3}   g(x-y) \psi_{\,T_{\be,\br}(s)}(y) \, \ud y  \qquad \textrm{and} \qquad  \widetilde{\mathcal{N}}_{\be, \br} g(x) := \sup_{s \in [-1,1]} | \widetilde{\mathcal{A}}_{\be, \br} g (x;s)|
\end{equation}
where
\begin{equation}\label{it max fn}
 \psi_{\,T_{\be,\br}(s)}(x) :=  \big(\prod_{j=1}^3 r_j\big) \, \big(1 + \sum_{j=1}^3 r_j|\inn{\be_j(s)}{y}|\big)^{-300}.
\end{equation}
Here the subscript $\be$ refers to the Frenet frame $\be := (\be_1, \be_2, \be_3)$. 

With the above definitions, the key iteration step is as follows.

\begin{proposition}\label{it prop} Let $0 < r < 1$, $K \geq 1$, $\tilde{r} = Kr$, $\br:=(r, K^{-1}, K^{-1})$ and $\ba=(a_2,a_3) \in [1/4, 4] \times [-1,1]$. With the above definitions,
\begin{equation*}
\int_{\R^3}  \sum_{\nu \in \Z} \big|(m_r^{\nu} \cdot b_{\ba})(D)f(x)\big|^2 \,w(x)\,\ud x \lesssim \int_{\R^3} \sum_{\tilde{\nu} \in \Z} \big|(m_{\tilde{r}}^{\tilde{\nu}}\cdot b_{\ba})(D)f(x)\big|^2\,  \widetilde{\mathcal{N}}_{\be, \br} \circ \widetilde{\mathcal{N}}_{\be, \br}\,  w(x)\,\ud x
\end{equation*}
for any non-negative $w \in L^1_{\mathrm{loc}}(\R^3)$.
\end{proposition}

\begin{proof} The proof is based on the following simple geometric observation, which motivates the use of the truncation. If $|s - \tilde{s}| \leq K r$, then the plates $\theta^{\ba, K}(s;\,r)$, $\theta^{\ba, K}(\tilde{s};\,r)$ are essentially parallel translates of one another. More precisely, if $\xi \in \theta^{\ba, K}(s;\,r)$, then 
\begin{equation*}
    \Big| \xi_1 - \sum_{j=2}^3a_j\big(g_j(s) - g_j(\tilde{s})\big) - \sum_{j=2}^3 \xi_j g_j(\tilde{s})\Big| \leq \Big|\xi_1 - \sum_{j=2}^3 \xi_j g_j(s)\Big| + \sum_{j=2}^3 |a_j - \xi_j||g_j(s) - g_j(\tilde{s})| \lesssim_{\bg} r
\end{equation*}
and, consequently, there exists some constant $C_{\bg}$ such that
\begin{equation}\label{sf it 1} 
    \theta^{\ba, K}(s;\,r) - \sum_{j=2}^3a_j\big(g_j(s) - g_j(\tilde{s})\big)\vec{e}_1 \subseteq \theta^{\ba, K}\big(\tilde{s};\,C_{\bg} r\big).
\end{equation}
In light of this observation, define the multipliers
\begin{equation}\label{sf it 2} 
    m^{\tilde{\nu}, \nu}_{\tilde{r}, r}(\xi) :=
    \frac{\eta\Big((2C_{\bg}r)^{-1}\Big( \xi_1 - \sum_{j=2}^3a_j\big(g_j(s_{\nu}) - g_j(\tilde{s}_{\tilde{\nu}})\big) - \sum_{j=2}^3 \xi_j g_j(\tilde{s}_{\tilde{\nu}})\Big) \Big)\, \tilde{b}_{\ba}(\xi)}
    {\sum_{i=-1}^1 m_{\tilde r}^{\tilde \nu+i} (\xi)}
\end{equation}
for $\tilde{\nu}$, $\nu \in \Z$ and $\tilde{s}_{\tilde{\nu}} := \tilde{r} \tilde{\nu}$, $s_{\nu} := r \nu$ and $\tilde{r}=Kr$, where $ \tilde{b}_{\ba}(\xi) := \prod_{j=2}^3\eta(K(\xi_j - a_j))$ so that $b_{\ba} = \tilde{b}_{\ba} \cdot b_{\ba}$. Thus, in view of \eqref{sf it 1}, we have 
\begin{equation} \label{reproducing-property} 
    m_r^{\nu} \cdot b_{\ba} =  m^{\tilde{\nu}, \nu}_{\tilde{r}, r} \cdot m_r^{\nu} \cdot b_{\ba} 
    \sum_{i=-1}^1 m_{\tilde r}^{\tilde \nu+i} (\xi) 
    \qquad \textrm{whenever $|s_{\nu} - \tilde{s}_{\tilde{\nu}}| \leq Kr=:\tilde{r}$.}
\end{equation}
Furthermore, since for fixed $\tilde{\nu}$ the multipliers $m_{\tilde{r},r}^{\tilde{\nu},\nu}$ correspond to essentially parallel frequency regions for $|s_\nu - \tilde{s}_{\tilde{\nu}}| \leq  5\tilde{r}$, Lemma~\ref{RdF lem} implies they satisfy a weighted $L^2$ inequality. Indeed, recall from \eqref{fsq Frenet 3} that the functions $g_{\ba}(s) := a_2g_2(s) + a_3g_3(s)$ satisfy the uniform regularity condition $|g_{\ba}'(s)| \sim 1$; recall that $\ba=(a_2,a_3) \in [1/4,4] \times [-1,1]$. From this we deduce that
\begin{equation*}
    \sup_{\nu_2 \in \Z} \sum_{\nu_1 \in \Z} e^{-r^{-1}|g_{\ba}(r \nu_1) - g_{\ba}(\tilde{r} \nu_2)|/2} \lesssim 1,
\end{equation*}
where the above inequality holds with a constant uniform in both $r$ and $\ba$. Thus, recalling the definition of the multipliers $m_{\tilde{r},r}^{\tilde{\nu}, \nu}$ from \eqref{sf it 2}, Lemma~\ref{RdF lem} implies that for fixed $\tilde{\nu} \in \Z$,
\begin{equation}\label{sf it 3}
    \int_{\R^3} \sum_{\substack{ \nu \in \Z \\ |s_{\nu} - \tilde{s}_{\tilde{\nu}}|\leq 5 \tilde{r}} } |m_{\tilde{r},r}^{\tilde{\nu}, \nu}(D)f(x)|^2 w(x)\,\ud x \lesssim \int_{\R^3} |f(x)|^2 \widetilde{\mathcal{N}}_{\be,\br} w(x)\,\ud x;
\end{equation}
indeed the inequality holds with $\widetilde{\mathcal{N}}_{\be,\br} w(x)$ replaced by the single average $\widetilde{\mathcal{A}}_{\be,\br}w(x; \tilde{s}_{\tilde{\nu}})$, but there is no loss in taking supremum over $s \in [-1,1]$ in view of other appearances of $\widetilde{\mathcal{N}}_{\be,\br}$ (see \eqref{first step} below).
From \eqref{reproducing-property} we get 
\begin{equation*}
      \sum_{\nu \in \Z} \big|(m_r^{\nu} \cdot b_{\ba})(D)f(x)\big|^2 \lesssim  \sum_{\substack{\tilde{\nu}, \nu \in \Z \\ |s_{\nu} - \tilde{s}_{\tilde{\nu}}|\leq 5 \tilde{r}} } \big|(m_r^{\nu}\cdot \tilde{b}_{\ba})(D) \circ  m^{\tilde{\nu}, \nu}_{\tilde{r}, r}(D) \circ (m_{\tilde{r}}^{\tilde{\nu}}\cdot b_{\ba})(D)f(x)\big|^2. 
\end{equation*}
By the Schwartz decay property of $\widecheck{\eta}$, the convolution kernel associated to the multiplier operator $(m_r^{\nu} \cdot \tilde{b}_{\ba})(D)$ satisfies
\begin{equation*}
    |(m_r^{\nu} \cdot \tilde{b}_{\ba})\;\widecheck{}\;(x)| \lesssim_N rK^{-2} \, \big(1 + r|x_1| + K^{-1}\sum_{j=2}^3|x_j + x_1 g_j(s)|\big)^{-100} \lesssim \psi_{\,T_{\be,\br}(s)}(x)
\end{equation*}
where the function $\psi_{\,T_{\be,\br}(s)}(x)$ is the $L^1$-normalised smooth cutoff defined in \eqref{it max fn}. To justify the second inequality in the above display we use \eqref{fsq Frenet 4}, which allows us to deduce that $\sum_{j=2}^3|x_j + x_1g_j(s)| \gtrsim \sum_{j=2}^3|\inn{\be_j(s)}{x}|$. Combining the preceding observations with a simple Cauchy--Schwarz and Fubini argument, 
\begin{equation}\label{first step}
    \int_{\R^3}  \sum_{\nu \in \Z} \big|(m_r^{\nu} \cdot b_{\ba})(D)f(x)\big|^2 \,w(x)\,\ud x \lesssim  \sum_{\tilde{\nu} \in \Z} \, \int_{\R^3} \! \sum_{\substack{\nu \in \Z \\ |s_{\nu} - \tilde{s}_{\tilde{\nu}}|\leq 5 \tilde{r}} } \big|m^{\tilde{\nu}, \nu}_{\tilde{r}, r}(D) \circ (m_{\tilde{r}}^{\tilde{\nu}}\cdot b_{\ba})(D)f(x)\big|^2 \widetilde{\mathcal{N}}_{\be,\br} w(x)\,\ud x.
\end{equation}
On the other hand, \eqref{sf it 3} implies
\begin{equation*}
   \int_{\R^3} \sum_{\substack{\nu \in \Z \\ |s_{\nu} - \tilde{s}_{\tilde{\nu}}|\leq 5 \tilde{r}} } \big|m^{\tilde{\nu}, \nu}_{\tilde{r}, r}(D) \circ (m_{\tilde{r}}^{\tilde{\nu}}\cdot b_{\ba})(D)f(x)\big|^2 \widetilde{\mathcal{N}}_{\be,\br}\, w(x)\,\ud x \lesssim \int_{\R^3} \big|(m_{\tilde{r}}^{\tilde{\nu}}\cdot b_{\ba})(D)f(x)\big|^2 \widetilde{\mathcal{N}}_{\be,\br} \circ \widetilde{\mathcal{N}}_{\be,\br}\, w(x)\,\ud x.
\end{equation*}
The two previous displays combine to give the desired estimate. 
\end{proof}




\subsection{Proof of the \texorpdfstring{$L^2$}{}-weighted estimate}\label{L2 wtd proof subsec} Lemma~\ref{it prop} is now repeatedly applied to prove Proposition~\ref{f SF prop}. 

\begin{proof}[Proof of Proposition~\ref{f SF prop}] First observe that by the definition of $\pi$ in \eqref{chi pi}, the containment property \eqref{fsq Frenet 2} and the definition of $m_r^{\nu}$ in \eqref{m r nu}, for each $\pi \in \mathcal{P}_0(r)$ there is an associated $\nu \in \Z$ such that $m_r^\nu(\xi)=1$ for $\xi \in \supp \chi_\pi$. Thus, a simple Cauchy--Schwarz and Fubini argument yields 
\begin{equation*}
    \int_{\R^3} \sum_{\pi \in \mathcal{P}_0(r)} |\chi_{\pi}(D)f(x)|^2 w(x)\,\ud x \lesssim \int_{\R^3} \sum_{\nu \in \Z} |(m_r^{\nu}\cdot b)(D)f(x)|^2 \widetilde{\mathcal{N}}_{\be, \br_*}w(x)\,\ud x,  
\end{equation*}
where $\br_* := (r, 1, 1)$. Take $K := r^{-\varepsilon/8}$ and decompose $b = \sum_{\,\ba \in K^{-1} \Z^2} b_{\ba}$ as in \eqref{fsq b dec}. By a pigeonholing, it follows that there exists a choice of $\ba \in [1/4,4]\times [-1,1]$ satisfying
\begin{equation*}
    \int_{\R^3} \sum_{\pi \in \mathcal{P}_0(r)} |\chi_{\pi}(D)f(x)|^2 w(x)\,\ud x \lesssim r^{-\varepsilon/2} \int_{\R^3} \sum_{\nu \in \Z} |(m_r^{\nu}\cdot b_{\ba})(D)f(x)|^2 \widetilde{\mathcal{N}}_{\be, \br_*}w(x)\,\ud x.  
\end{equation*}

Define the sequence
\begin{equation*}
    \br_M:=(r_M, K^{-1}, K^{-1}) \quad  \text{where $\,\, r_M:= K^M r\,\,$ for $\,\,M \geq 0$}
\end{equation*}
and recursively define a sequence of maximal operators by 
\begin{equation*}
\widetilde{\mathcal{N}}_{\be, \br}^{\, 0}:=\widetilde{\mathcal{N}}_{\be, \br_0} \circ \widetilde{\mathcal{N}}_{\be, \br_0} \circ \widetilde{\mathcal{N}}_{\be, \br_*} \qquad \text{ and } \qquad 
\widetilde{\mathcal{N}}_{\be, \br}^M:= \widetilde{\mathcal{N}}_{\be, \br_{M}} \circ \widetilde{\mathcal{N}}_{\be, \br_{M}}  \circ \widetilde{\mathcal{N}}_{\be, \br}^{M-1} \quad  \text{ for $M \geq 1$.}
\end{equation*}
We now repeatedly apply Proposition~\ref{it prop} to deduce that 
\begin{equation}\label{fsq 1}
    \int_{\R^3} \sum_{\nu \in \Z} |(m_r^{\nu} \cdot b_{\ba})(D)f(x)|^2 \widetilde{\mathcal{N}}_{\be, \br_*}w(x)\,\ud x \leq C^{M} \int_{\R^3} \sum_{\nu \in \Z} |(m_{r_M}^{\nu} \cdot b_{\ba})(D)f(x)|^2 \widetilde{\mathcal{N}}_{\be, \br}^{M-1}w(x)\,\ud x,
\end{equation}
provided $r_M \leq 1$. In particular, if $M := \floor{8/\varepsilon} - 1$, then $r^{\varepsilon/8}\leq r_M \leq 1$ and, consequently, there are only $O(r^{-\varepsilon/8})$ values of $\nu$ which contribute to the right-hand sum in \eqref{fsq 1}. Thus, one readily deduces that
\begin{equation*}
    \int_{\R^3} \sum_{\nu \in \Z} |(m_{r_M}^{\nu}\cdot b_{\ba})(D)f(x)|^2 \widetilde{\mathcal{N}}^{M-1}_{\be, \br}w(x)\,\ud x \lesssim r^{-\varepsilon/8} \int_{\R^3} |f(x)|^2 \widetilde{\mathcal{N}}_{\gamma,r}^{\,(\varepsilon)} \,  w(x)\,\ud x
\end{equation*}
where $\widetilde{\mathcal{N}}_{\gamma,r}^{\, (\varepsilon)} := \widetilde{\mathcal{N}}_{\be, \br_M} \circ \widetilde{\mathcal{N}}^{M-1}_{\be, \br}$. Combining the preceding observations concludes the proof of the $L^2$ weighted inequality, with the above choice of maximal operator.

It remains to show that the iterated maximal operator $\widetilde{\mathcal{N}}_{\gamma,r}^{\, (\varepsilon)}$ satisfies the $L^2$ bound from \eqref{f SF eq}. However, this is an immediate consequence of Proposition~\ref{3d Nik prop} of the following subsection.
\end{proof}




\subsection{Boundedness of the maximal functions}
From the proof of Proposition~\ref{f SF prop}, we see that the maximal function $\widetilde{\mathcal{N}}_{\gamma,r}^{\, (\varepsilon)}$ is obtained by repeatedly composing operators of the form $\widetilde{\mathcal{N}}_{\be, \br}$, as defined in \eqref{wtd 3d Nik ops}, where:
\begin{itemize}
    \item The family of curves $\be$ corresponds to the Frenet frame $(\be_1, \be_2, \be_3)$ associated to $\gamma$;
    \item The scales $\br = (r_1, r_2, r_3)$ depend on $r$ and $\varepsilon$ and vary over the different factors of the composition. Each featured tuple $\br = (r_1, r_2, r_3)$ satisfies
    \begin{equation*}
        \mathrm{ecc}(\br) \leq r^{-1}
    \end{equation*}
    where the \textit{eccentricity} $ \mathrm{ecc}(\br)$ is the ratio of $\max_j r_j$ and $\min_j r_j$.
\end{itemize}
In particular, to prove the $L^2$ bound \eqref{f SF eq} it suffices to show that, for all $\varepsilon_\circ>0$,
\begin{equation}\label{3d Nik reduction}
    \|\widetilde{\mathcal{N}}_{\be, \br}\|_{L^2(\R^3) \to L^2(\R^3)} \lesssim_{\varepsilon_{\circ}} \mathrm{ecc}(\br)^{\varepsilon_{\circ}}.
\end{equation}

To prove \eqref{3d Nik reduction}, we will in fact work with a more general setup, replacing $\be$ with a general family of smooth curves in $\R^n$ satisfying a non-degeneracy hypothesis. Let $\be := (\be_1, \dots, \be_n)$ where $\be_j \colon [-1,1] \to S^{n-1}$ is a smooth curve in the unit sphere in $\R^n$ for $1 \leq j \leq n$. Suppose these curves satisfy 
\begin{equation*}
    \big| \bigwedge_{j=1}^n \be_j(s) \big| \gtrsim 1 \qquad \textrm{for all $s \in [-1,1]$.}
\end{equation*}
Note that the $\be_j$ notation was previously reserved for the Frenet frame. In applications, we always take the $\be_j$ to be the Frenet vectors, and therefore there should be no conflict in the above choice of notation.

Given a tuple $\br := (r_1,\dots, r_n) \in (0,\infty)^n$ and $s \in [-1,1]$ define the parallelepiped
\begin{equation*}
    T_{\be, \br}(s) := \Big\{ x \in \R^n : x = \sum_{j=1}^n \lambda_j \be_j(s) \textrm{ where } \lambda_j \in [-r_j^{-1}, r_j^{-1}] \textrm{ for $1 \leq j \leq n$} \Big\}. 
\end{equation*}
Associated to these sets are the averaging operators and the maximal operator
\begin{equation}\label{Nik ops}
    \mathcal{A}_{\be,\br}f(x;s) := \fint_{T_{\be, \br}(s)} f(x-y)\,\ud y \quad \textrm{and} \quad \mathcal{N}_{\be,\br}f(x) := \sup_{s \in [-1,1]} |\mathcal{A}_{\be,\br}f(x;s)|
\end{equation}
defined for $f \in L^1_{\mathrm{loc}}(\R^n)$. The $\mathcal{N}_{\be,\br}$ satisfy favourable $L^2$ estimates. 

\begin{proposition}\label{3d Nik prop} With the above definitions, for all $\varepsilon > 0$ we have the norm bound
\begin{equation*}
    \|\mathcal{N}_{\be,\br}f\|_{L^2(\R^n) \to L^2(\R^n)} \lesssim_{\be, \varepsilon} \mathrm{ecc}(\br)^{\varepsilon},
\end{equation*}
where the \textit{eccentricity} $\mathrm{ecc}(\br) \geq 1$ is defined to be the ratio of $\max_j r_j$ and $\min_j r_j$. 
\end{proposition}

This proposition is based on a classical maximal bound due to C\'ordoba~\cite{Cordoba1982}. The details of the proof are provided below.\medskip

We generalise the weighted operators introduced in \eqref{wtd 3d Nik ops} by setting
\begin{equation}\label{wtd Nik ops}
     \widetilde{\mathcal{A}}_{\be, \br} f (x;s) :=   \int_{\R^n}   f(x-y) \psi_{\,T_{\be,\br}(s)}(y) \, \ud y \quad \textrm{and} \quad \widetilde{\mathcal{N}}_{\be, \br} f(x) := \sup_{s \in [-1,1]} | \widetilde{\mathcal{A}}_{\be, \br} f (x;s)|  
\end{equation}
where $\psi_{\,T_{\be,\br}(s)}$ is a smooth weight function adapted to the parallelepiped $T_{\be,\br}(s)$, given by 
\begin{equation}\label{gen it max fn}
 \psi_{\,T_{\be,\br}(s)}(y) :=  \big(\prod_{j=1}^n r_j\big) \, \big(1 + \sum_{j=1}^n r_j|(\bm{E}(s)^{-1} y)_j|\big)^{-100n}
\end{equation}
where $\bm{E}(s)$ denotes the $n \times n$ matrix whose $j$th column is $\be_j(s)$ for $1 \leq j \leq n$. If $(\be_j(s))_{j=1}^n$ forms an orthonormal frame, then $(\bm{E}(s)^{-1}y)_j = (\bm{E}(s)^{\top} y)_j = \inn{\be_j(s)}{y}$ and so \eqref{gen it max fn} generalises the definition \eqref{it max fn}. Note that the operators in \eqref{wtd Nik ops} correspond to weighted version of the averaging operator and Nikodym maximal function in \eqref{Nik ops}. Moreover, by dominating $ \psi_{\,T_{\be,\br}(s)}$ by a weighted sum of characteristic functions, it is clear that Proposition~\ref{3d Nik prop} implies analogous $L^2$ bounds for the $\widetilde{\mathcal{N}}_{\be, \br}$ operators.\medskip

In view of the preceding discussion, the estimate \eqref{f SF eq} for the maximal function $\widetilde{\mathcal{N}}_{\gamma,r}^{\, (\varepsilon)}$ appearing in Proposition~\ref{f SF prop} follows as a consequence of Proposition~\ref{3d Nik prop}.

\begin{proof}[Proof of Proposition~\ref{3d Nik prop}] Write $R := \mathrm{ecc}(\br)$ and let $\varepsilon > 0$ be given. We begin with some basic reductions. By pigeonholing, it suffices to show 
\begin{equation*}
    \|\mathcal{N}_{\be,\br}\|_{L^2(\R^n) \to L^2(\R^n)} \lesssim_{\varepsilon} R^{\varepsilon/2} 
\end{equation*}
where now the maximal operator $\mathcal{N}_{\be,\br}$ is redefined so that the supremum is taken over some subinterval $I_{\varepsilon} \subseteq [-1,1]$ of length $R^{-\varepsilon/2}$ rather than the whole of $[-1,1]$. Furthermore, if $|s_1-s_2| \leq R^{-1}$, then $T_{\be, \br}(s_1)$ and $T_{\be, \br}(s_2)$ define essentially the same parallelepiped, and therefore we may further restrict the supremum to some dyadic $R^{-1}$-net $\mathfrak{S}_{\varepsilon}$ in $I_{\varepsilon}$.\medskip

Let $a \in [-1,1]$ denote the centre of the interval $I_{\varepsilon}$ and $N := \ceil{1/\varepsilon}$. For $1 \leq j \leq n$ let $p_j$ denote the degree $N-1$ Taylor polynomial of $\be_j$ centred at $a$ and define $\bm{p} := (p_1, \dots, p_n)$. By Taylor's theorem, 
\begin{equation*}
    |p_j (s) - \be_j(s)| \lesssim_{\gamma} R^{-N\varepsilon} \leq R^{-1} \qquad \textrm{for all $s \in I_{\varepsilon}$}
\end{equation*}
and therefore there exists a constant $C \geq 1$, independent of $\br$, such that
\begin{equation*}
    T_{\bm{p}, C^{-1}\br}(s) \subseteq T_{\be, \br}(s) \subseteq T_{\bm{p}, C \br}(s) \qquad \textrm{for all $s \in I_{\varepsilon}$.}
\end{equation*}
In light of this observation, henceforth we may assume without loss of generality that the $\be_j$ are all polynomial mappings. Under this hypothesis, the $\be_j$ no longer map into the sphere; however, we may assume that over the domain $I_{\varepsilon}$ they map into, say, a $1/10$-neighbourhood of $S^{n-1}$.\medskip

Since the operators are all positive, it suffices to show
\begin{equation*}
    \|\sup_{s \in \mathfrak{S}_{\varepsilon}} |\mathcal{A}_{\be,\br}f(\,\cdot\,;s)| \|_{L^2(\R^n)} \lesssim_{\varepsilon} R^{\varepsilon} \|f\|_{L^2(\R^n)}
\end{equation*}
for all $f \in L^2(\R^n)$ continuous and non-negative. Fixing such an $f$, define the averages
\begin{equation*}
    \mathcal{A}_{\omega, r} f(x) := \int_{\R} f(x - t\omega) \chi_r(t)\,\ud t  \qquad \textrm{for $\omega \in \R^n$ with $\big||\omega| - 1\big| < 1/10$ and $r > 0$,}
\end{equation*}
where $\chi_r(t) := r^{-1}\chi_1(r^{-1}t)$ for some $\chi_1 \in C^{\infty}_c(\R)$ non-negative which satisfies $\chi_1(s) = 1$ for $|s| \leq 1$. Thus, by the Fubini--Tonelli theorem,
\begin{equation}\label{3d Nik 1}
    \mathcal{A}_{\be,\br}f(x;s) \lesssim \mathcal{A}_{\be_n(s), r_n} \circ \cdots \circ \mathcal{A}_{\be_1(s), r_1}f(x).
\end{equation}
Writing $\mathcal{A}_{\be_j}f(x;s) := \mathcal{A}_{\be_j(s), 1}f(x)$, we may combine \eqref{3d Nik 1} with a simple scaling argument the reduce to problem to showing
\begin{equation}\label{3d Nik 2}
    \|\sup_{s \in \mathfrak{S}_{\varepsilon}} |\mathcal{A}_{\be_j}f(\,\cdot\,;s)|\|_{L^2(\R^n)} \lesssim (\log R) \, \|f\|_{L^2(\R^n)} \qquad \textrm{for $1 \leq j \leq n$.}  
\end{equation}
The previous display is essentially a consequence of a maximal estimate proved in \cite[p.223]{Cordoba1982}. There similar maximal operators are considered for smooth curves $\gamma \colon [-1,1] \to S^{n-1}$ under the key hypothesis that $\gamma$ cross any affine hyperplane a bounded number of times. Since we are considering polynomial curves $\be_j$, the fundamental theorem of algebra ensures either:
\begin{enumerate}[a)]
    \item The curve $\be_j$ crosses any affine hyperplane a bounded number of times, where the bound depends on the degrees of the component polynomials, or
    \item There exists an affine hyperplane which contains the image of $\be_j$.
\end{enumerate}
In the former case, we may deduce \eqref{3d Nik 2} directly through appeal to the result from \cite[p.223]{Cordoba1982}.\footnote{It is remarked that the argument in \cite{Cordoba1982} carries through for a curve which maps into a $1/10$-neighbourhood of the sphere (rather than the sphere itself), provided the curve satisfies the finite crossing property.} In the latter case, we may apply the maximal bound from  \cite{Cordoba1982} over a lower dimensional affine subspace and combine this with a Fubini argument to again deduce the desired result. 
\end{proof}




\subsection{Scaling properties}\label{Nik scale subsec} We conclude this section with a discussion of the scaling properties of the maximal function $\widetilde{\mathcal{N}}_{\gamma,r}^{\, (\varepsilon)}$ and, in particular, fill in the gap in proof of Proposition~\ref{L4 forward SF prop} by proving the Claim therein.\medskip

We begin by introducing a general setup for rescaling the operators $\widetilde{\mathcal{N}}_{\be, \br}$ when defined with respect to a Frenet frame; as in the previous subsection, here we work in general dimensions. Fix $\gamma \colon [-1,1] \to \R^n$ a non-degenerate curve with $\gamma \in \mathfrak{G}(\delta)$ and $\sigma \in [-1,1]$, $0 < \lambda < 1$ be such that $[\sigma - \lambda, \sigma + \lambda] \subseteq [-1,1]$. Consider the rescaled curve
\begin{equation*}
    \gamma_{\sigma, \lambda}(\tilde{s}) := \big([\gamma]_{\sigma, \lambda}\big)^{-1}\big(\gamma(\sigma + \lambda \tilde{s}) - \gamma(\sigma)\big)
\end{equation*}
as defined in Definition~\ref{rescaled curve def}. Let $\be = (\be_1, \dots, \be_n)$ denote the Frenet frame defined with respect to $\gamma$ and $\tilde{\be} = (\tilde{\be}_1, \dots, \tilde{\be}_n)$ denote the Frenet frame defined with respect to $\widetilde{\gamma} := \gamma_{\sigma, \lambda}$. We suppose $\br = (r_1, \dots, r_n) \in (0,1]^n$ satisfies 
\begin{equation}\label{Nik scale hyp}
    r_i \leq \lambda r_{i+1} \qquad \textrm{for $1 \leq i \leq n-1$}
\end{equation}
and define $\tilde{\br} := D_{\lambda} \cdot \br$ where $D_{\lambda} := \mathrm{diag}(\lambda, \dots, \lambda^n)$ is as in \eqref{gamma transformation}.\medskip

\begin{lemma}\label{gen Nik scale lem} If $f \in L^1_{\mathrm{loc}}(\R^n)$ is non-negative, then, with the above definitions,
\begin{equation}\label{gen Nik scale eq}
    \big([\gamma]_{\sigma,\lambda}\big)^{-1} \circ \widetilde{\mathcal{N}}_{\tilde{\be}, \tilde{\br}} \circ [\gamma]_{\sigma,\lambda} \cdot  f(x) \lesssim_{\gamma} \widetilde{\mathcal{N}}_{\be, \br} f(x) \qquad \textrm{for all $x \in \R^n$}.
\end{equation}
\end{lemma}

Here we think of a matrix $M \in \mathrm{GL}(\R, n)$ as acting on $L^2(\R^n)$ by $M \cdot f := f\circ M$ for all $f \in L^2(\R^n)$. Thus, the left-hand side corresponds to the operator $\widetilde{\mathcal{N}}_{\tilde{\be}, \tilde{\br}}$ conjugated by the invertible operator $[\gamma]_{\sigma,\lambda} \colon L^2(\R^n) \to L^2(\R^n)$.\medskip

Before presenting the proof of Lemma~\ref{gen Nik scale lem}, we use the result to verify the rescaling step in the proof of Proposition~\ref{L4 forward SF prop}. In view of the discussion in \S\ref{L2 wtd proof subsec} and by a simple rescaling argument, we know that the maximal function\footnote{Recall, in the setup in Proposition~\ref{L4 forward SF prop} we have $\widetilde{\gamma} := \gamma_{\sigma, \lambda}$, where $\sigma := 2^{-\ell} \mu$ and $\lambda := 2^{-\ell}$, and $\tilde{r} := 2^{-(k-3\ell)/2}$.}
\begin{equation*}
    \widetilde{\mathcal{N}}^{\,\mu, (\varepsilon)}_{k,\ell} := \mathrm{Dil}_{2^{k-3\ell}} \circ \widetilde{\mathcal{N}}^{(\varepsilon)}_{\widetilde{\gamma},\tilde{r}} \circ \mathrm{Dil}_{2^{-(k-3\ell)}} 
\end{equation*}
 corresponds to a repeated composition of operators of the form $\widetilde{\mathcal{N}}_{\tilde{\be}, \tilde{\br}}$ where the $\tilde{\br} = (\tilde{r}_1, \tilde{r}_2, \tilde{r}_3)$ satisfy 
 \begin{equation*}
     \tilde{r}_1 \leq \tilde{r}_2 \leq \tilde{r}_3 \qquad \textrm{and} \qquad \mathrm{ecc}(\tilde{\br}) \lesssim 2^{(k-3\ell)/2}.
 \end{equation*} 
 Consequently, by Lemma~\ref{gen Nik scale lem}, the conjugate\
 \begin{equation*}
     \big([\gamma]_{\sigma,\lambda}\big)^{-1} \circ \widetilde{\mathcal{N}}^{\,\mu, (\varepsilon)}_{k,\ell} \circ [\gamma]_{\sigma,\lambda}
 \end{equation*}
 is dominated by a maximal function $\widetilde{\mathcal{N}}^{(\varepsilon)}_{k,\ell}$ given by a repeated composition of operators of the form $\widetilde{\mathcal{N}}_{\be, \br}$ where each $\br = (r_1, r_2, r_3)$ satisfies 
 \begin{equation*}
     r_1 \leq \lambda r_2 \leq \lambda^2 r_3 \qquad \textrm{and} \qquad \mathrm{ecc}(\br) \lesssim 2^{(k+\ell)/2}.
 \end{equation*}
 Furthermore, there are only $O_{\varepsilon}(1)$ factors in this composition. The just given definition for $\widetilde{\mathcal{N}}^{(\varepsilon)}_{k,\ell}$ is independent of $\mu$ and, by Proposition~\ref{3d Nik prop}, for all $\varepsilon_{\circ} > 0$ the operator $\widetilde{\mathcal{N}}^{(\varepsilon)}_{k,\ell}$ is bounded on $L^2(\R^3)$ with operator norm $O_{\varepsilon}(2^{\varepsilon k})$. Thus, we have verified all the outstanding claims in the proof of Proposition~\ref{L4 forward SF prop}.

\begin{proof}[Proof of Lemma~\ref{gen Nik scale lem}] Consider the conjugated operator on the left-hand side of \eqref{gen Nik scale eq}. By applying a change of variables to the integral defining the underlying averages, the problem is quickly reduced to the pointwise estimate  
\begin{equation*}
    |\det [\gamma]_{\sigma, \lambda}|^{-1} \cdot \psi_{T_{\tilde{\be}, \tilde{\br}}(\tilde{s})} \circ  \big([\gamma]_{\sigma, \lambda}\big)^{-1}(y) \lesssim  \psi_{T_{\be, \br}(s)} (y)
\end{equation*}
for the weight functions as defined in \eqref{gen it max fn}, where $s= \sigma + \lambda \tilde{s}$. Suppose $y \in \R^n$ satisfies
\begin{equation*}
    R \leq \sum_{j=1}^n r_j |\inn{\be_j(s)}{y}| \leq 2 R 
\end{equation*}
for some $R \geq 1$. From the definition of the weight function from \eqref{gen it max fn}, and the orthonormality of the Frenet frame, the problem is further reduced to showing
\begin{equation}\label{Nik scale 1}
    \sum_{j=1}^n \tilde{r}_j |\inn{\tilde{\be}_j(\tilde{s})}{\tilde{y}}| \gtrsim R \qquad \textrm{where $\tilde{y} := \big([\gamma]_{\sigma, \lambda}\big)^{-1}(y)$.}
\end{equation}

Let $\alpha = \big([\gamma]_{s, \lambda}\big)^{-1}(y)$ so that, by the definition of the matrix $[\gamma]_{s, \lambda}$, we have 
\begin{equation*}
    y = \sum_{j=1}^n \lambda^j \alpha_j \gamma^{(j)}(s). 
\end{equation*}
Taking the inner product of both sides of this identity with respect to the vectors $\be_j(s)$, it follows that the vectors $\big(\inn{\be_j(s)}{y}\big)_{j=1}^n$ and $\big(\lambda^j \alpha_j\big)_{j=1}^n$ are related by an \textit{upper-triangular} matrix transformation, which is also an $O(\delta)$ perturbation of the identity. For this observation, we use the fact that $\langle \be_1(s), \dots, \be_j(s)\rangle = \langle \gamma^{(1)}(s), \dots, \gamma^{(j)}(s)\rangle$ for $1 \leq j \leq n$, owing to the definition of the Frenet frame. 

In view of the hypothesis \eqref{Nik scale hyp} which, in particular, implies $r_i \leq r_{i+1}$ for $1 \leq i \leq n-1$, the above observation yields that
\begin{equation}\label{Nik scale 3}
    r_j \lambda^j |\alpha_j| \lesssim R \qquad \textrm{for $1 \leq j \leq n$.}
\end{equation}
Furthermore, by pigeonholing, there exists some $1 \leq J \leq n$ such that 
\begin{equation*}
    r_J |\inn{\be_J(s)}{y}| \geq R/n \quad \textrm{and}  \quad r_j |\inn{\be_j(s)}{y}| < R/n \qquad \textrm{for $J + 1 \leq j \leq n$.}
\end{equation*}
Thus, by the same argument used to show \eqref{Nik scale 3}, provided $\delta$ is chosen sufficiently small,
\begin{equation}\label{Nik scale 4}
    r_J \lambda^J |\alpha_J| \sim R. 
\end{equation}

Since $\widetilde{\gamma}^{(j)}(\tilde{s}) = \lambda^j \big([\gamma]_{\sigma,\lambda}\big)^{-1} \gamma^{(j)}(s)$ for $j \geq 1$, it follows that $[\widetilde{\gamma}]_{\tilde{s}} = \big([\gamma]_{\sigma,\lambda}\big)^{-1} \circ [\gamma]_{s,\lambda}$ and, consequently,  
\begin{equation*}
    \tilde{y} = \big([\gamma]_{\sigma, \lambda}\big)^{-1}  (y) = \big([\gamma]_{\sigma, \lambda}\big)^{-1} \circ [\gamma]_{\lambda,s} (\alpha) = [\widetilde{\gamma}]_{\tilde{s}} (\alpha).
\end{equation*}
Thus, we have $\alpha = \big([\widetilde{\gamma}]_{\tilde{s}}\big)^{-1} (\tilde{y})$ and, arguing as before, this implies the vectors $\big(\inn{\tilde{\be}_j(\tilde{s})}{\tilde{y}}\big)_{j=1}^n$ and $\alpha$ are also related by an upper-triangle matrix transformation, which is again an $O(\delta)$ perturbation of the identity. From this observation, provided $\delta$ is chosen sufficiently small, we see that
\begin{equation*}
    \tilde{r}_J|\inn{\tilde{\be}_J(\tilde{s})}{\tilde{y}}| \gtrsim  r_J \lambda^J|\alpha_J| - \delta  \sum_{j = J+1}^n \big(r_J \lambda^{J - j}r_j^{-1}\big) r_j \lambda^j |\alpha_j| \gtrsim R,
\end{equation*}
where the final inequality uses the hypothesis \eqref{Nik scale hyp} together with \eqref{Nik scale 3} and \eqref{Nik scale 4}. This implies the desired bound \eqref{Nik scale 1}.

\end{proof}




\section{Proof of the \texorpdfstring{$\R^{3+1} \to \R^3$}{} Nikodym maximal estimate}\label{Nikodym sec}

In this section we establish Proposition~\ref{Nikodym prop}. We begin by recalling the basic setup. Let $\gamma \colon [-1,1] \to \R^3$ be a smooth, non-degenerate curve with Frenet frame $(\be_j)_{j=1}^3$. Given $\br \in (0,1)^3$ and $s \in [-1,1]$, consider the \textit{plates}
\begin{equation*}
    \mathcal{T}_{\br}(s) := \big\{ (y,t) \in \R^3 \times [1,2] :   \big|\inn{y - t\gamma(s)}{\be_j(s)}\big| \leq r_j \, \textrm{ for $j=1,2,3$} \big\} 
\end{equation*}
and the associated averaging and maximal operators
 \begin{equation*}
    \mathcal{A}_{\br}^{\mathrm{sing}} g(x; s) :=  \fint_{\mathcal{T}_{\br}(s)} g(x-y, t) \,\ud y \ud t \quad \textrm{and} \quad \mathcal{N}_{\br}^{\,\mathrm{sing}} g(x) := \sup_{-1 \leq s \leq 1} |\mathcal{A}_{\br}^{\mathrm{sing}} g(x; s)|. 
 \end{equation*}
We assume the exponents satisfy the conditions 
\begin{equation*}
    r_3 \leq r_2 \leq r_1 \leq r_2^{1/2} \qquad \textrm{and} \qquad r_2 \leq r_{1}^{1/2} r_3^{1/2}
\end{equation*}
and the goal is to establish the $L^2$ bound 
  \begin{equation}\label{Nikodym recall}
      \|\mathcal{N}_{\br}^{\,\mathrm{sing}} g\|_{L^2(\R^3)} \lesssim |\log r_3|^3 \|g\|_{L^2(\R^4)}.
  \end{equation}
To prove this norm inequality we will rely on the Fourier transform and reduce the problem to certain oscillatory integral estimates. The argument is a (significant) elaboration of that used to establish a lower dimensional variant of \eqref{Nikodym recall} in \cite{MSS1992}. We shall make heavy use of the frequency decomposition used to analyse the helical averaging operator in \S\ref{J=3 sec}. 

\begin{proof}[Proof of Proposition~\ref{Nikodym prop}] The argument is somewhat involved and is therefore broken into steps.\medskip

 \noindent\textit{Initial reductions}. Let $0 < \delta_0 \ll 1$ be a small parameter, as introduced at the beginning of \S\ref{sec:slow decay cone}. By familiar localisation and rescaling arguments, we may assume $\gamma$ satisfies $\gamma(\,\cdot\,) - \gamma(0) \in \mathfrak{G}_3(\delta_0)$. Further, we may replace $\mathcal{A}_{\br}^{\mathrm{sing}} g(x; s)$ with the localised version $\mathcal{A}_{\br}^{\mathrm{sing}} g(x; s) \chi(s)$, where $\chi \in C^{\infty}_c(\R)$ is supported in $I_0 := [-\delta_0, \delta_0]$. Note that this model situation is already enough for our application in \S\ref{L2 wtd 3+1 sec}.\medskip

 \noindent\textit{Fourier representation}. The first step is to derive an alternative representation of the averages $\mathcal{A}_{\br}^{\mathrm{sing}}g$ in terms of an oscillatory integral operator. Given $a \in C^{\infty}_c(\widehat{\R}^3 \times \R \times \R)$, define
 \begin{align*}
     \mathcal{A}[a]g(x;s) &:= \frac{1}{(2\pi)^3}\int_1^2 \int_{\R^3} \int_{\widehat{\R}^3} e^{i \inn{x-y-t\gamma(s)}{\xi}} a(\xi; s; t)\,\ud \xi\, g(y,t)\ud y\, \ud t \\
     &= \frac{1}{(2\pi)^3} \int_{\widehat{\R}^3} e^{i \inn{x}{\xi}}\int_1^2 e^{-it \inn{\gamma(s)}{\xi}} a(\xi; s; t) \tilde{g}(\xi,t) \,\ud t\,\ud \xi,
 \end{align*}
  where $\tilde{g}$ denotes the Fourier transform of $g$ with respect to the $y$-variable only. The associated maximal operator is then defined by
  \begin{equation*}
      \mathcal{N}[a] g(x) := \sup_{-1 \leq s \leq 1} |\mathcal{A}[a] g(x; s)|. 
  \end{equation*}
 Without loss of generality, to prove Proposition~\ref{Nikodym prop} it suffices to consider the estimate for $g$ Schwartz and taking values in $[0,\infty)$. Fix $\psi \in C^{\infty}_c(\widehat{\R})$ with $\supp \psi \subseteq [-1,1]$ such that $\widecheck{\psi}$ takes values in the positive real line and $\widecheck{\psi}(y) \gtrsim 1$ for $|y| \leq 1$. Define
 \begin{equation*}
     a_{\br}(\xi; s) := \prod_{j=1}^3 \psi\big(r_j  \inn{\xi}{\be_j(s)}\big) \, \chi(s) 
 \end{equation*}
 so that, by integral formula for the inverse Fourier transform and a change of variable, 
 \begin{equation*}
     \frac{1}{|\mathcal{T}_{\br}(s)|} \bbone_{\mathcal{T}_{\br}(s)}(y, t)\chi(s) \lesssim \prod_{j=1}^3 r_j^{-1} \widecheck{\psi}\big(r_j^{-1}  \inn{y - t\gamma(s)}{\be_j(s)}\big)\chi(s) = \frac{1}{(2 \pi)^3} \int_{\hat{\R}^3} e^{i \inn{y - t\gamma(s)}{\xi}} a_{\br}(\xi;s;t)\,\ud \xi.
 \end{equation*}
Thus, the pointwise inequality
 \begin{equation*}
    |\mathcal{A}_{\br}^{\mathrm{sing}} g(x;s)| \lesssim |\mathcal{A}[a_{\br}]g(x;s)|
 \end{equation*}
holds and therefore it suffices to bound the operator $\mathcal{N}[a_{\br}]$.\medskip
 
\noindent\textit{Sobolev embedding} Given $a \in C^{\infty}_c(\widehat{\R}^3 \times \R \times \R)$, by elementary Sobolev embedding,
  \begin{equation}\label{gen Sobolev}
    \| \mathcal{N}[a] g\|_{L^2(\R^3)}^2 \leq  \| \mathcal{A}[a]g \|_{L^2(\R^{3+1})}^2 + 2 \prod_{\iota \in \{0,1\}}\| \partial_s^{\iota}\, \mathcal{A}[a]g \|_{L^2(\R^{3+1})};
  \end{equation}
  indeed, this bound is a simple and standard consequence of the fundamental theorem of calculus and the Cauchy--Schwarz inequality (see for instance ~\cite[Chapter XI, $\S$3.2]{Stein1993}). Observe that $\partial_s\,\mathcal{A}[a]$ is an operator of the same form as $\mathcal{A}[a]$ and, in particular, 
 \begin{equation}\label{N gen deriv}
    \partial_s\, \mathcal{A}[a] = \mathcal{A}[\mathfrak{d}_s a] \qquad \textrm{where} \qquad    \mathfrak{d}_s a(\xi; s; t) := -i t \inn{\gamma'(s)}{\xi} \, a(\xi; s; t) + \partial_s a(\xi; s; t).
    \end{equation}
These observations reduce the problem to proving estimates of the form
\begin{equation}\label{N gen norm}
    \| \mathcal{A}[\mathfrak{d}_s^{\iota}\,a]g \|_{L^2(\R^{3+1}) \to L^2(\R^{3+1})} \leq B^{(\iota - 1/2)} \qquad \textrm{for $\iota \in \{0, 1\}$}
\end{equation}
for suitable symbols $a$ and constants $B \geq 1$. In particular, it suffices to decompose the original symbol $a_{\br}$ into $O(|\log r_3|^3)$ many pieces and show that \eqref{N gen norm} holds for some choice of $B \geq 1$ on each piece.\medskip
  
\noindent \textit{Reduction to oscillatory integral estimates}. Continuting to work with a general $a \in C^{\infty}_c(\widehat{\R}^3 \times \R \times \R)$, it follows from Plancherel's theorem in the $x$-variable and the Cauchy--Schwarz inequality that
  \begin{align}
      \| \mathcal{A}[a]g\|_{L^2(\R^{3+1})}^2 &\leq  \int_{\widehat{\R}^3} \int_{\R} |T_{\xi}[a]\tilde{g}(\xi;\,\cdot\,)(t)\tilde{g}(\xi;t)|\, \ud t \, \ud \xi \notag \\
      &\leq  \int_{\widehat{\R}^3} \|T_{\xi}[a]\tilde{g}(\xi;\,\cdot\,)\|_{L^2(\R)}\|\tilde{g}(\xi;\,\cdot\,)\|_{L^2(\R)}  \, \ud \xi \label{N gen square}
  \end{align}
where, for each $\xi \in \widehat{\R}^3$, the operator $T_{\xi}[a]$ acts on univariate functions by integrating (in the $t'$-variable) against the kernel
  \begin{equation}\label{N gen ker}
     \mathcal{K}[a](t,t';\xi) := \int_{\R} e^{i (t-t') \inn{\gamma(s)}{\xi}} \overline{a(\xi; s; t)}a(s,t';\xi)\bbone_{[1,2]^2}(t, t')\,\ud s.
  \end{equation}
It suffices to show that
  \begin{equation}\label{N gen univar}
      \| T_\xi [\fd_s^\iota a] \tilde{g}(\xi; \cdot) \|_{L^2(\R)} \leq B^{2 \iota -1} \| \tilde{g}(\xi; \cdot) \|_{L^2(\R)} \qquad \textrm{ for $\iota \in \{0,1\}$}
  \end{equation}
  holds uniformly in $\xi \in \widehat{\R}^3$. Indeed, in this case the norm bound \eqref{N gen norm} would follow via \eqref{N gen square} and a further application of Plancherel's theorem in the $\xi$-variable.
 By the Schur test, the inequality \eqref{N gen univar} is reduced to verifying the oscillatory integral estimates
  \begin{equation}\label{N gen Schur}
    \sup_{t' \in [1,2]} \int_1^2 |\mathcal{K}[\fd_s^\iota a](t,t';\xi)| \,\ud t, \quad \sup_{t \in [1,2]} \int_1^2 |\mathcal{K}[\fd_s^\iota a](t,t';\xi)| \,\ud t' \leq B^{2\iota -1}, \qquad \iota \in \{0,1\}
  \end{equation}
hold uniformly over all $\xi \in \hat{\R}^3$.\medskip 

\noindent \textit{Initial decomposition}. In order to obtain favourable estimates, it is necessary to first decompose the original symbol $a_{\br}$ into a number of localised pieces. This decomposition is similar to that used in \S\ref{J=3 sec} and is described in detail presently. Later in the proof, the kernel estimates \eqref{N gen Schur} are verified for each piece of the decomposition and the resulting norm bounds are combined to estimate the entire operator.\smallskip

Define $\delta_1 := \delta_0^3$, $\delta_2 := \delta_0$ and $\delta_3 := 9/10$ and for $1\leq J \leq 3$ let $\Omega_J$ denote the set of $\xi \in \hat{\R}^3$ satisfying
  \begin{align*}
      \inf_{s \in I_0} |\inn{\gamma^{(J)}(s)}{\xi}| &\geq \delta_J |\xi|, \\
      \inf_{s \in I_0} |\inn{\gamma^{(j)}(s)}{\xi}| &\leq \delta_j |\xi| \qquad \textrm{for $1 \leq j \leq J-1$.}
  \end{align*}
  Provided $\delta_0 > 0$ is chosen sufficiently small, the condition $\gamma(\,\cdot\,) - \gamma(0) \in \mathfrak{G}_3(\delta_0)$ ensures that these sets partition $\hat{\R}^3$. By pigeonholing,\footnote{As we are interested in $L^2$ estimates here, we are free to decompose the symbol using the rough partition of unity $1 \equiv \bbone_{\Omega_1} + \bbone_{\Omega_2} + \bbone_{\Omega_3}$.} it suffices to work with the symbols $a_{\br}^J(\xi;s):= a_{\br}(\xi;s)\bbone_{\Omega_J}(\xi)$ for $1 \leq J \leq 3$.\smallskip

Decompose the symbol into dyadic frequency bands by writing
\begin{equation*}
    a_{\br} = \sum_{k = 0}^{\infty} a_{\br,k} \qquad \textrm{where} \qquad  a_{\br,k}(\xi; s) :=   \left\{ \begin{array}{ll}
        a_{\br}^J(\xi; s) \cdot \beta^k(\xi) & \textrm{for $k \geq 1$} \\
         a_{\br}^J(\xi; s) \cdot \eta(\xi) & \textrm{for $k =0$}
     \end{array} \right. .
\end{equation*}
Here, for notational convenience, we suppress the choice of $J$ in the notation. Since $r_3 \leq r_1, r_2$, only the first $O(|\log r_3|)$ terms of the above sum are non-zero, so it suffices to show
 \begin{equation}\label{N freq loc est}
   \|\mathcal{N}[a_{\br,k}]\|_{L^2(\R^4) \to L^2(\R^3)}  \lesssim k^2 \qquad \textrm{for all $k \in \N_0$.}
 \end{equation}
In particular, note that $2^{k}\lesssim r_3^{-1}$. \medskip

\noindent\underline{$J = 1$ case.} Suppose $\xisupp a_{\br,k} \subseteq \Omega_1$. Here a simple integration-by-parts argument yields
\begin{equation*}
    \sup_{t' \in [1,2]} \int_1^2 |\mathcal{K}[\mathfrak{d}_s^{\iota} a_{\br, k}](t,t';\xi)| \,\ud t, \quad \sup_{t \in [1,2]} \int_1^2 |\mathcal{K}[\mathfrak{d}_s^{\iota} a_{\br, k}](t,t';\xi)| \,\ud t' \lesssim 2^{k(2\iota - 1)} \qquad \textrm{ for $\iota \in \{0,1\}$.}
  \end{equation*} 
In view of our earlier observations, the bound \eqref{N freq loc est} therefore holds in this case with a uniform bound in $k$. \medskip

\noindent\underline{$J = 2$ case.}  Suppose $\xisupp a_{\br,k} \subseteq \Omega_2$. If $\xi \in \Omega_2$, then the equation $\inn{\gamma'(s)}{\xi} = 0$ has a unique solution in $\frac{5}{4} \cdot I_0$ which we denote by $\theta(\xi)$. Indeed, this follows from a simple calculus exercise, similar to the proof of Lemma~\ref{theta2 lem}.\medskip 

\noindent \textit{Further decomposition} Here the symbol $a_{\br,k}$ is further decomposed by writing
   \begin{equation*}
       a_{\br,k} = \sum_{\ell = 0}^{\floor{k/2}}  a_{\br,k,\ell}
\qquad \textrm{where} \qquad
    a_{\br, k,\ell}(\xi;s) := 
    \left\{\begin{array}{ll}
       \displaystyle a_{\br,k}(\xi; s)\beta\big(2^{\ell}|s - \theta(\xi)|\big)   & \textrm{if $0 \leq \ell < \floor{k/2}$}  \\[6pt]
       \displaystyle a_{\br,k}(\xi; s)\eta\big(2^{\floor{k/2}}|s - \theta(\xi)|\big) & \textrm{if $\ell = \floor{k/2}$}
    \end{array}\right. .
    \end{equation*}
Since $|\inn{\gamma''(s)}{\xi}| \sim 2^k$ for all $(\xi;s) \in \supp a_{\br, k, \ell}$, one has the relation $2^k \leq r_2^{-1}$.
\medskip

\noindent \textit{Kernel estimates}. The kernels are analysed using stationary phase techniques. 

\begin{lemma}\label{J=2 N ker lem} If $k \in \N$, $0 \leq \ell \leq \floor{k/2}$ and $\iota \in \{0,1\}$, then 
      \begin{equation}\label{J=2 N ker}
    \sup_{t' \in [1,2]} \int_1^2 |\mathcal{K}[\mathfrak{d}_s^{\iota} a_{\br,k,\ell}](t,t';\xi)| \,\ud t, \quad \sup_{t \in [1,2]} \int_1^2 |\mathcal{K}[\mathfrak{d}_s^{\iota} a_{\br, k,\ell}](t,t';\xi)| \,\ud t' \lesssim 2^{(k-\ell)(2\iota - 1)}.
  \end{equation} 
  \end{lemma}

\begin{proof} If $\ell=\floor{k/2}$, then the localisation of the symbol ensures that $|s-\theta (\xi)|\lesssim 2^{-\ell}$ for all $(\xi;s)\in \supp a_{\br,k,\ell}$. The bound for $\iota=0$ then follows immediately from the size of the $s$-support of $a_{\br,k,\ell}$. For $\iota=1$, note that by the mean value theorem, we may write
\begin{equation}\label{J=2 Nik 1}
    \inn{\gamma'(s)}{\xi} = \omega(\xi; s) \, (s - \theta(\xi))
\end{equation}
where $|\omega(\xi;s)| \sim 2^k$ on $\supp a_{k,\ell}$. Consequently, 
\begin{equation}\label{J=2 Nik 2}
    | \inn{\gamma'(s)}{\xi}| \lesssim 2^{k/2} \qquad \textrm{ for all $(\xi; s) \in \supp a_{\br,k,\ell}$.}
\end{equation}
Furthermore, by the definition of $a_{\br}$ and of the Frenet frame $\{\be_j(r)\}_{j=1}^3$, the relation $r_3 \leq r_2 \leq r_1 \lesssim r_2^{1/2} \leq 2^{-k/2}$ implies
\begin{equation}\label{J=2 Nik 3}
    |\partial_s a_{\br, k,\ell}(\xi;s)|\lesssim 2^{k/2}.
\end{equation}
In view of the definition of $\fd_s$  in \eqref{N gen deriv}, the bounds \eqref{J=2 Nik 2} and \eqref{J=2 Nik 3} immediately imply that $|\fd_s a_{\br,k,\ell} (\xi;s)| \lesssim 2^{k/2}$ and the bound for $\iota=1$ now follows immediately from the size of the $s$-support of $a_{\br,k,\ell}$ and the definition of $\mathcal{K}$ in \eqref{N gen ker}.

If $0 \leq \ell < \floor{k/2}$, then the localisation of the symbols ensures that
\begin{equation}\label{J=2 N ker 1}
    |s - \theta(\xi)| \sim 2^{-\ell} \qquad \textrm{for all $(\xi;s) \in \supp a_{\br, k,\ell}$.}
\end{equation}
Consequently, by directly applying \eqref{J=2 N ker 1} in \eqref{J=2 Nik 1}, we have the bounds 
\begin{equation}\label{J=2 N ker 2}
    |\inn{\gamma'(s)}{\xi}| \sim 2^{k-\ell}, \quad |\inn{\gamma^{(N)}(s)}{\xi}| \lesssim 2^k \qquad \textrm{for $N \geq 2, \,\,\, (\xi;s) \in \supp a_{\br,k,\ell}$.}
\end{equation}
Moreover, by the definition of $a_{\br}$, the first relation above immediately implies $2^{k-\ell} \leq r_1^{-1}$; recall that $r_2, r_3\leq 2^{-k}$. Thus, by the definition of the Frenet frame $\{\be_j(s)\}_{j=1}^3$, the symbol satisfies
\begin{equation}\label{J=2 N ker 3}
        |\partial_s^N a_{\br,k,\ell}(\xi;s)| \lesssim 2^{\ell N} = 2^{-(k-2\ell)N} 2^{(k - \ell)N} \qquad \textrm{for all $N \in \N_0$.} 
\end{equation}
Thus, we may bound the kernel via repeated integration-by-parts. In particular, applying Lemma~\ref{non-stationary lem} with $\phi(s):=(t-t')\inn{\gamma(s)}{\xi}$  and $R := 2^{k-2\ell}|t-t'|$, we deduce that 
\begin{equation*}
    |\mathcal{K}[\mathfrak{d}_s^{\iota} a_{\br,k,\ell}](\xi; t, t')| \lesssim_N 2^{2(k - \ell)\iota} 2^{-\ell} \big(1 + 2^{k-2\ell}|t - t'| \big)^{-N}, \qquad \textrm{for $\iota \in \{0,1\}$}.
\end{equation*}
 The additional $2^{2(k - \ell)}$ factor arises in the bound for the derived operator owing to the formula \eqref{N gen deriv} for the corresponding symbol (and in particular, due to the first bound in \eqref{J=2 N ker 2}, the bounds in \eqref{J=2 N ker 3} and the relation $0 \leq \ell < \floor{k/2}$) and the form of the kernel as described in \eqref{N gen ker}.  Integrating both sides of the above display in either $t$ or $t'$, the desired estimate \eqref{J=2 N ker} follows.
\end{proof}

\noindent \textit{Putting everything together}. In view of the kernel estimates from Lemma~\ref{J=2 N ker lem} and the discussion at the beginning of the proof, it follows that
\begin{equation*}
   \| \mathcal{A}[\mathfrak{d}_s^{\iota} a_{\br,k,\ell}]g \|_{L^2(\R^4) \to L^2(\R^{3+1})}  \lesssim 2^{(k-\ell)(\iota - 1/2)} \qquad \textrm{for all $0 \leq \ell \leq \floor{k/2}$ and $\iota \in \{0,1\}$.}
\end{equation*}
 Combining these bounds with \eqref{gen Sobolev}, it follows that
\begin{align*}
   \| \mathcal{N}[a_{\br,k,\ell}]g \|_{L^2(\R^4) \to L^2(\R^3)}  &\lesssim 1 \quad \textrm{for all $0 \leq \ell \leq \floor{k/2}$}
\end{align*}
The frequency localised maximal bound \eqref{N freq loc est} immediately follows (with linear dependence on $k$)  from the triangle inequality.\medskip

\noindent\underline{$J = 3$ case.} Suppose $\xisupp a_{\br,k} \subseteq \Omega_3$. As in Lemma~\ref{theta2 lem}, if $\xi \in \Omega_3$, then the equation $\inn{\gamma''(s)}{\xi} = 0$ has a unique solution in $[-1,1]$, which we denote by $\theta_2(\xi)$. As in Lemma~\ref{theta1 lem}, if $u(\xi) < 0$, where
\begin{equation*}
   u(\xi) := \inn{\gamma'\circ \theta_2(\xi)}{\xi}, 
\end{equation*}
then the equation $\inn{\gamma'(s)}{\xi} = 0$ has a precisely two solutions in $[-1,1]$, which we denote by $\theta_1^{\pm}(\xi)$. We will further assume without loss of generality that $\inn{\gamma^{(3)}(s)}{\xi} > 0$ for all $\xi \in \xisupp a_{\br,k}$. \medskip

\noindent \textit{Further decomposition} Here the symbol $a_{\br,k}$ is decomposed in a manner similar (but not quite identical) to that used in \S\ref{J=3 sec}. First perform a dyadic decomposition of $u(\xi)$ by writing
    \begin{equation*}
       a_{\br,k} = \sum_{\ell = 0}^{\floor{k/3}}  a_{\br,k,\ell} + \sum_{\ell = 0}^{\floor{k/3}-1}  a_{\br,k,\ell}^+
\end{equation*}
where
\begin{equation*}
    a_{\br,k,\ell}(\xi; s) := 
    \left\{\begin{array}{ll}
       \displaystyle a_{\br,k}(\xi; s)\beta^-\big(2^{-k+ 2\ell}\bu(\xi)\big)   & \textrm{if $0 \leq \ell < \floor{k/3}$}  \\[6pt]
       \displaystyle a_{\br,k}(\xi; s)\eta\big(2^{-k + 2\floor{k/3}}\bu(\xi)\big) & \textrm{if $\ell = \floor{k/3}$}
    \end{array}\right. 
    \end{equation*}
and the $a_{\br,k,\ell}^+$ are defined similarly but with $\beta^+$ in place of $\beta^-$. Here $\beta = \beta^- + \beta^+$ is the decomposition of the bump function described in \S\ref{loc curv subsec}. The symbols $a_{\br,k,\ell}^+$ are relatively easy to analyse, and are dealt with using an argument similar to that of the $J=2$ case. Henceforth, we focus exclusively on the $a_{\br,k,\ell}$.\smallskip

We further decompose each $a_{\br,k,\ell}$ with respect to the distance of the $s$-variable to the root $\theta_2(\xi)$. Once again it is convenient to introduce a fine tuning constant $\rho > 0$. Similar to \eqref{akell dec}, define
\begin{equation}\label{N J3 s loc a}
    a_{\br,k,\ell, 0}(\xi; s) :=   a_{\br,k,\ell}(\xi; s) \eta \big(\rho 2^{\ell} |s-\theta_2(\xi)|\big)
       \qquad \textrm{for $0 \leq \ell \leq \floor{k/3}$.}
\end{equation}
Note, in contrast with  \eqref{akell dec}, we have not decomposed with respect to $|s - \theta_1^{\pm}(\xi)|$ for $\ell < \floor{k/3}$. Such a decomposition does appear later: here it is necessary to localise simultaneously with respect to \textit{both} roots $\theta_2(\xi)$ and $\theta_1^{\pm}(\xi)$.  
Also in contrast with the analysis of \S\ref{J=3 sec}, here it is not possible to reduce the problem to studying the $s$-localised pieces in \eqref{N J3 s loc a}. Consequently, we also consider the $s$-localisation of the symbol to the remaining dyadic shells, viz. \begin{equation*}
    a_{\br,k,\ell, m}(\xi;s):=   a_{\br,k,\ell}(\xi) \beta \big(\rho 2^{\ell - m}|s-\theta_2(\xi)|\big)
       \qquad \textrm{for $0 \leq \ell \leq \floor{k/3}$.}
\end{equation*}

The most difficult terms to estimate correspond to $0 \leq \ell < \floor{k/3}$ and $m = 0$. These symbols require a further decomposition. In particular, for $0 \leq \ell < \floor{k/3}$ let
\begin{equation*}
    b_{\br,k,\ell, m}(\xi;s) := 
    \begin{cases}
a_{\br,k,\ell, 0}(\xi;s) \eta\big(\rho^{-1} 2^{(k-\ell)/2}\displaystyle\min_{\pm}|s- \theta_1^{\pm}(\xi)|\big) & \textrm{if $m = 0$} \\[5pt]
a_{\br,k,\ell, 0}(\xi;s) \beta\big(\rho^{-1} 2^{(k-\ell)/2 - m }\displaystyle\min_{\pm}|s- \theta_1^{\pm}(\xi)|\big) & \textrm{if $1 \leq m < \floor{\frac{k-3\ell}{2}}$} 
\\[5pt]
a_{\br,k,\ell, 0}(\xi;s) \big(1-\eta\big(\rho^{-1} 2^{(k-\ell)/2 - m }\displaystyle\min_{\pm}|s- \theta_1^{\pm}(\xi)|\big)\big) & \textrm{if $m = \floor{\frac{k-3\ell}{2}}$} 
\end{cases}.
\end{equation*}
Observe that Lemma~\ref{root control lem} already implies that $|s - \theta_1^{\pm}(\xi)| \lesssim \rho^{-1} 2^{-\ell}$ for $(\xi; s) \in \supp a_{\br,k,\ell, 0}$. Thus, $\rho 2^{-\ell} \lesssim |s-\theta_1^{\pm}(\xi)| \lesssim \rho^{-1} 2^{-\ell}$ for $(\xi;s) \in \supp b_{\br,k,\ell,m}$ for $m=\floor{\frac{k-3\ell}{2}}$. 

Combining the above definitions and observations, the symbol may be written as
\begin{equation*}
       a_{\br,k} = \sum_{\ell = 0}^{\floor{k/3}} \sum_{m = 0}^{\ell} a_{\br,k,\ell, m} =  \sum_{(\ell,m)\in \Lambda_a(k)} a_{\br,k,\ell, m} + \sum_{(\ell, m) \in \Lambda_b(k)} b_{\br,k,\ell, m}
\end{equation*}
where
\begin{align*}
    \Lambda_a(k) &:= \big\{(\ell, m) \in \N_0^2 : 0 \leq \ell \leq \floor{\tfrac{k}{3}} \textrm{ and } 1 \leq m \leq \ell \big\} \cup \big\{\big(\floor{\tfrac{k}{3}}, 0\big) \big\}, \\
    \Lambda_b(k) &:= \big\{(\ell, m) \in \N_0^2 : 0 \leq \ell < \floor{\tfrac{k}{3}} \textrm{ and } 0 \leq m \leq \floor{\tfrac{k-3\ell}{2}} \big\}.
\end{align*}
Note that the range of $m$ in the definition of $\Lambda_a(k)$ is restricted since $a_{\br,k,\ell,m}$ is identically zero whenever $m > \ell$.\medskip

\noindent \textit{Kernel estimates}. The kernels are analysed using stationary phase techniques. 

\begin{lemma}\label{J=3 N ker lem} Let $k \in \N$ and $\iota \in \{0,1\}$.
\begin{enumerate}[a)]
    \item If $(\ell, m) \in \Lambda_a(k)$, then
      \begin{equation}\label{J=3 N ker a}
    \sup_{t' \in [1,2]} \int_1^2 |\mathcal{K}[\mathfrak{d}_s^{\iota} a_{\br,k,\ell, m}](t,t';\xi)| \,\ud t, \quad \sup_{t \in [1,2]} \int_1^2 |\mathcal{K}[\mathfrak{d}_s^{\iota} a_{\br,k,\ell, m}](t,t';\xi)| \,\ud t' \lesssim 2^{(k-2\ell + 2m)(2\iota - 1)}.
  \end{equation} 
    \item If $(\ell, m) \in \Lambda_b(k)$, then
      \begin{equation}\label{J=3 N ker b}
    \sup_{t' \in [1,2]} \int_1^2 |\mathcal{K}[\mathfrak{d}_s^{\iota} b_{\br,k,\ell, m}](t,t';\xi)| \,\ud t, \quad \sup_{t \in [1,2]} \int_1^2 |\mathcal{K}[\mathfrak{d}_s^{\iota} b_{\br,k,\ell, m}](t,t';\xi)| \,\ud t' \lesssim 2^{((k - \ell)/2 + m)(2\iota - 1)}.
  \end{equation} 
\end{enumerate}

\end{lemma}

\begin{proof} The argument is similar to that used to prove Lemma~\ref{J=3 s loc lem}. \medskip

\noindent a) Let $(\ell, m) \in \Lambda_a(k)$. If $(\ell,m)=(\floor{k/3},0)$, then the localisation of the $a_{\br,k,\ell, m}$ symbols ensures that
\begin{equation}\label{J=3 Nik 1}
     |\bu(\xi)| \lesssim 2^{k/3} \quad \text{ and } \quad |s-\theta_2(\xi)| \lesssim  \rho^{-1}2^{-k/3}  \quad \textrm{for all $(\xi;s) \in \supp a_{\br, k,\ell, m}$.}
\end{equation}
The bound \eqref{J=3 N ker a} for $\iota=0$ follows immediately from the size of the $s$-support of $a_{\br,k,\ell,m}$. For $\iota=1$, apply the familiar Taylor expansion to write 
\begin{equation}\label{J=3 Nik 2}
\begin{split}
    \inn{\gamma'(s)}{\xi} &= u(\xi) + \omega_1(\xi; s) \, (s - \theta_2(\xi))^2, \\
    \inn{\gamma''(s)}{\xi} &= \omega_2(\xi; s) \, (s - \theta_2(\xi))
\end{split}
\end{equation}
where $|\omega_j(\xi;s)| \sim 2^k$ on $\supp a_{\br, k, \ell, m}$ for $j=1$, $2$. Consequently, by directly applying \eqref{J=3 Nik 1}, we have the upper bounds
\begin{equation}\label{J=3 Nik 3}
    | \inn{\gamma'(s)}{\xi}|   \lesssim  \rho^{-2} 2^{k/3}, \quad
    | \inn{\gamma''(s)}{\xi}|   \lesssim \rho^{-1} 2^{2k/3}, 
     \qquad \textrm{for all $(\xi;s) \in \supp a_{\br,k,\ell,m}$.} 
\end{equation}
Note that the relations $r_2 \leq r_1 \leq r_2^{1/2}$ and $r_3 \leq r_2 \leq r_1^{1/2}r_3^{1/2}$ imply, in particular, $r_1 \leq r_3^{1/3} \lesssim 2^{-k/3}$ and $r_2 \leq r_3^{2/3} \lesssim 2^{-2k/3}$. It then follows from the definitions of $a_{\br}$ and of the Frenet frame $\{\be_j(r)\}_{j=1}^3$ that
\begin{equation}\label{J=3 Nik 4}
    |\partial_s a_{\br,k,\ell, m}(\xi;s)| \lesssim 2^{k/3}.
\end{equation}
In view of the definition of $\fd_s$  in \eqref{N gen deriv}, the first bound in \eqref{J=3 Nik 3} and \eqref{J=3 Nik 4} immediately imply that $|\fd_s a_{\br, k, \ell, m}(\xi;s)| \lesssim 2^{k/3}$, and the bound for $\iota=1$ now follows immediately from the size of the $s$-support of $a_{\br,k,\ell,m}$ and the definition of $\mathcal{K}$ in \eqref{N gen univar}.

Now suppose $0 \leq \ell \leq \floor{k/3}$ and $1 \leq m \leq \ell$. Then the localisation of the $a_{\br,k,\ell, m}$ symbols ensures that
\begin{equation}\label{J=3 N ker 1 a}
     |\bu(\xi)| \lesssim 2^{k - 2\ell} \quad \text{ and } \quad |s-\theta_2(\xi)| \sim  \rho^{-1}2^{-\ell + m}  \quad \textrm{for all $(\xi;s) \in \supp a_{\br,k,\ell, m}$.}
\end{equation}
%
Provided $\rho$ is chosen sufficiently small, by directly applying \eqref{J=3 N ker 1 a} in \eqref{J=3 Nik 2}, we have the bounds
\begin{equation}\label{J=3 N ker 2 a}
    | \inn{\gamma'(s)}{\xi}|   \sim  \rho^{-2} 2^{k - 2\ell + 2m}, \quad
    | \inn{\gamma''(s)}{\xi}|   \sim \rho^{-1} 2^{k - \ell + m}, \quad
    | \inn{\gamma^{(N)}(s)}{\xi}|   \lesssim_N   2^k 
     \qquad \textrm{for $N \geq 3$.} 
\end{equation}
By the definition of $a_{\br}$, the first and second bounds above immediately imply $2^{k-2\ell+2m} \leq r_1^{-1}$ and $2^{k-\ell+m}\leq r_2^{-1}$, whilst $2^k \leq r_3^{-1}$. Thus, by the definition of the Frenet frame $\{\be_j(s)\}_{j=1}^3$ and the bounds \eqref{J=3 N ker 2 a}, the symbol satisfies
\begin{equation}\label{J=3 N ker 3 a}
    |\partial_s^N a_{\br, k,\ell, m}(\xi;s)| \lesssim 2^{(\ell - m)N} = 2^{-(k-3\ell + 3m)N} 2^{(k - 2\ell + 2m)N} \qquad \textrm{for all $N \in \N_0$.}
\end{equation}
Thus, we may bound the kernel via repeated integration-by-parts. In particular, applying Lemma~\ref{non-stationary lem} with $\phi(s):=(t-t')\inn{\gamma(s)}{\xi}$ and $R:=2^{k-3\ell+3m}|t-t'|$, we deduce that 
\begin{equation*}
    |\mathcal{K}[\mathfrak{d}_s^{\iota} a_{\br,k,\ell,m}](\xi; t, t')| \lesssim_N 2^{2(k - 2\ell + 2m)\iota} 2^{-\ell + m} \big(1 + 2^{k-3\ell + 3m}|t - t'| \big)^{-N}.
\end{equation*}
The additional $2^{2(k - 2\ell + 2m)\iota}$ arises in the bound for the derived operator $\fd_s$ owing to the formula \eqref{N gen deriv} for the corresponding symbol (and in particular, due to the bounds in \eqref{J=3 N ker 2 a} and in \eqref{J=3 N ker 3 a} and the relation $0 \leq \ell - m \leq \ell \leq \floor{k/3}$) and the form of the kernel $\mathcal{K}$ as described in \eqref{N gen ker}. Finally, by integrating both sides of the above display in either $t$ or $t'$, the desired estimate \eqref{J=3 N ker a} follows.\medskip
 
 \noindent b) Let $(\ell, m) \in \Lambda_b(k)$. If $m =0$, then the localisation of the $b_{\br,k,\ell, m}$ symbols ensures that
\begin{equation}\label{J=3 Nik 1 b}
     |\bu(\xi)| \sim 2^{k-2\ell} \quad \text{ and } \quad \min_{\pm}|s-\theta_1^{\pm}(\xi)| \lesssim  \rho 2^{-(k-\ell)/2}  \quad \textrm{for all $(\xi;s) \in \supp b_{\br,k,\ell, m}$.}
\end{equation}
The bound \eqref{J=3 N ker b} for $\iota=0$ follows immediately from the size of the $s$-support of $b_{\br,k,\ell,m}$. For $\iota=1$, apply the familiar Taylor expansion to write 
\begin{equation}\label{J=3 Nik 2 b}
    \begin{split}
    \inn{\gamma'(s)}{\xi} &= v^{\pm}(\xi)\, (s - \theta_1^{\pm}(\xi)) + \omega_1^{\pm}(\xi; s) \, (s - \theta_1^{\pm}(\xi))^2, \\
    \inn{\gamma''(s)}{\xi} &= v^{\pm}(\xi) +  \omega_2^{\pm}(\xi; s) \, (s - \theta_1^{\pm}(\xi))
\end{split}
\end{equation}
where $|\omega_j^{\pm}(\xi;s)| \sim 2^k$ on $\supp b_{\br, k, \ell, m}$ for $j=1$, $2$. 
Consequently, in view of Lemma~\ref{root control lem} and \eqref{J=3 Nik 1 b}, and provided $\rho>0$ is chosen sufficiently small, we have the bounds,
\begin{equation}\label{J=3 Nik 4 b}
    | \inn{\gamma'(s)}{\xi}| \lesssim \rho 2^{(k - \ell)/2}, \quad   | \inn{\gamma''(s)}{\xi}| \sim  2^{k - \ell} \qquad \textrm{for all $(\xi;s) \in \supp b_{\br,k,\ell,m}$,} 
\end{equation}
using the relation $0 \leq \ell \leq \floor{k/3}$.

By the definition of $a_{\br}$, the second bound above implies $r_2 \lesssim 2^{-(k-\ell)}$ and therefore $r_1 \leq r_2^{1/2} \lesssim 2^{-(k-\ell)/2}$, whilst $r_3 \lesssim 2^{-k}$. Thus, by the definition of the Frenet frame $\{\be_j(s)\}_{j=1}^3$ and the bounds \eqref{J=3 Nik 4 b}, the symbol satisfies
\begin{equation}\label{J=3 Nik 5 b}
    |\partial_s b_{\br,k,\ell,m}(\xi;s)| \lesssim_N 2^{(k-\ell)/2},
\end{equation}
using the relation $0 \leq \ell < \floor{k/3}$.
In view of the definition of $\fd_s$ in \eqref{N gen deriv}, the first bound in \eqref{J=3 Nik 4 b} and \eqref{J=3 Nik 5 b} immediately implies that $|\fd_s b_{\br,k,\ell,m}(\xi;s)|\lesssim 2^{(k-\ell)/2}$, and the bound for $\iota=1$ now follows immediately from the size of the $s$-support of $b_{\br,k,\ell,m}$ and the definition of $\mathcal{K}$ in \eqref{N gen ker}.
\smallskip

Now suppose $0 < m < \floor{\frac{k-3\ell}{2}}$. Then the localisation of the $b_{\br, k,\ell, m}$ symbols ensures that
\begin{equation}\label{J=3 N 1 b}
     |\bu(\xi)| \sim 2^{k-2\ell} \quad \text{ and } \quad \min_{\pm}|s-\theta_1^{\pm}(\xi)| \sim  \rho 2^{-(k-\ell)/2 + m}  \quad \textrm{for all $(\xi;s) \in \supp b_{\br,k,\ell, m}$.}
\end{equation}
Using the convexity argument from the proof of Lemma~\ref{J=3 s loc lem}, we may bound 
\begin{equation}\label{J=3 Nik 3 b}
    |\inn{\gamma'(s)}{\xi}| \geq \min_{\pm} \frac{|u(\xi)| |s-\theta_1^{\pm}(\xi)|} {|\theta_2(\xi)- \theta_1^{\pm}(\xi)|} \qquad \textrm{for all $(\xi;s) \in \supp b_{\br,k,\ell, m}$.}
\end{equation}
Consequently, using Lemma~\ref{root control lem} and \eqref{J=3 N 1 b} in \eqref{J=3 Nik 2 b} and \eqref{J=3 Nik 3 b}, and provided $\rho>0$ is chosen sufficiently small,
\begin{equation}\label{J=3 N 2 b}
    | \inn{\gamma'(s)}{\xi}| \sim \rho 2^{(k - \ell)/2 + m}, \quad   | \inn{\gamma''(s)}{\xi}| \sim  2^{k - \ell} \quad \textrm{and} \quad |\inn{\gamma^{(N)}(s)}{\xi}|  \lesssim_N   2^k \qquad \textrm{for all $N \geq 3$.} 
\end{equation}
For the upper bound in the first derivative in the above display, we use the restriction $m \leq \floor{\tfrac{k-3\ell}{2}}$. It is for this reason that we simultaneously localise with respect to \textit{both} $\theta_2(\xi)$ and $\theta_1^{\pm}(\xi)$. 
In particular,
\begin{equation*}
    |\inn{\gamma^{(N)}(s)}{\xi}|\lesssim 2^k \sim 2^{k-((k-\ell)/2+m)N}  |\inn{\gamma'(s)}{\xi}|^N \lesssim 2^{-2m(N-1)}   |\inn{\gamma'(s)}{\xi}|^N \quad \text{ for all $N \geq 3$,}
\end{equation*}
where in the last inequality one uses the restriction $m \leq \floor{\frac{k-3\ell}{2}}$ and the fact $N \geq 3$.

By the definition of $a_{\br}$, the first and second bounds in \eqref{J=3 N 2 b} imply $r_1 \leq 2^{-(k-\ell)/2 - m}$ and $r_2 \leq 2^{-(k-\ell)}$, whilst $r_3 \leq 2^{-k}$. Thus, by the definition of the Frenet frame $\{\be_j(s)\}_{j=1}^3$ and the bounds \eqref{J=3 N 2 b}, the symbol satisfies
\begin{equation}\label{J=3 N 3 b}
    |\partial_s^N b_{\br,k,\ell, m}(\xi;s)| \lesssim 2^{((k-\ell)/2 - m)N} = 2^{-2mN} 2^{((k-\ell)/2 + m)N} \qquad \textrm{for all $N \in \N_0$,}
\end{equation}
using the restriction $m\leq \floor{\frac{k-3\ell}{2}}$. 
Thus, we may bound the kernel via repeated integration-by-parts. In particular, applying Lemma~\ref{non-stationary lem} with $\phi(s):=(t-t')\inn{\gamma(s)}{\xi}$ and  $R := 2^{2m}|t-t'|$, we deduce that
\begin{equation*}
    |\mathcal{K}[\mathfrak{d}_s^{\iota} b_{\br,k,\ell,m}](\xi; t, t')| \lesssim_N 2^{(k - \ell + 2m)\iota} 2^{-(k-\ell)/2 + m} \big(1 + 2^{2m}|t - t'| \big)^{-N}.
\end{equation*}

The additional $2^{(k - \ell + 2m)\iota}$ arises in the bound for the derived operator $\fd_s$ owing to the formula \eqref{N gen deriv} for the corresponding symbol (and in particular, due to the bounds in \eqref{J=3 N 2 b} and in \eqref{J=3 N 3 b}) and the form of the kernel $\mathcal{K}$ as described in \eqref{N gen ker}. Finally, by integrating both sides of the above display in either $t$ or $t'$, the desired estimate \eqref{J=3 N ker a} follows.\smallskip
 
 Finally, consider the case $m=\floor{\frac{k-3\ell}{2}}$. Then the localisation of the $b_{\br,k,\ell, m}$ symbols ensures that 
\begin{equation}\label{J=3 Nikodym 1 b}
     |\bu(\xi)| \sim_{\rho} 2^{k-2\ell} \quad \text{ and } \quad \min_{\pm}|s-\theta_1^{\pm}(\xi)| \sim_{\rho}   2^{-\ell}  \quad \textrm{for all $(\xi;s) \in \supp b_{\br,k,\ell, m}$.}
\end{equation}
Using Lemma~\ref{root control lem} and \eqref{J=3 Nikodym 1 b} in \eqref{J=3 Nik 2 b} and \eqref{J=3 Nik 3 b}, we have the bounds
\begin{equation*}
    | \inn{\gamma'(s)}{\xi}| \sim 2^{k - 2\ell}, \quad   | \inn{\gamma''(s)}{\xi}| \lesssim  2^{k - \ell}, \quad \textrm{and} \quad |\inn{\gamma^{(N)}(s)}{\xi}|  \lesssim_N   2^k \qquad \textrm{for all $N \geq 3$.} 
\end{equation*}
By the definition of $a_{\br}$, the first bound above implies $r_1 \lesssim_{\rho} 2^{-(k-2\ell)}$ and, as $r_3 \lesssim 2^{-k}$, one has $r_2 \leq r_1^{1/2}r_3^{1/2} \lesssim_{\rho} 2^{k-\ell}$. Thus, by the definition of the Frenet frame $\{\be_j(s)\}_{j=1}^3$ and the bounds \eqref{J=3 N 2 b}, the symbol satisfies
\begin{equation*}
    |\partial_s^N b_{\br,k,\ell, m}(\xi;s)| \lesssim 2^{\ell N} = 2^{-(k-3\ell)N} 2^{(k-2\ell)N} \qquad \textrm{for all $N \in \N_0$.}
\end{equation*}
Thus, we may bound the kernel via repeated integration-by-parts. In particular, applying Lemma~\ref{non-stationary lem} with $\phi(s):=(t-t')\inn{\gamma(s)}{\xi}$ and  $R := 2^{k-3\ell}|t-t'|$, we deduce that
\begin{equation*}
    |\mathcal{K}[\mathfrak{d}_s^{\iota} b_{\br,k,\ell,m}](\xi; t, t')| \lesssim_N 2^{2(k - 2\ell)\iota} 2^{-\ell} \big(1 + 2^{k-3\ell}|t - t'| \big)^{-N}.
\end{equation*}
The additional $2^{2(k - 2\ell)\iota}$ arises in the bound for the derived operator $\fd_s$ owing to the formula \eqref{N gen deriv} for the corresponding symbol (and in particular, due to the bounds in \eqref{J=3 N 2 b} and in \eqref{J=3 N 3 b} and the restriction $ \ell \leq \floor{k/3}$) and the form of the kernel $\mathcal{K}$ as described in \eqref{N gen ker}. Finally, by integrating both sides of the above display in either $t$ or $t'$, the desired estimate \eqref{J=3 N ker a} follows
\medskip
\end{proof}

\noindent \textit{Putting everything together}. In view of the kernel estimates from Lemma~\ref{J=3 N ker lem} and the discussion at the beginning of the proof, it follows that
\begin{align*}
   \| \mathcal{A}[\mathfrak{d}_s^{\iota} a_{\br,k,\ell,m}]g \|_{L^2(\R^4) \to L^2(\R^{3+1})}  &\lesssim 2^{(k-2\ell + 2m)(\iota - 1/2)} \quad \textrm{for all $(\ell, m) \in \Lambda_a(k)$}, \\
   \| \mathcal{A}[\mathfrak{d}_s^{\iota} b_{\br,k,\ell,m}]g \|_{L^2(\R^4) \to L^2(\R^{3+1})}  &\lesssim 2^{((k - \ell)/2 + m)(\iota - 1/2)} \quad \textrm{for all $(\ell, m) \in \Lambda_b (k)$},
\end{align*}
for $\iota \in \{0,1\}$. Combining these bounds with \eqref{gen Sobolev}, it follows that
\begin{align*}
   \| \mathcal{N}[a_{\br,k,\ell,m}]g \|_{L^2(\R^4) \to L^2(\R^3)}  &\lesssim 1 \quad \textrm{for all $(\ell, m) \in \Lambda_a(k)$}, \\\
   \| \mathcal{N}[b_{\br,k,\ell,m}]g \|_{L^2(\R^4) \to L^2(\R^3)}  &\lesssim 1 \quad \textrm{for all $(\ell, m) \in \Lambda_b (k)$}.
\end{align*}
Since the cardinalities of $\Lambda_a(k)$ and $\Lambda_b(k)$ are $O(k^2)$, the frequency localised maximal bound \eqref{N freq loc est} immediately follows from the triangle inequality. Summing over $k$ then concludes the proof of the proposition. 
\end{proof}




\section{Necessary conditions}\label{nec cond sec}

In this final section we show the condition $p > 3$ in Theorem~\ref{intro max thm} is necessary. Moreover, we prove the following result, which is valid in arbitrary dimensions $n \geq 2$.

\begin{proposition} If $n \geq 2$ and $\gamma \colon I \to \R^n$ is a smooth non-degenerate curve, then
\begin{equation*}
    \|M_{\gamma}\|_{L^p(\R^n) \to L^p(\R^n)} = \infty \qquad \textrm{for $1 \leq p \leq n$.}
\end{equation*}
\end{proposition}

\begin{proof} By localisation of the operator and applying the rescaling from \S\ref{curve sym sec}, it suffices to consider the case where
\begin{equation*}
    \gamma(\,\cdot\,) - \gamma(0) \in \mathfrak{G}_n(\delta_0) \qquad \textrm{and} \qquad \inn{\gamma(0)}{\vec{e}_n} \neq 0
\end{equation*}
for $\delta_0 := 10^{-n}$, say. By reparametrising the curve, we may also assume that the first component of $\gamma \colon [-1,1] \to \R^n$ is of the form $\gamma_1(s) = s + a_1$ for some $a_1 \in \R$. 

By a simple projection argument, it suffices to study the boundedness of a maximal operator defined over the Euclidean plane. In particular, fix $a = (a_1, a_2) \in \R^2$ with $a_2 \neq 0$ and a smooth function $h \colon [-1,1] \to \R$ satisfying 
\begin{equation}\label{nec prop 1}
    h^{(j)}(0) = 0 \quad \textrm{for $0 \leq j \leq n-1$} \quad \textrm{and} \quad h^{(n)}(0) \neq 0. 
\end{equation}
Define the maximal operator
\begin{equation*}
    \mathcal{M}_h f(x) = \sup_{t > 0} \Big| \int_{\R} f\big(x_1 - t(s + a_1), x_2 - t(h(s) + a_2) \big)\chi(s)\ud s\Big| 
\end{equation*}
where $\chi \in C^{\infty}_c(\R)$ is non-negative, satisfies $\chi(s) = 1$ for $|s| \leq 1/2$ and has support contained in the interior of $[-1,1]$. To prove the proposition, it suffices to show
\begin{equation}\label{nec prop 2}
      \|\mathcal{M}_h\|_{L^p(\R^2) \to L^p(\R^2)} = \infty \qquad \textrm{for $1 \leq p \leq n$.}
\end{equation}
Furthermore, since the maximal operator is trivially bounded on $L^{\infty}$ it suffices to consider the case $p=n$ only. 

By Taylor expansion and \eqref{nec prop 1}, we have
\begin{equation*}
    |h(s)| \leq D_h \cdot |s|^n \qquad \textrm{for $|s| \leq 1$} \quad \textrm{where} \quad D_h :=  \frac{1}{n!} \sup_{|s| \leq 1} |h^{(n)}(s)|.
\end{equation*}
For $0 < r < 1$ let $f_r := \bbone_{K(r)}$ denote the indicator function of the set
\begin{equation*}
    K(r) := \big\{ y = (y_1, y_2) \in \R^2 : |y_1 - a_1| \leq r \textrm{ and } |y_2 - a_2| \leq D_h \cdot  r^n \big\}
\end{equation*}
and observe that
\begin{equation}\label{nec prop 3}
    \|f_r\|_{L^n(\R^2)} \sim_h r^{(n+1)/n}.
\end{equation}
Now let $\delta^n \leq \lambda \leq 1$ be a dyadic number and suppose $x \in E_{\lambda}(r)$ where
\begin{equation*}
    E_{\lambda}(r) := \Big\{x = (x_1, x_2) \in \R^2 : \Big|x_1 - \frac{a_1}{a_2} \, x_2\Big| \leq \frac{r}{2} \textrm{ and } \lambda \leq \frac{x_2}{a_2} - 1 < 2\lambda \Big\}.
\end{equation*}
If we define $t_x := a_2^{-1}x_2 - 1 \in [\lambda,2\lambda]$, then for any $s \in \R$ satisfying $|s| \leq \frac{1}{2} \cdot \lambda^{(n-1)/n} r$ we have 
\begin{align*}
    |x_1 - t_x(t_x^{-1}s +  a_1) - a_1| &\leq \Big|x_1 - \frac{a_1}{a_2} \, x_2\Big| + |s| \leq r, \\
    |x_2 - t_x(h(t_x^{-1} s) + a_2) - a_2| &= |t_x||h(t_x^{-1} s)| \leq D_h \cdot \lambda^{-(n-1)}|s|^n \leq D_h \cdot r^n.
\end{align*}
From these observations, we conclude that
\begin{equation*}
  \textrm{if $x \in E_{\lambda}(r)$ and $|s| \leq \frac{1}{2} \, \lambda^{(n-1)/n} r$, then }  \big(x_1 - t_x(t_x^{-1}s + a_1), x_2 - t_x(h(t_x^{-1}s) + a_2)\big) \in K(r).
\end{equation*}
 Performing a change of variable in the underlying averaging operator, we deduce that \begin{equation*}
 \mathcal{M}_h f_r(x) \gtrsim \lambda^{-1/n} r \qquad \textrm{for all $x \in E_{\lambda}(r)$,} \end{equation*}
 where here we pick up an extra factor of $\lambda^{-1}$ owing to the Jacobian. Consequently,
\begin{equation}\label{nec prop 4}
    \|\mathcal{M}_h f_r\|_{L^n(\R^2)} \gtrsim  \Big(\sum_{\substack{\lambda \,:\, \mathrm{dyadic}\ \\ r^n \leq \lambda \leq 1}} \lambda^{-1}r^n |E_{\lambda}(r)|\Big)^{1/n} \sim_a |\log r|^{1/n} r^{(n + 1)/n}.
\end{equation}
Comparing \eqref{nec prop 3} and \eqref{nec prop 4}, we see that the ratio of $\|\mathcal{M}_h f_r\|_{L^n(\R^2)}$ and $\|f_r\|_{L^n(\R^2)}$ is unbounded in $r$ and therefore \eqref{nec prop 2} holds for $p=n$, as desired. 
\end{proof}




\appendix




\section{An abstract broad/narrow decomposition}\label{BG appendix}

Here we provide an abstract version of the broad/narrow decomposition in Lemma \ref{broad narrow lemma E}. For the sake of self-containedness of this appendix, we recall some of the definitions introduced in \S\ref{multilinear subsec}.

Let $\fI$ denote the collection of all dyadic subintervals of $[-1,1]$ and for any dyadic number $0 < r \leq 1$ let $\fI(r)$ denote the subset of $\fI$ consisting of all intervals of length $r$. Let $\mathfrak{I}_{\geq r}$ denote the union of the $\mathfrak{I}(\lambda)$ over all dyadic $\lambda$ satisfying $r \leq \lambda \leq 1$. Given any pair of dyadic scales $0 < \lambda_1 \leq \lambda_2 \leq 1$ and $J \in \fI(\lambda_2)$, let $\fI(J;\,\lambda_1)$ denote the collection of all $I \in \fI(\lambda_1)$ which satisfy $I \subseteq J$.

For $d \in \N$ and each dyadic number $0 < \lambda \leq 1$ let $\fI^d_{\mathrm{sep}}(\lambda)$ denote the collection of $d$-tuples of intervals $\vec{I} = (I_1, \dots, I_d) \in \fI(\lambda)^d$ which satisfy the separation condition
\begin{equation*}
    \mathrm{dist}(I_1, \dots, I_d) := \min_{1 \leq \ell_1 < \ell_2 \leq d} \mathrm{dist}(I_{\ell_1}, I_{\ell_2}) \geq \lambda.
\end{equation*}
Given dyadic scales $0 < \lambda_1 \leq \lambda_2 \leq 1$ and $J \in \fI(\lambda_2)$, let $\fI^d_{\mathrm{sep}}(J;\,\lambda_1)$ denote the collection of all $d$-tuples of intervals $\vec{I} = (I_1, \dots, I_d) \in \fI^d_{\mathrm{sep}}(\lambda_1)$ which satisfy $I_\ell \subseteq J$ for all $1 \leq \ell \leq d$.

The dyadic decomposition from \eqref{rsq dyadic dec} is one instance of an `abstract' notion of dyadic decomposition, introduced in the following definition.

\begin{definition} Let $(X,\mu)$ be a measure space and $F \colon X \to \C$ a measurable function and $0 < r \leq 1$. A sequence $(F_I)_{I \in \fI_{\geq r}}$ of measurable functions $F_I \colon X \to \C$ is said to be a \textit{dyadic decomposition of $F$ up to scale $r$} if it satisfies
\begin{equation*}
F_{[0,1]} = F \qquad \textrm{and} \qquad F_{J} = \sum_{I \in \fI(J; \lambda_1)} F_I \qquad \textrm{for all $J \in \fI(\lambda_2)$}
\end{equation*}
whenever $0 < r \leq \lambda_1 \leq \lambda_2 \leq 1$ are dyadic. Here the identities are understood to hold $\mu$ almost everywhere.
\end{definition}

The broad/narrow decomposition result from which Lemma \ref{broad narrow lemma E} follows is the following.   

\begin{lemma}\label{broad narrow lemma} Let $(X,\mu)$ be a measure space, $k \in \N$ with $k \geq 2$ and $\varepsilon > 0$. For all $r > 0$ there exist dyadic numbers $r_{\mathrm{n}}$ and $r_{\mathrm{b}}$ satisfying
\begin{equation}\label{broad narrow equation 1}
   r < r_{\mathrm{n}} \lesssim_{\varepsilon, k} r, \qquad r < r_{\mathrm{b}} \leq 1
\end{equation} 
such that the following holds. If $F \in L^p(X)$ for some $1 \leq p < \infty$ and $(F_I)_{I \in \fI}$ is a dyadic decomposition of $F$ up to scale $r$, then
\begin{equation}\label{broad narrow equation 2}
    \|F\|_{L^p(X)} \lesssim_{\varepsilon, k} r^{-\varepsilon} \Big( \sum_{I \in \fI(r_{\mathrm{n}})}\|F_I\|_{L^p(X)}^p\Big)^{1/p} + r^{-\varepsilon} \Big(\sum_{\substack{J \in \fI(Cr_{\mathrm{b}}) \\ \vec{I} \in \fI^k_{\mathrm{sep}}(J;\,r_{\mathrm{b}})} }\big\|\prod_{j=1}^{k}|F_{I_{j}}|^{1/k}\big\|_{L^p(X)}^{p} \Big)^{1/p},
\end{equation}
where $C = C_{\varepsilon, k} \geq 1$ is a dyadic number depending only on $\varepsilon$ and $k$. 
\end{lemma} 

The intervals $I \in \fI(r_{\mathrm{n}})$ are referred to as \textit{narrow intervals} whilst the $k$-tuples of intervals $\vec{I} \in \fI^k_{\mathrm{sep}}(I;\,r_{\mathrm{b}})$ are referred to as \textit{broad interval tuples}.

The key ingredient in the proof of Lemma~\ref{broad narrow lemma} is a 1-parameter variant of the Bourgain--Guth decomposition from \cite{Bourgain2011} due to Ham--Lee \cite{Ham2014}.

\begin{lemma}[Ham--Lee \cite{Ham2014}]\label{Ham--Lee lemma} Let $1 \leq p < \infty$ and $k \in \N$ with $k \geq 2$. Suppose $0 <\ell_1, \dots,  \ell_{k-1} \leq 1$ are dyadic numbers such that
\begin{equation*}
1 =: \ell_0 \geq \ell_1 \geq \dots \geq \ell_{k-1} > 0.
\end{equation*}
If $(X,\mu)$ is a measure space and $(F_I)_{I \in \fI}$ is a dyadic decomposition of $F \in L^p(X)$, then for any $\ell>0$,
\begin{align*}
\Big(\sum_{J \in \fI(\ell)} \|F_J\|_{L^p(X)}^p\Big)^{1/p} &\leq 4 \sum_{i=1}^{k-1} \ell_{i-1}^{-2(i-1)} \Big(\sum_{I \in \fI(\ell_i \ell)}\|F_I\|_{L^p(X)}^p\Big)^{1/p} \\
& \quad + \ell_{k-1}^{-2(k-1)} \Big(\sum_{\substack{J \in \fI(\ell) \\ \vec{I} \in \fI^k_{\mathrm{sep}}(J;\ell_{k-1} \ell)}} \big\|\prod_{j=1}^{k}|F_{I_j}|^{1/k}\big\|_{L^p(X)}^p\Big)^{1/p}.
\end{align*}
\end{lemma}

Rather than working in the relatively abstract setting of dyadic decompositions of measurable functions, Ham--Lee \cite[ Lemma 2.8]{Ham2014} apply the decomposition only in the concrete setting of Fourier extension operators associated to space curves. However, the proof is elementary, relying little on the exact form of the extension operator, and can easily be adapted to yield Lemma~\ref{Ham--Lee lemma}. For completeness, the details are presented at the end of the section. 

Lemma~\ref{broad narrow lemma} is deduced by applying Lemma~\ref{Ham--Lee lemma} iteratively, for appropriately chosen dyadic scales $\ell_1, \dots, \ell_{k-1}$.

\begin{proof}[Proof of Lemma~\ref{broad narrow lemma}] Fix $\varepsilon > 0$ and $k \in \N$ with $N \geq 2$. Define the dyadic scales $1 \geq \ell_1 \geq  \cdots \geq  \ell_{k-1} > 0$ recursively so as to satisfy
\begin{equation*}
  \frac{\log \big(4 \vee (k-1)\big)}{\log \ell_1^{-1}} \leq \frac{\varepsilon}{6}, \qquad  \frac{\log \ell_{j-1}^{-2(j-1)}}{\log \ell_j^{-1}} \leq \frac{\varepsilon}{6} \quad \textrm{for $2 \leq j \leq k-1$.}
\end{equation*}
Now fix $r > 0$, $F \in L^p(X)$ and $(F_I)_{I \in \fI}$ a dyadic decomposition of $F$. If $r \gtrsim_{\varepsilon, k} 1$, then the desired result immediately follows from the triangle inequality and so $r$ may be assumed to be smaller than some small constant $c_{\varepsilon, k}$, depending only on $\varepsilon$ and $k$ and chosen for the purposes of the forthcoming argument; in particular we can assume $\ell_{k-1} > r$.

Let $\mathcal{W}$ denote the set of all finite words formed from the alphabet $\{1, \dots, k-1\}$. Given any $w \in \mathcal{W}$ and $1 \leq j \leq k-1$ write $[w]_j$ for the number of occurrences of $j$ in $w$ and $|w| := [w]_1 + \cdots + [w]_{k-1}$ for the length of the word.

 Let $\ell^{w} := \prod_{j=1}^{k-1} \ell_j^{[w]_j}$ for any $w \in \mathcal{W}$ and define 
\begin{equation*}
    \mathcal{A}(r) := \big\{ \alpha \in \mathcal{W} : r < \ell^{\alpha} \leq r/\ell_{k-1} \big\}, \quad \mathcal{B}(r) := \{ \beta \in \mathcal{W} : \ell^{\beta} > r/\ell_{k-1} \}.
\end{equation*}
Finally, for each $N \in \N_0$ define
\begin{gather*}
     \mathcal{A}_{\leq N}(r) := \big\{ \alpha \in \mathcal{A}(r): |\alpha| \leq N \big\}, \quad \mathcal{B}_{\leq N}(r) := \{ \beta \in \mathcal{B}(r) : |\beta| \leq N \}, \\
      \mathcal{A}_N(r) := \{ \alpha \in \mathcal{A}(r) : |\alpha| = N \}, \qquad \mathcal{B}_N(r) := \{ \beta \in \mathcal{B}(r) : |\beta| = N \}.
\end{gather*}

An iterative application of Lemma~\ref{Ham--Lee lemma} yields the following key claim.

\begin{claim} For all $N \in \N_0$,
\begin{align}\label{iteration claim}
\|F\|_{L^p(X)} &\leq \sum_{\alpha \in \mathcal{A}_{\leq N}(r) \cup \mathcal{B}_N(r)} M_{\varepsilon, k}^{\alpha} \Big(\sum_{I \in \fI(\ell^{\alpha})}\|F_I\|_{L^p(X)}^p\Big)^{1/p}\\ 
\nonumber
 & + \ell_{k-1}^{-2(k-1)}\sum_{\beta \in \mathcal{B}_{\leq N-1}(r)} M_{\varepsilon, k}^{\beta} \Big(\sum_{\substack{J \in \fI(\ell^{\beta}) \\\vec{I} \in \fI^k_{\mathrm{sep}}(J;\ell_{k-1} \ell^{\beta}) }}\big\|\prod_{j=1}^{k}|F_{I_j}|^{1/k}\big\|_{L^p(X)}^p\Big)^{1/p}.
\end{align}
where $M_{\varepsilon, k}^{\alpha} := 4^{|\alpha|} \prod_{j=1}^{k-1} \ell_{j-1}^{-2(j-1)[\alpha]_j}$.
\end{claim}

\begin{proof}[Proof (of Claim)] The proof is by induction on $N$. The case $N = 0$ is vacuous and thus one may assume, by way of induction hypothesis, that \eqref{iteration claim} holds for some $N \geq 0$. It remains to establish the inductive step.

Consider the terms on the right-hand side of \eqref{iteration claim} of the form 
\begin{equation*}
\Big(\sum_{I \in \fI(\ell^{\beta})}\|F_I\|_{L^p(X)}^p\Big)^{1/p} \qquad \textrm{for $\beta \in \mathcal{B}_N(r)$.}    
\end{equation*}
 Applying Lemma~\ref{Ham--Lee lemma} to each of these terms,
\begin{align*}
\|F\|_{L^p(X)} &\leq \sum_{\alpha \in \mathcal{A}_{\leq N}(r)} M_{\varepsilon, k}^{\alpha} \Big(\sum_{I \in \fI(\ell^{\alpha})}\|F_I\|_{L^p(X)}^p\Big)^{1/p} \\
& + \sum_{\beta \in \mathcal{B}_{N}(r)}  M_{\varepsilon, k}^{\beta} 4\sum_{i=1}^{k-1} \ell_{i-1}^{-2(i-1)} \Big(\sum_{I \in \fI(\ell_i \ell^{\beta})}\|F_I\|_{L^p(X)}^p\Big)^{1/p}\\
& + \ell_{k-1}^{-2(k-1)}\sum_{\beta \in \mathcal{B}_{\leq N}(r)} M_{\varepsilon, k}^{\beta} \Big(\sum_{\substack{J \in \fI(\ell^{\beta}) \\ \vec{I} \in \fI^k_{\mathrm{sep}}(J;\ell_{k-1} \ell^{\beta})} }\big\|\prod_{j=1}^{k}|F_{I_j}|^{1/k}\big\|_{L^p(X)}^p \Big)^{1/p}.
\end{align*}
From the definitions, 
\begin{equation*}
    \mathcal{A}_{\leq N+1}(r) \cup \mathcal{B}_{N+1}(r) = \mathcal{A}_{\leq N}(r) \cup \mathcal{A}_{N+1}(r) \cup \mathcal{B}_{N+1}(r),
\end{equation*}
where the union is disjoint. Furthermore, the set $\mathcal{A}_{N+1}(r) \cup \mathcal{B}_{N+1}(r)$ precisely corresponds to the set of words obtained by adding a single letter to one of the words in $\mathcal{B}_{ N}(r)$. Combining these observations, the induction readily closes.
\end{proof}

 Using the claim, the proof of Lemma~\ref{broad narrow lemma} quickly follows from the choice of scales $\ell_j$. Indeed, first observe that for $N := \max_{\alpha \in \mathcal{A}(r)} |\alpha|$ it follows that $\mathcal{B}_N(r) = \emptyset$ and thus
\begin{equation*}
    \mathcal{A}_{\leq N}(r) \cup \mathcal{B}_N(r) = \mathcal{A}(r) \qquad \textrm{and} \qquad \mathcal{B}_{\leq N-1}(r) = \mathcal{B}(r).  
\end{equation*}
Note that each $w \in \mathcal{A}(r) \cup \mathcal{B}(r)$ satisfies $(\ell^{w})^{-1} < r^{-1}$ and therefore
\begin{equation}\label{how many steps?}
   |w| \log \ell_1^{-1} \leq \sum_{j=1}^{k-1} [w]_j \log \ell_j^{-1}  \leq \log r^{-1}. 
\end{equation}
By the choice of $\ell_1$, it follows that 
\begin{equation}\label{4 to w}
    4^{|w|} \leq 4^{\log r^{-1}/\log \ell_1^{-1}} = r^{-\log 4/\log \ell_1^{-1}} \leq r^{-\varepsilon/6},
\end{equation}
whilst, similarly,
\begin{equation*}
\# \mathcal{A}(r) \cup \mathcal{B}(r)  \leq  \#\{w \in \mathcal{W} :  |w| \leq \log r^{-1}/\log \ell_1^{-1}\} \leq r^{-\log(k-1)/\log \ell_1^{-1}} \leq r^{-\varepsilon/6}. 
\end{equation*}
On the other hand, as a further consequence of \eqref{how many steps?} and the choice of scales $\ell_j$, if $w \in \mathcal{A}(r) \cup \mathcal{B}(r)$, then
\begin{equation}\label{prod of ell's}
    \log \prod_{j=1}^{k-1} \ell_{j-1}^{-2(j-1)[w]_j} = \sum_{j=1}^{k-1} [w]_j\log \ell_j^{-1} \frac{\log \ell_{j-1}^{-2(j-1)}}{\log \ell_j^{-1}} \leq \log r^{-\varepsilon/6}. 
\end{equation}
The estimates \eqref{4 to w} and \eqref{prod of ell's} imply that
\begin{equation*}
M_{\varepsilon,k}^\alpha= 4^{|\alpha|}\prod_{j=1}^{k-1} \ell_{j-1}^{-2(j-1)[w]_j}  \leq r^{-\varepsilon/3}, 
\end{equation*}
where $M_{\varepsilon,k}^\alpha$ are the constants appearing in the above claim. Combining these observations with \eqref{iteration claim} for the choice of $N$ as above,
\begin{align*}
\|F\|_{L^p(X)} &\leq r^{-\varepsilon/3} \sum_{\alpha \in \mathcal{A}(r)} \Big( \sum_{ I \in \fI(\ell^{\alpha})}\|F_I\|_{L^p(X)}^p\Big)^{1/p}\\
& \quad + r^{-\varepsilon/3}\ell_{k-1}^{-2(k-1)} \sum_{\beta \in \mathcal{B}(r)} \Big(\sum_{\substack{J \in \fI(\ell^{\beta}) \\ \vec{I} \in \fI^k_{\mathrm{sep}}(J; \ell_{k-1}\ell^{\beta}) }}\big\|\prod_{j=1}^{k}|F_{I_j}|^{1/k}\big\|_{L^p(X)}^p\Big)^{1/p}.
\end{align*}
Finally, since $\ell_{k-1}^{-1} \lesssim_{\varepsilon, k} 1$ and $\# \mathcal{A}(r), \# \mathcal{B}(r), \leq r^{-\varepsilon/6}$, by pigeonholing there exists some $\alpha_{\mathrm{n}} \in \mathcal{A}(r)$ and $\beta_{\mathrm{b}} \in \mathcal{B}(r)$ such that, if $r_{\mathrm{n}} := \ell^{\alpha_{\mathrm{n}}}$ and $r_{\mathrm{b}} := \ell_{k-1}  \ell^{\beta_{\mathrm{b}}}$ and $C_{\varepsilon,k} := \ell_{k-1}^{-1}$, then the desired inequality \eqref{broad narrow equation 2} holds. It is easy to see that these parameters also satisfy \eqref{broad narrow equation 1} directly from and the relevant definitions.
\end{proof}

To close this section, the proof of Lemma~\ref{Ham--Lee lemma} is presented, following the argument in \cite{Ham2014}.

\begin{proof}[Proof of Lemma~\ref{Ham--Lee lemma}] For notational convenience, given $m \in \N$ and $J \in \fI(\ell)$ define
\begin{equation*}
    \pi^m_J(F)(x) := \max_{\vec{I}^{k-1} \in \fI_{\mathrm{sep}}^m(J;\ell_{k-1}\ell) } \prod_{j=1}^m |F_{I^{m-1}_j}(x)|^{1/m}. 
\end{equation*}
When $m = 1$ this reduces to $\pi^m_J(F)(x) = |F_J(x)|$. The main step in the proof of Lemma~\ref{Ham--Lee lemma} is the following pointwise estimate. 

\begin{claim} For all $m \in \N$ and $J \in \fI(\ell)$, the pointwise estimate
\begin{equation*}
  \pi^m_J(F)(x)  \leq 4 \max_{\substack{I^m \in \fI(J;\ell_m\ell)}} |F_{I^m}(x)| + \ell_m^{-2} \, \pi^{m+1}_J(F)(x)
\end{equation*}
holds for $\mu$-almost all $x \in X$.
\end{claim}

\begin{proof} Fix $x \in X$ and $\vec{I}^{m-1} = (I^{m-1}_1,\dots,I^{m-1}_m) \in \fI^m(\ell_{m-1}\ell)$ with $I^{m-1}_j \subset J$ for $1 \leq j \leq m$. For each $j$ there exists an interval $I^{m,*}_j \in \fI(\ell_m \ell)$ satisfying
\begin{equation*}
    I^{m,*}_j \subset I^{m-1}_j \quad \textrm{and} \quad |F_{I^{m,*}_j}(x)| = \max_{\substack{I^{m}_j \in \fI(I_j^{m-1};\ell_m\ell)}} |F_{I^{m}_j}(x)|. 
\end{equation*}

There are two cases to consider:\\

\noindent \textbf{Narrow case:} Either one of the following two conditions hold:
\begin{enumerate}[i)]
    \item For all $1 \leq j \leq m$, if $I^{m}_j \in \fI(\ell_m\ell)$ satisfies $I^m_j \subset I^{m-1}_j$ and $\mathrm{dist}(I^m_j, I^{m,*}_j) \geq \ell_m \ell$, then 
    \begin{equation*}
        |F_{I^m_j}(x)| \leq \Big(\frac{\ell_{m}}{\ell_{m-1}}\Big)|F_{I^{m,*}_j}(x)|.
    \end{equation*}
    \item  The selected interval $I_j^{m,*} \in \fI (\ell_m \ell)$ above satisfies
    \begin{equation*}
        \min_{1 \leq j \leq m} |F_{I^{m,*}_j}(x)| \leq \Big(\frac{\ell_{m}}{\ell_{m-1}}\Big)^m\max_{1 \leq j \leq m} |F_{I^{m,*}_j}(x)|.
    \end{equation*}
\end{enumerate}

\noindent \textbf{Broad case:} The conditions of the narrow case fail.\\

\subsubsection*{The narrow case} If condition i) of the narrow case holds, then 
\begin{equation*}
    |F_{I^{m-1}_j}(x)| \leq 3 |F_{I^{m, *}_j}(x)| + \sum_{\substack{ I^m_j \in \fI(\ell_m\ell),\, I^m_j \subset I^{m-1}_j \\ \mathrm{dist}(I^m_j, I^{m,*}_j) \geq \ell_m\ell }}|F_{I^{m}_j}(x)| \leq 4 |F_{I_j^{m,*}}(x)|,
\end{equation*}
since there are at most $\ell_{m-1}/\ell_m$ intervals $I^m_j \in \fI(\ell_m\ell)$ contained in $I^{m-1}_j$. Thus, in this case, 
\begin{equation}\label{Ham--Lee 1}
    \prod_{j=1}^m |F_{I^{m-1}_j}(x)|^{1/m} \leq 4 \max_{I^m \in \fI(J;\ell_m\ell) } |F_{I^{m}}(x)|.
\end{equation}

Now suppose that condition ii) of the narrow case holds. Thus,
\begin{equation*}
    \prod_{j=1}^m |F_{I^{m-1}_j}(x)|^{1/m} \leq \Big(\frac{\ell_{m-1}}{\ell_{m}}\Big) \prod_{j=1}^m |F_{I^{m,*}_j}(x)|^{1/m} \leq \max_{1 \leq j \leq m} |F_{I^{m,*}_j}(x)|
\end{equation*}
where the first inequality follows since there are at most $\ell_{m-1}/\ell_m$ intervals $I_j^m \in \mathfrak{I}(\ell_m \ell)$ contained in $I_j^{m-1}$. Once again, \eqref{Ham--Lee 1} holds (in fact, it holds with a constant 1 rather 4). Hence, a favourable estimate holds in the narrow case.

\subsubsection*{The broad case} Suppose the broad case holds. By definition, condition i) from the narrow fails. Consequently, there exists some $1 \leq j_0 \leq m$ and an interval $I_{j_0}^{m,**} \in \fI(\ell_m\ell)$ satisfying 
\begin{equation*}
   I_{j_0}^{m,**} \subseteq I_{j_0}^{m-1}, \qquad \mathrm{dist} (I_j^{m,**}, I_j^{m,*}) \geq \ell_m \ell \qquad  \textrm{and} \qquad |F_{I_{j_0}^{m,*}}(x)| \leq \Big(\frac{\ell_{m-1}}{\ell_m}\Big)|F_{I_{j_0}^{m,**}}(x)|.
\end{equation*}
On the other hand, condition ii) from the narrow case also fails and, consequently,
\begin{equation*}
    \max_{1 \leq j \leq m} |F_{I^{m,*}_j}(x)| \leq \Big(\frac{\ell_{m-1}}{\ell_m}\Big)^m |F_{I^{m,*}_{j_0}}(x)| \leq \Big(\frac{\ell_{m-1}}{\ell_m}\Big)^{m+1}|F_{I_{j_0}^{m,**}}(x)|.
\end{equation*}
Thus, for each $1 \leq j \leq m$, it follows that
\begin{equation*}
|F_{I^{m,*}_j}(x)|^{1/m} \leq \Big(\frac{\ell_{m-1}}{\ell_m}\Big)^{1/m}|F_{I^{m,*}_j}(x)|^{1/(m+1)}|F_{I_{j_0}^{m,**}}(x)|^{1/m(m+1)}.
\end{equation*}
Finally, taking the product of the above estimate over all $j$, one deduces that
\begin{align*}
    \prod_{j=1}^m |F_{I^{m-1}_j}(x)|^{1/m} &\leq \Big(\frac{\ell_{m-1}}{\ell_{m}}\Big) \prod_{j=1}^m |F_{I^{m,*}_j}(x)|^{1/m} \\
    &\leq \Big(\frac{\ell_{m-1}}{\ell_{m}}\Big)^2 \Big(\prod_{j=1}^m |F_{I^{m,*}_j}(x)|^{1/(m+1)}\Big) |F_{I_{j_0}^{m,**}}(x)|^{1/(m+1)} \\
     &\leq \ell_m^{-2} \, \pi^{m+1}_J(F)(x),
\end{align*}
where in the last inequality we use the separation condition. Hence, in the broad case a favourable estimate also holds. 
\end{proof}

By repeated application of the claim and the relation $\ell_1 \geq \cdots \geq \ell_{k-1}$,
\begin{equation*}
    |F_J(x)| \leq 4\sum_{m=1}^{k-1} \ell_{m-1}^{-2(m-1)} \max_{I^m \in \fI(J;\ell_m\ell)} |F_{I^m}(x)| + \ell_{k-1}^{-2(k-1)} \pi^{k}_J(F)(x)
\end{equation*}
for $\mu$-almost every $x \in X$. Bounding all the maxima in the above display by the corresponding $\ell^p$ expressions and integrating over $x \in X$, one deduces that 
\begin{align*}
    \|F_J\|_{L^p(X)} &\leq 4 \sum_{m=1}^{k-1} \ell_{m-1}^{-2(m-1)} \Big(\sum_{I^m \in \fI(J;\ell_m\ell)} \|F_{I^m}\|_{L^p(X)}^p\Big)^{1/p} \\
    & \qquad + \ell_{k-1}^{-2(k-1)} \Big(\sum_{\vec{I} \in \fI^k_{\mathrm{sep}}(J;\ell_{k-1}\ell)} \big\|\prod_{j=1}^k |F_{I_j}|^{1/k}\big\|_{L^p(X)}^p\Big)^{1/p}
\end{align*}
Finally, taking a $\ell^p$ sum over $J$ of both sides of the above inequality and applying the triangle inequality concludes the proof.  
\end{proof}




\section{A pointwise square function inequality}\label{RdF appendix}

Here we provide the simple proof of Lemma~\ref{RdF lem}, which is a slight extension of an argument due to Rubio de Francia \cite{RdF1983}. Given $G \colon \Z^m \to \R^n$ define
\begin{equation*}
\vvvert G \vvvert := \sup_{\nu_2 \in \Z^m} \sum_{\nu_1 \in \Z^m} e^{- |G(\nu_1) - G(\nu_2)|/2}.
\end{equation*}
By rescaling and a simple limiting argument, Lemma~\ref{RdF lem} is a consequence of the following pointwise bound.  

\begin{lemma} Let $\psi \in \mathscr{S}(\widehat{\R}^n)$ and $G \colon \Z^m \to \R^n$. For all $M$, $N \in \N$ the pointwise inequality \begin{equation*}
\sum_{\nu \in \Z^m\cap[-M,M]^m} \big|\psi\big(D - G(\nu)\big)f(x)\big|^2 \lesssim_{\psi, N} \vvvert G \vvvert \int_{\R^n}  |f(x - y)|^2 (1 + |y|)^{-N}\,\ud y 
\end{equation*}
holds for all $f \in \mathscr{S}(\R^n)$, with an implied constant independent of $M$.
\end{lemma}

\begin{proof} Let $a = (a_{\nu})_{\nu \in \Z^m}$ be a sequence supported in $\Z^m \cap [-M, M]^m$ satisfying $\|a\|_{\ell^2} = 1$. Consider the function
\begin{equation*}
  \sum_{\nu \in \Z^m} a_{\nu} \psi\big(D - G(\nu)\big)f(x) = \mathcal{K} \ast f(x)
\end{equation*}
where the kernel $\mathcal{K}$ is given by
\begin{equation*}
   \mathcal{K}(x) := \frac{1}{(2\pi)^n} \int_{\widehat{\R}^n}e^{i\inn{x}{\xi}} \sum_{\nu \in \Z^m} a_{\nu} \psi\big(\xi - G(\nu)\big)\,\ud \xi =\Big[ \sum_{\nu \in \Z^m} a_{\nu} e^{i \inn{x}{G(\nu)}} \Big]\widecheck{\psi}(x).
\end{equation*}
By duality, it suffices to show
\begin{equation*}
   |\mathcal{K} \ast f(x)|^2 \leq \vvvert G \vvvert   \int_{\R^n}  |f(x - y)|^2 (1 + |y|)^{-N}\,\ud y.
\end{equation*}

Applying the Cauchy--Schwarz inequality,
\begin{equation*}
    |\mathcal{K} \ast f(x)|^2 \leq \int_{\R^n} \Big| \sum_{\nu \in \Z^m} a_{\nu} e^{i \inn{y}{G(\nu)}} \Big|^2 |\widecheck{\psi}(y)|\,\ud y \int_{\R^n} |f(x-y)|^2 |\widecheck{\psi}(y)|\,\ud y
\end{equation*}
and so, in view of the rapid decay of $\widecheck{\psi}$, the problem is further reduced to showing 
\begin{equation*}
    \int_{\R^n} \Big| \sum_{\nu \in \Z^m} a_{\nu} e^{i \inn{y}{G(\nu)}} \Big|^2 |\widecheck{\psi}(y)|\,\ud y \lesssim \vvvert G \vvvert. 
\end{equation*}

Since  $\psi \in \mathscr{S}(\widehat{\R}^n)$ we have $|\widecheck\psi (y)|\lesssim \phi(y)$ where $\phi(z):= (1+z^2)^{-n-1}$.  Consider $\phi(z)$ for $|\Im(z)|\le 1/2$ and observe that, by contour integration, $|\widehat{\phi}(\xi) |\lesssim e^{-|\xi|/2} $ for $\xi\in \widehat \R^n$. Hence  
\begin{align*}
   \int_{\R^n} \Big| \sum_{\nu \in \Z^m} a_{\nu} e^{i \inn{y}{G(\nu)}} \Big|^2 |\widecheck{\psi}(y)|\,\ud y &\lesssim  \sum_{\nu_1, \nu_2 \in \Z^m} \overline{a_{\nu_1}}a_{\nu_2}  \widehat{\phi}\big(G(\nu_1) - G(\nu_2)\big) \\ &\lesssim_{\psi}   \sum_{\nu_1, \nu_2 \in \Z^m} |a_{\nu_1}||a_{\nu_2}|  e^{-|G(\nu_1)-G(\nu_2)|/2}.
\end{align*} 
The right-hand side of the above inequality is then bounded by $\vvvert  G \vvvert$ via the Cauchy--Schwarz inequality and the Schur test, as $\| a \|_{\ell^2}=1$.
\end{proof}




\section{Derivative bounds for implicitly defined functions}

Let $\Omega$, $I \subseteq \R$ be open intervals and $G \colon \Omega \times I \to \C$ a $C^{\infty}$ mapping.  Suppose $\partial_y G(x,y)$ is non-vanishing on $\Omega \times I$ and $y \colon \Omega \to I$ is a $C^{\infty}$ mapping such that
\begin{equation*}
    G(x,y(x)) = 0 \qquad \textrm{for all $x \in \Omega$.}
\end{equation*}

\begin{lemma}\label{imp deriv lem} Let $G \colon \Omega \times I \to \C$ and $y \colon \Omega \to I$ be as above and suppose $A, M_1$, $M_2 > 0$ are constants such that
\begin{equation}\label{imp deriv eq}
\left\{\begin{array}{rcl}
    \big|(\partial_y G)(x,y(x))\big| &\geq& AM_2, \\[5pt]
    \big|(\partial_x^{\alpha_1} \partial_y^{\alpha_2}G)(x,y(x))\big| &\lesssim_{\alpha} & AM_1^{\alpha_1} M_2^{\alpha_2} \qquad \textrm{for all $\alpha \in \N_0^2\setminus\{0\}$.} 
\end{array} \right. 
\end{equation}
Then the function $y$ satisfies
\begin{equation}\label{der bounds implicit function}
    |y^{(j)}(x)| \lesssim_j M_1^j M_2^{-1} \qquad \textrm{for all $j \in \N$.}
\end{equation}
Consequently, for all $C^{\infty}$ functions $H \colon \Omega \times I \to \C$ for which there exists some constant $B > 0$ such that
\begin{equation}\label{imp deriv hyp}
    |(\partial_x^{\alpha_1}\partial_y^{\alpha_2}H)(x,y(x))| \lesssim_N B M_1^{\alpha_1} M_2^{\alpha_2} \qquad \textrm{for all $\alpha \in \N_0^2 \setminus \{0\}$,}
\end{equation}
one has
\begin{equation}\label{Faa di Bruno eq 2}
    \Big|\frac{\ud^N}{\ud x^N} H(x, y(x))\Big| \lesssim_N B M_1^N \qquad \textrm{for all $N \in \N$.}
\end{equation}
\end{lemma}

Before giving the proof of Lemma~\ref{imp deriv lem}, we make some preliminary observations. A simple induction argument shows that there exists a sequence of coefficients $(C_{\alpha, d})_{d \in \N_0^j}$, depending only on $j$ and $\alpha$, such that for all $C^{\infty}$ functions $H \colon \Omega \times I \to \C$ the identity
\begin{equation}\label{Faa di Bruno eq}
    \frac{\ud^j}{\ud x^j} H(x, y(x)) = \sum_{\substack{\alpha \in \N_0^2 \setminus\{0\} \\ \alpha_1, \alpha_2 \leq j}} (\partial_{x}^{\alpha_1} \partial_y^{\alpha_2}H)(x,y(x)) \sum_{\substack{   d_1 + \cdots + j d_j = j - \alpha_1 \\  d_1 + \cdots + d_j = \alpha_2}} C_{\alpha, d} \prod_{i=1}^j y^{(i)}(x)^{d_i}
\end{equation}
holds. The precise values of the $C_{\alpha,d}$ are given by the multivariate Fa\`a di Bruno formula: see \cite[Theorem 4.2]{LP}. Similarly, for $1 \leq k \leq |\alpha|$ there exists a sequence of coefficients $(C_{k, e})_{e \in \mathcal{E}(\alpha,k)}$, depending only on $\alpha$, such that
\begin{equation}\label{Faa di Bruno ratio}
    \partial_x^{\alpha_1} \partial_y^{\alpha_2}\big[(\partial_yG)(x,y)^{-1}\big] = \sum_{k = 1}^{|\alpha|} (\partial_yG)(x,y)^{-k-1}  \sum_{e \in \mathcal{E}(\alpha,k)}  C_{k,e} \prod_{\beta \preceq \alpha} (\partial_x^{\beta_1} \partial_y^{\beta_2  + 1}G)(x,y)^{e_{\beta}}
\end{equation}
where 
\begin{equation*}
    \mathcal{E}(\alpha,k) := \Big\{ e = (e_{\beta})_{\beta \preceq \alpha} : e_{\beta} \in \N_0 \textrm{ for all $\beta \preceq \alpha$ and }  \sum_{\beta \preceq \alpha} \beta_{\ell} \cdot e_{\beta} = \alpha_{\ell} \textrm{ for $\ell = 1, 2$, } \sum_{\beta \preceq \alpha} e_{\beta} = k \Big\}
\end{equation*}
and the notation $\beta \preceq \alpha$ refers to those $\beta \in \N_0^2 \setminus \{0\}$ which satisfy $\beta_{\ell} \leq \alpha_{\ell}$ for $\ell = 1,2$.
Once again, the precise values of the $C_{k,e}$ are given by the multivariate Fa\`a di Bruno formula.

Both identities \eqref{Faa di Bruno eq} and \eqref{Faa di Bruno ratio} play a r\^ole in the proof of Lemma~\ref{imp deriv lem}. 
\begin{proof} By scaling, it suffices to show the case $A=1$. The proof of \eqref{der bounds implicit function} proceeds by (strong) induction on $j$. By implicit differentiation, 
\begin{equation}\label{imp deriv eq 1}
    y'(x) = Q(x, y(x)) \quad \textrm{where} \quad    Q(x,y) :=  - (\partial_xG)(x,y) \cdot (\partial_yG)(x,y)^{-1} \quad \textrm{for $(x,y) \in \Omega \times I$.}
\end{equation}
Thus, the $j=1$ case is an immediate consequence of this identity together with the hypothesised bounds \eqref{imp deriv eq}. Now let $j \geq 1$ and suppose $|y^{(i)}(x)| \lesssim_i M_1^i M_2^{-1}$ holds for all $1 \leq i \leq j$.

To bound the higher order derivative $y^{(j+1)}$ we make use of the differential identity \eqref{Faa di Bruno eq}, taking $H := Q$. In particular, \eqref{Faa di Bruno eq} together with \eqref{imp deriv eq 1} directly imply that
\begin{equation}\label{imp deriv eq 5}
    y^{(j+1)}(x) = \sum_{\substack{\alpha \in \N_0^2 \setminus\{0\} \\ \alpha_1, \alpha_2 \leq j}} (\partial_{x}^{\alpha_1} \partial_y^{\alpha_2} Q)(x,y(x)) \sum_{\substack{   d_1 + \cdots + j d_j = j - \alpha_1 \\  d_1 + \cdots + d_j = \alpha_2}} C_{\alpha, d} \prod_{i=1}^j y^{(i)}(x)^{d_i}.
\end{equation}
The bound \eqref{der bounds implicit function} is now reduced to showing 
\begin{equation}\label{imp deriv eq 4}
    |(\partial_x^{\alpha_1} \partial_y^{\alpha_2}Q)(x,y(x))| \lesssim_{\alpha}  M_1M_2^{-1}M_1^{\alpha_1}M_2^{\alpha_2}.
\end{equation}
Indeed, once \eqref{imp deriv eq 4} is established, one may use this inequality to bound the derivatives of $Q$ appearing on the right-hand side of \eqref{imp deriv eq 5} and the induction hypothesis to bound the $y^{(i)}(x)$ terms. Consequently, one deduces that
\begin{equation*}
|y^{(j+1)}(x)| \lesssim_j  M_1^{j + 1}M_2^{- 1}.
\end{equation*}
This closes the induction and completes the proof of \eqref{der bounds implicit function}.

Turning to the proof of \eqref{imp deriv eq 4}, note that \eqref{Faa di Bruno ratio} and the hypothesised bounds \eqref{imp deriv eq} imply
\begin{equation}\label{imp deriv eq 2}
    \big|\partial_x^{\alpha_1} \partial_y^{\alpha_2}\big[(\partial_yG)(x,y)^{-1}\big]|_{y = y(x)}\big| \lesssim_{\alpha} M_2^{-1}M_1^{\alpha_1}M_2^{\alpha_2} \qquad \textrm{for all $\alpha \in \N_0^2 \setminus \{0\}$.}
\end{equation}
On the other hand, \eqref{imp deriv eq} immediately implies that 
\begin{equation}\label{imp deriv eq 3}
       \big|\partial_x^{\alpha_1} \partial_y^{\alpha_2}(\partial_x G)(x,y)|_{y = y(x)}\big| \lesssim_{\alpha} M_1M_1^{\alpha_1}M_2^{\alpha_2} \qquad \textrm{for all $\alpha \in \N_0^2 \setminus \{0\}$.} 
\end{equation}
Combining \eqref{imp deriv eq 2} and \eqref{imp deriv eq 3} with the Leibniz rule one obtains \eqref{imp deriv eq 4}.

The bound \eqref{Faa di Bruno eq 2} is a simple consequence of \eqref{der bounds implicit function} and \eqref{imp deriv hyp} via the formula \eqref{Faa di Bruno eq}.
\end{proof}

Lemma~\ref{imp deriv lem} immediately implies the following multivariate extension. Let $\Omega \subseteq \R^n$ be an open set, $I \subseteq \R$ an open interval and $G \colon \Omega \times I \to \C$ a $C^{\infty}$ mapping, for some $N \in \N$. Suppose $\partial_y G(\bm{x}, y)$ is non-vanishing on $\Omega \times I$ and $y \colon \Omega \to I$ is a $C^{\infty}$ mapping such that
\begin{equation*}
    G(\bm{x},y(\bm{x})) = 0 \qquad \textrm{for all $\bm{x} \in \Omega$.}
\end{equation*}
For $\be \in S^{n-1}$ let $\nabla_{\be}$ denote the directional derivative operator with respect to $\bm{x}$ in the direction of $\be$. Suppose $A$, $M_1$, $M_2 > 0$ are constants such that 
\begin{equation}\label{multi imp deriv}
\left\{\begin{array}{rcl}
    \big|(\partial_y G)(\bm{x},y(\bm{x}))\big| &\geq& A M_2, \\[5pt]
    \big|(\nabla_{\be}^{\alpha_1}\partial_y^{\alpha_2}G)(\bm{x},y(\bm{x}))\big| &\lesssim_N& A M_1^{\alpha_1} M_2^{\alpha_2} 
\end{array} \right.  \qquad \textrm{for all $\alpha \in \N_0^2\setminus\{0\}$ and all $\bm{x} \in \Omega$.}
\end{equation}
Then the function $y$ satisfies
\begin{equation}\label{multi imp der bound}
    |\nabla_{\be}^{N} y(\bm{x})| \lesssim_N M_1^N M_2^{-1} \qquad \textrm{for all $\bm{x} \in \Omega$ and all $N \in \N_0$.}
\end{equation}
Similarly, \eqref{Faa di Bruno eq 2} has a multivariate extension. In particular, suppose, in addition to the above, that $H \colon \Omega \times I \to \C$ a $C^{\infty}$ mapping and $B > 0$ is a constant such that
\begin{equation}\label{multi imp deriv hyp}
    \big|(\nabla_{\be}^{\alpha_1}\partial_y^{\alpha_2}H)(\bm{x},y(\bm{x}))\big| \lesssim_N B M_1^{\alpha_1} M_2^{\alpha_2} \qquad \textrm{for all $\alpha \in \N_0^2\setminus\{0\}$.}
\end{equation}
Then it follows from \eqref{Faa di Bruno eq 2} that
\begin{equation}\label{multi Faa di Bruno eq 2}
    \big|\nabla_{\be}^N H(\bm{x}, y(\bm{x}))\big| \lesssim_N B M_1^N  \qquad \textrm{for all $\bm{x} \in \Omega$ and all $N \in \N$.}
\end{equation}

For the purposes of this paper, we are interested in the special case where $G$, $H \colon \R^n \times I \to \R$ are both linear in the $\bm{x}$ variable. Thus, $G$ and $H$ are of the form
\begin{equation*}
    G(\bm{x},y) = \inn{g(y)}{\bm{x}}, \quad  H(\bm{x},y) = \inn{h(y)}{\bm{x}},
\end{equation*}
for some $C^{\infty}$ functions $g$, $h \colon I \to \R^n$. Furthermore, the conditions in \eqref{multi imp deriv} can be written as
\begin{equation}\label{multi imp deriv 2}
    \left\{\begin{array}{rcl}
    |\inn{g'\circ y(\bm{x})}{\bm{x}}| &\geq& A M_2, \\[2pt]
    |\inn{g^{(N)}\circ y(\bm{x})}{\bm{x}}| &\lesssim_N& A M_2^{N}  \\[2pt]
    |\inn{g^{(N)}\circ y(\bm{x})}{\be}| &\lesssim_N& A M_1 M_2^{N} 
\end{array} \right. \qquad\textrm{for all $N \in \N$ and all $\bm{x} \in \Omega$} 
\end{equation}
and the condition in \eqref{multi imp deriv hyp} can be written as
\begin{equation}\label{multi imp deriv 3}
    \left\{\begin{array}{rcl}
    |\inn{h^{(N)}\circ y(\bm{x})}{\bm{x}}| &\lesssim_N& B M_2^{N}  \\[2pt]
    |\inn{h^{(N)}\circ y(\bm{x})}{\be}| &\lesssim_N& B M_1 M_2^{N} 
\end{array} \right. \qquad\textrm{for all $N \in \N$ and all $\bm{x} \in \Omega$.} 
\end{equation}

\begin{example}[Application to Lemma~\ref{J=3 ker lem}]\label{deriv ex}
Let $\gamma \in \mathfrak{G}_3(\delta_0)$, and $\theta_2: \widehat{\R}^3 \backslash\{0\} \to I_0$ satisfying
\begin{equation*}
    \inn{\gamma'' \circ \theta_2 (\xi)}{\xi}=0.
\end{equation*}
We apply the previous result with $g=\gamma''$ and $h=\gamma'$. If $B \leq A$ the conditions \eqref{multi imp deriv 2} and \eqref{multi imp deriv 3} read succinctly as
\begin{equation*}
    \left\{\begin{array}{rcl}
    |\inn{\gamma^{(3)}\circ \theta_2(\xi)}{\xi}| &\geq& A M_2, \\[2pt]
    |\inn{\gamma^{(1+N)}\circ \theta_2(\xi)}{\xi}| &\lesssim_N& B M_2^{N}  \\[2pt]
    |\inn{\gamma^{(1+N)}\circ \theta_2(\xi)}{\be}| &\lesssim_N& B M_1 M_2^{N} 
\end{array} \right.
\qquad\textrm{for all $N \in \N$ and all $\xi \in \Omega \subset \hat{\R}^3 \backslash \{0\}$},
\end{equation*}
which imply
\begin{equation*}
    |\nabla_{\bm{e}}^N \theta_2(\xi)| \lesssim_N M_1^N M_2^{-1} \quad \text{ and } \quad |\nabla_{\bm{e}}^N \inn{\gamma' \circ \theta_2(\xi)}{\xi}| \lesssim_N BM_1^N.
\end{equation*}
for all $N \in \N$ and all $\xi \in \Omega \subset \hat{\R}^3 \backslash \{0\}$.

The application with respect to $\theta_1^\pm: \widehat{\R}^3 \backslash \{0\} \to I_0$ satisfying
\begin{equation*}
    \inn{\gamma' \circ \theta_1^\pm (\xi)}{\xi}=0
\end{equation*}
is similar, with $g=\gamma'$ (we do not require to take an auxiliary $h$ in this case).
\end{example}




\section{Integration-by-parts}

For $a \in C^{\infty}_c(\R)$ supported in an interval $I \subset \R$ and $\phi \in C^{\infty}(I)$, define the oscillatory integral
\begin{equation*}
    \mathcal{I}[\phi, a] := \int_{\R} e^{i \phi(s)} a(s)\,\ud s.
\end{equation*}
The following lemma is a standard application of integration-by-parts.

\begin{lemma}[Non-stationary phase]\label{non-stationary lem} Let $R \geq 1$ and $\phi, a$ be as above. Suppose that for each $j \in \N_0$ there exist constants $C_j \geq 1$ such that the following conditions hold on the support of $a$:
\begin{enumerate}[i)]
    \item $|\phi'(s)| > 0$,
    \item $|\phi^{(j)}(s)| \leq C_j R^{-(j-1)}|\phi'(s)|^j\,\,$ for all $j \geq 2$,
    \item $|a^{(j)}(s)| \leq C_j R^{-j}|\phi'(s)|^j\,\,$ for all $j \geq 0$.
\end{enumerate}
Then for all $N \in \N_0$ there exists some constant $C(N)$ such that
\begin{equation*}
    |\mathcal{I}[\phi, a]| \leq C(N) \cdot |\supp a| \cdot R^{-N}.
\end{equation*}
Moreover, $C(N)$ depends on $C_1, \dots, C_N$ but is otherwise independent of $\phi$ and $a$ and, in particular, does not depend on $r$. 
\end{lemma}

\begin{proof} Taking $D := \phi'(s)^{-1} \partial_s$, repeated integration-by-parts implies that 
\begin{equation*}
    \mathcal{I}[\phi, a] = (-i)^{-N}\int_{\R} e^{i\phi(s)} (D^*)^Na(s)\,\ud s
\end{equation*}
where $D^*$ is the `adjoint' differential operator $D^* \colon a \mapsto -\partial_s \big[(\phi')^{-1} \cdot a \big]$. Thus, the proof boils down to establishing a pointwise estimate
\begin{equation*}
    |(D^*)^Na(s)| \leq C(N) \cdot R^{-N}
\end{equation*}
under the hypotheses of the lemma. 

It is in fact convenient to prove a more general inequality
\begin{equation}\label{non-stationary 1}
    |\partial_s^j(D^*)^Na(s)| \leq C(j, N)  \cdot R^{-N-j}\cdot |\phi'(s)|^j, \qquad \textrm{for all $j, N \in \N_0$},
\end{equation}
where the $C(j, N)$ again only the constants $C_k$ for $1 \leq k \leq N+j$. The inequality \eqref{non-stationary 1} is amenable to induction on the parameter $N$. Indeed, if $N = 0$, then \eqref{non-stationary 1} reduces to hypothesis iii), which establishes the base case. 

Assume the inequality \eqref{non-stationary 1} holds for some $N \geq 0$ and all $j$. By the Leibniz rule,
\begin{equation}\label{non-stationary 2}
  \partial_s^j(D^*)^{N+1}a(s) = \sum_{i=0}^{j+1} \binom{j+1}{i} \big[\partial_s^i (\phi')^{-1}\big](s) \cdot \big[\partial^{j+1-i} (D^*)^N a\big](s).  
\end{equation}
Using the induction hypothesis, one may immediately bound
\begin{equation}\label{non-stationary 3}
    \big|\big[\partial^{j+1-i} (D^*)^N a\big](s)\big| \leq C(j+1-i, N) \cdot R^{-N - 1 - j + i} \cdot |\phi'(s)|^{j + 1 - i}.
\end{equation}
On the other hand, an induction argument shows that there exists a polynomial $\wp \in \R[X_0, \dots, X_i]$, with coefficients depending only on $i$, with the following properties:
\begin{enumerate}[a)]
    \item $\wp$ is a linear combination of monomials $X_0^{\alpha_0}\cdots X_i^{\alpha_i}$ for multi-indices $(\alpha_0, \dots, \alpha_i)$ satisfying
    \begin{equation*}
        0 \cdot \alpha_0 + 1 \cdot \alpha_1 + \cdots + i \cdot \alpha_i = \alpha_0 + \alpha_1 + \cdots +  \alpha_i =  i.
    \end{equation*}
    \item The identity
    \begin{equation*}
 \big[\partial_s^K (\phi')^{-1}\big](s) = \frac{\wp\big(\phi'(s), \dots, \phi^{(i+1)}(s)\big)}{\phi'(s)^{i+1}} \qquad \textrm{holds for all $s \in I$.}
\end{equation*}
\end{enumerate}
If $(\alpha_0, \dots, \alpha_i)$ is a monomial satisfying a), then hypothesis ii) of the lemma implies that
\begin{equation*}
    \prod_{k=0}^i |\phi^{(k+1)}(s)|^{\alpha_k} \lesssim R^{-i} \cdot |\phi'(s)|^{2i},
\end{equation*}
where the implied constant is here allowed to depend on the $C_k$ for $1 \leq k \leq i+1$. Consequently, from the formula in b) above one deduces that 
\begin{equation}\label{non-stationary 4}
    |\big[\partial_s^i (\phi')^{-1}\big](s)|  \lesssim R^{-i} \cdot |\phi'(s)|^{i-1}.
\end{equation}
Substituting the bounds \eqref{non-stationary 3} and \eqref{non-stationary 4} into \eqref{non-stationary 2}, the induction now closes provided $C(j,N)$ is appropriately defined. 
\end{proof}




\bibliography{Reference}
\bibliographystyle{amsplain}

\end{document}